\documentclass[reqno,english]{amsart}
\usepackage{macros}
\usepackage{caption}

\newcommand{\Var}{\mathrm{Var}}
\newcommand{\arctanh}{\mathrm{arctanh}}
\newcommand{\sol}{\mathbf{X}}
\newcommand{\calP}{\mathcal{P}}
\newcommand{\calZ}{\mathcal{Z}}
\newcommand{\sign}{\mathrm{sign}}

\title{
Satisfiability and Freezing in Random integer programs}
\author{Dylan J. Altschuler, Ruben Ascoli, and Will Perkins}
\date{\today}

\address{Department of Mathematics, The University of Texas at Austin}
\email{dylan.altschuler@austin.utexas.edu}

\address{School of Mathematics, Georgia Institute of Technology}
\email{rascoli3@gatech.edu}

\address{School of Computer Science, Georgia Institute of Technology}
\email{math@willperkins.org}

\begin{document}

\begin{abstract}
    We study a model of a random centrally symmetric polytope and ask about the existence of points in the polytope with non-zero, integer coordinates.  Our main results are the location of a sharp threshold for the existence of these integer points, a characterization of the typical empirical distribution of their coordinates, and showing that near the feasibility threshold these integer points are almost always totally isolated --- at linear Hamming distance to the nearest other integer point in the polytope.  

    The behavior of these random integer programs is related to that of the symmetric binary perceptron (the case when coordinates of points are $\pm 1$-valued).  The lack of symmetry in the set of potential solutions leads to new questions and phenomena, such as the distribution of the coordinate sizes in typical solutions, and requires the development of new information-theoretic techniques to deal with higher-dimensional optimization problems arising in the method of moments. 
\end{abstract}

\maketitle

\section{Introduction}

\label{secIntro}

How should worst-case algorithmic hardness be reconciled with practical and empirical evidence of tractability in real-world instances? There are several different philosophies for moving beyond worst-case analysis, each opening up different areas of theoretical computer science. For example, one can look for approximation guarantees or algorithms that succeed on special classes of structured instances. Another particularly fruitful approach is to look at the tractability of ``typical" instances – those drawn from a specified probability distribution.

Random $k$-SAT has become the prototypical model in this last direction, garnering intense research attention after the empirical discovery of several surprising phase transitions ~\cite{selman1996generating}. The probability that a random instance of $k$-SAT with a given constraint density is satisfiable appeared to drop rapidly from near $1$ to near $0$ as the density increased. Moreover, the algorithmic cost of solving these random instances seemed to spike near the same density.  These observations led  researchers in computer science, mathematics, and physics to try to understand how \textit{algorithmic} aspects of random $k$-SAT, such as whether typical instances are computationally hard, are related to \textit{probabilistic} behaviour, such as the existence of a sharp threshold for satisfiability or the geometry of the solution space for typical instances. 

The past two decades have clarified that algorithmic and probabilistic aspects of random $k$-SAT are closely linked, albeit in subtle ways \cite{achlioptas2008algorithmic,krzakala2007gibbs,mezard2009information,ding2022proof}. Moreover, the theory developed for random $k$-SAT turns out to be surprisingly general as highly similar phenomena have been discovered in a huge variety of other random optimization problems. We now bring this perspective to bear on a simple model of random integer programming.

Our model takes the random, centrally symmetric polytope obtained by intersecting origin-symmetric slabs whose normal vectors are independent Gaussians. This provides the feasible region of our random symmetric program. We then study integer points with no zero coordinates, which we call
``punctured lattice points,'' within the polytope. Given the dimension, the slab width, and the number of constraints, how many solutions to this random integer
program typically exist, and what is their structure?

We draw inspiration from the closely related symmetric binary perceptron (SBP) model, which studies the same polytope but asks for solutions from the discrete hypercube $\{ \pm 1 \}^n$. The SBP exhibits several striking phenomena, including a sharp threshold for satisfiability, a constant-size scaling window, and so-called \textit{frozen-1-RSB} behavior in which typical solutions are isolated and macroscopically far from any other solutions~\cite{APZ,PX24,ALS22,dja-ejp,ss}.  Curiously, despite this isolation of typical solutions, efficient algorithms can find solutions in atypical clusters comprising an exponentially small fraction of the solution space~\cite{abbe2022binary}. Analogously, our motivations for studying the current model of random integer programming are:
\begin{itemize}
    \item  Understanding the geometry of integer points in typical polytopes.
    \item Understanding how the frozen $1$-RSB phenomenon generalizes to broader settings.  
    \item Understanding the power, generality, and limitations of the fine second-moment-based analysis that has proven successful in analyzing the SBP.
\end{itemize}

\subsection{The model and main results}
\label{secmodelresults}

We first formalize our model of a random polytope.  In what follows, $n$ refers to the dimension of the problem, and we will be interested in asymptotics as $n$ diverges. The number of constraints will be $m = \Theta(n)$, and $\kappa >0$ will be a fixed parameter governing the shape of the polytope.   We say an event $A$ occurs \emph{with high probability} (whp) if  $\PP{A} \to 1$ as $n \to \infty$.

\begin{definition}[Model of random polytope]
    Let $\calP_{n,m,\kappa}$ be a random $n$-dimensional polytope obtained by intersecting $m$ random, origin-symmetric slabs of margin $\kappa$. That is, for a sequence of iid Gaussian vectors $A_i \sim N(0,n^{-1}I_n)$,  
    \[
        \calP_{n,m,\kappa} = \bigcap_{i=1}^m\cb{x \in \RR^n\,:\, |\la x,\, A_i \ra | \le \kappa}\,.
    \]
\end{definition}

Equivalently, letting $A$ be the $m\times n$ matrix with row $i$ given by $A_i$, 
\[
    \calP_{n,m,\kappa} = \cb{x \in \RR^n\,:\, \|Ax\|_\infty \le \kappa}\,.
\]
Next, we define our main object of study: the solution set.

\begin{definition}[Punctured lattice points and solutions]
    Define
    \[
        \Lambda_n := (\ZZ\setminus\{0\})^n.
    \]
    The random integer solution space is denoted by
     \[
        \sol_{n,m,\kappa} :=\calP_{n,m,\kappa} \cap \Lambda_n\,.
    \]
\end{definition}

\begin{figure}[t!]
    \centering
    \includegraphics[width=0.5\linewidth]{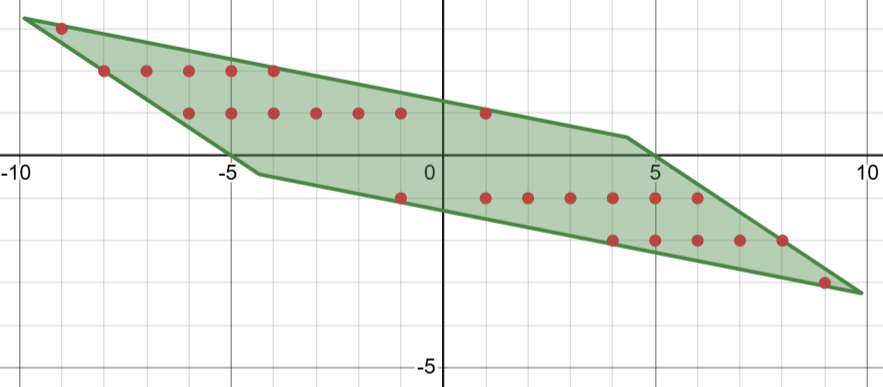}
    \caption{An example of the random polytope in two dimensions, with two pairs of parallel hyperplanes.  The ``punctured lattice points'' in red comprise the solution set $\sol_{n,m,\kappa}$. }
    \label{fig:2d}
\end{figure}

Our main questions now can be stated precisely. In the physics terminology from the study of perceptrons as models of neural networks, the \textit{capacity} is the random variable given by the largest $m$ for which $\sol_{n,m,\kappa}$ is non-empty. 

\begin{question}
    For fixed $\kappa >0$, what is the typical distribution of the capacity?  In particular, 
    \begin{itemize}
        \item What is the satisfiability threshold, i.e., the largest value of $\alpha$ such that $ \sol_{n,\alpha n,\kappa}$ is non-empty whp? 
        \item How sharp is the satisfiability transition? More pecisely, we ask for the width of the scaling window: how much must $m$ increase for the probability of satisfiability to decrease from $1-\epsilon$ to $\epsilon$? On what scale does the capacity typically fluctuate?
    \end{itemize}
\end{question}
\begin{question}
    When $ \sol_{n,m,\kappa}$ is   non-empty, what is its typical size  and structure?  In particular, what is the typical coordinate profile of an element of $ \sol_{n,m,\kappa}$ and what is the geometry of $ \sol_{n,m,\kappa}$ as a set?
\end{question}

We give complete answers to all of these questions for a wide range of parameters $\kappa$ and $\alpha$.  A few definitions are needed to state the results.

\begin{definition}[Scalar discrete Gaussian]\label{def:DG}
    For $r > 1$, let $\mu_{r}$ denote the ``punctured'' discrete Gaussian supported on $\ZZ \setminus \{0\}$ having mean zero and variance $r^2$. For a uniquely determined $\sigma := \sigma(r)$ and $Z:=Z(r)$, its probability mass function is
    \begin{equation}
        \mu_{r}(x) := \frac{1}{Z}\exp\left(-\frac{x^2}{2\sigma^2}\right),
        \qquad x \in \ZZ \setminus \{0\}.
    \end{equation}
    Equivalently, $\mu_r$ is the unique maximum-entropy distribution on
    $\ZZ\setminus\{0\}$ with mean zero and variance $r^2$. We also set
    $\mu_1$ as the uniform distribution on $\{-1,1\}$. 
\end{definition}

The punctured discrete Gaussian plays a distinguished role as the approximate distribution of a single coordinate from a uniformly random punctured lattice point of Euclidean norm $r \sqrt{n}$. We will also show in Section~\ref{sec:1profile} that the Shannon entropy  $H(\mu_r)$ measures the exponential growth rate of punctured lattice points with norm $r \sqrt{n}$ for fixed $r \ge 1$ as $n$ diverges. Letting $W$ be a scalar standard Gaussian, set
\begin{align*}
    p(a) := \PP{|W| \le a}.
\end{align*}
\begin{definition}[First-moment prediction for  satisfiability]   
The \textit{annealed free energy} function is given by
\begin{equation}\label{eq:F1-def}
    F_1(\kappa,\alpha,r):=H(\mu_r)+\alpha\log p(\kappa/r).
\end{equation}
The \textit{first-moment prediction for the satisfiability threshold} is
\begin{equation}\label{eq:capacity definition}
        \alpha_c(\kappa)
        :=\sup_{r\ge1}\frac{H(\mu_r)}{-\log p(\kappa/r)}.
    \end{equation}
Given $\kappa > 0$ and $\alpha > 1$, the \emph{maximizing radius} is $$r_*(\kappa,\alpha)=\arg\max_{r\ge1}F_1(\kappa,\alpha,r).$$ 
\end{definition}

The existence and uniqueness of $r_*(\kappa, \alpha)$ are proven in \cref{sec:1profile}.
The function $F_1(\kappa,\alpha,r)$ measures the exponential growth rate of the expected number of solutions with norm $r\sqrt n$. By Markov's inequality, $F_1$ provides an upper bound on the true number of solutions; we use the second moment method to provide a matching lower bound. 

Next, we construct the central object in the second moment method, $F_2$, which describes the exponential rate of growth of the expected number of \textit{pairs} of solutions with prescribed Euclidean norms and angle. Pairs of punctured lattice vectors with prescribed norms and angle are described by the following planar analogue of the scalar
discrete Gaussian in \cref{def:DG}. 
\begin{definition}[Planar discrete Gaussian]\label{def:DG 2}
    Let $r,\hat r>1$ and $\theta\in(-1,1)$. Set 
    \[
        \Sigma_{r,\hat{r},\theta} := \begin{pmatrix}
            r^2&\theta r\hat r\\
            \theta r\hat r&\hat r^2
        \end{pmatrix}.
    \]
    Assuming that there exists a valid probability distribution on $(\ZZ \setminus \{0\})^2$ with mean zero and covariance $\Sigma_{r,\hat{r},\theta}$, we define the planar discrete Gaussian, denoted by $\mu_{r,\hat r,\theta}^{(2)}$, as the maximum-entropy such distribution.
\end{definition}
Convexity of the feasible set of such distributions and the strict concavity of entropy imply that the maximum entropy distribution is unique and hence the above definition is well-posed; see \cite[Theorem~2.5]{dg-theta}. As discussed shortly in \cref{prop:dg2}, the density of $\mu^{(2)}$ is also given by an exponential quadratic under mild assumptions (which are necessary because we exclude vectors with zero coordinates).

For $u, v>0$ and $\theta\in[-1,1]$, define \[
    p^{(2)}(u,v,\theta)
    :=\PP{|W|\le u,\ |W_\theta|\le v},
\]
where $W$ and $W_\theta$ are centered jointly Gaussian random variables with
covariance $\theta$ and unit variance. Finally, define the \emph{second moment profile function} by
\begin{equation}\label{eq:F2-definitionintro}
    F_2(\kappa,\alpha,r,\theta)
    :=H(\mu^{(2)}_{r,r,\theta})
      +\alpha\log p^{(2)}(\kappa/r,\kappa/ r,\theta).
\end{equation}
Note that both terms are even in $\theta$. Also observe that the first moment profile is related to the second moment profile at $\theta = 0$ by
\begin{equation}\label{eq:F2-at-zero}
    F_2(\kappa,\alpha,r,0)
    =2F_1(\kappa,\alpha,r) \,.
\end{equation}

We are able to obtain a complete description of our model assuming the following condition on the second moment profile function:

\begin{hypothesis}[Second-moment condition]\label{hyp:maxima}
Fix $\kappa>0$ and $\alpha\in(1,\alpha_c(\kappa)]$, and let
$r_*=r_*(\kappa,\alpha)$. Then
\begin{enumerate}
    \item
    \begin{equation}\label{eq:F2-maxima}
        F_2(\kappa,\alpha,r_*,\theta)
        <F_2(\kappa,\alpha,r_*,0)
        \qquad
        \text{for every }\theta\in(-1,1)\setminus\{0\};
    \end{equation}
    \item
    \[
        \partial_{\theta\theta}^2 F_2(\kappa,\alpha,r_*,0) < 0
    \]
\end{enumerate}
\end{hypothesis}

These conditions are user-friendly in the sense that they concern a single-variable analytic function on a bounded domain.
We analytically verify \cref{hyp:maxima} in the following parameter range. 

\begin{proposition}[Verification of the second-moment condition]
\label{thm:unique-max}
If $\kappa\ge10$, then
$\alpha_c(\kappa)>1+\kappa^{-2}$ and \cref{hyp:maxima} holds whenever
\begin{equation}\label{eq:profile-main-range}
    1+\frac1{\kappa^2}\le\alpha\le\alpha_c(\kappa).
\end{equation}
\end{proposition}

Our first main result determines the satisfiability threshold and the width of the scaling window.
\begin{theorem}[Sharp threshold and critical window]\label{thm:main}
Fix $\kappa>0$ such that $\alpha_c(\kappa)>1$ and
\cref{hyp:maxima} holds for $\alpha=\alpha_c(\kappa)$. 
There exists a function $f_\kappa: \mathbb R \to [0,1]$ with $\lim_{t\to-\infty}f_\kappa(t)=0$ and
$\lim_{t\to+\infty}f_\kappa(t)=1$ so that the following holds.  For each fixed $T > 0$, uniformly for $t \in  [-T,T]$, when 
$m=\alpha_cn+t\in\mathbb N$,
\begin{equation}\label{eq:main-intro-scaling-limit}
    \PP { \sol_{n,m,\kappa} = \emptyset } = f_\kappa(t)+o_n(1)
\end{equation}
In particular, $\alpha_c$ is the sharp satisfiability threshold, and the scaling window is of constant order.
\end{theorem} 
The function $f_\kappa$ is given by an explicit formula deferred to \cref{thm:main2}.

Our next main result describes the structure of typical solutions. Say that a solution $x\in \sol_{n,m,\kappa}$ is \emph{$\delta$-frozen} if for all $y\in \sol_{n,m,\kappa}\setminus\{x\}$, we have $\left|\{i:x_i\ne y_i\}\right|>\delta n$. Moreover,
for $x = (x_1,\dots,x_n) \in\Lambda_n$, let
\[
\nu_x:=\frac1n\sum_{i=1}^n\delta_{x_i}
\]
be the empirical distribution of its coordinates. Finally, we say ``almost all solutions'' (or pairs of solutions) satisfy a property if the fraction of solutions with the property converges in probability to one as $n\to \infty$. 
\begin{theorem}[Structure of solutions and solution space]
\label{thmStructure}
Fix $\kappa>0$ and $\alpha\in(1,\alpha_c(\kappa))$ such that
\cref{hyp:maxima} holds. Let
$m=\lfloor\alpha n\rfloor$, $\calZ=|\sol_{n,m,\kappa}|$, and $r_*=r_*(\kappa,\alpha)$. The following hold with high probability:
\begin{enumerate}
    \item\label{item:freezing} (Strong freezing.)
    For some $\delta = \delta(\kappa,\alpha)>0$, almost all solutions $x \in \sol_{n,m,\kappa}$ are $\delta$-frozen.

    \item\label{item:typical orthogonal} (Typical pairs of solutions are almost orthogonal)
    Almost all pairs of solutions $(x,y)\in\sol_{n,m,\kappa}^2$ satisfy $$|\la x,y\ra|\le \sqrt n\log n \,.$$
    
    \item \label{item:typical empirical} (Empirical distribution of a typical solution)
    For almost all  solutions $ x \in \sol_{n,m,\kappa} $, $$ \|\nu_x-\mu_{r_*}\|_{\mathrm{TV}} =o(1)\,.$$

    \item \label{item:lognormal} The random variable $\calZ/ \E \calZ$ has a log-normal limit (explicitly given in \cref{lem:lognormal}).
\end{enumerate}
\end{theorem}

\subsection{Discussion and future directions}

The first-moment prediction $\alpha_c(\kappa)$ is an upper bound on the satisfiability threshold: above it, the expected number of solutions tends to zero. For $1<\alpha<\alpha_c(\kappa)$, we restrict the solution count to vectors with length close to $r_*\sqrt n$ and group pairs by their normalized inner product $\theta$. Laplace's method reduces the exponential rate of the second moment to maximizing $F_2(\theta) := F_2(\kappa,\alpha,r_*,\theta)$. Since $F_2(0)=2F_1(\kappa,\alpha,r_*)$, if there is some $\theta$ with $F_2(\theta) > F_2(0)$ then there is an exponential mismatch between the second moment and the squared first moment. \cref{hyp:maxima} asks $F_2$ to both have a strict maximum on $(-1,1)$ at $\theta = 0$ and also to have negative curvature at zero. Together, these enable sharp asymptotics of the second moment via Laplace's method and yield a coarse threshold. (A number of additional techniques, such as a dense analogue of the small subgraph conditioning, are then needed to obtain a sharp threshold and the log-normal limit. We give a detailed account in \cref{sec:ov}.)

\begin{figure}
    \centering
    \begin{minipage}{0.5\textwidth} 
    \centering
    \includegraphics[width=1.05\linewidth]{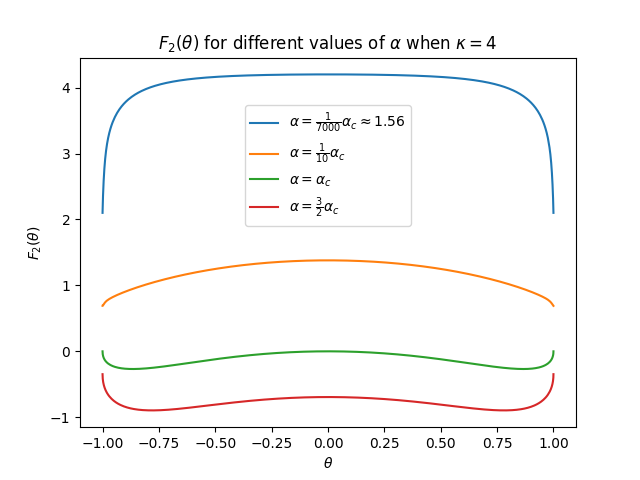}
    \caption*{\textbf{(a.)} The function $F_2$ for $\kappa=4$ and various values of $\alpha$.}%
    \end{minipage}%
    \begin{minipage}{0.5\textwidth}
    \centering
    \includegraphics[width=1.05\linewidth]{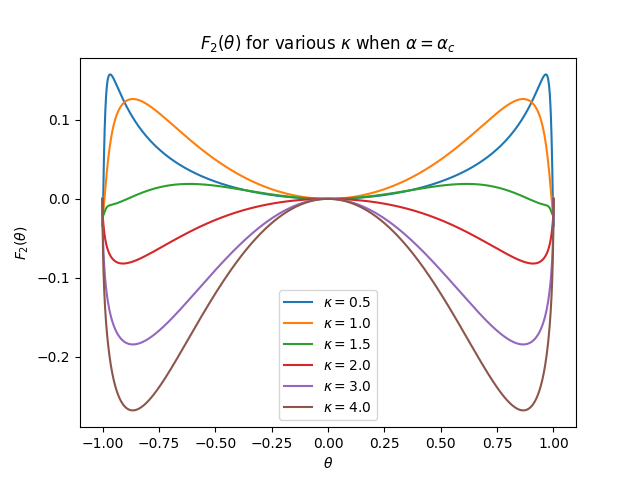}
    \caption*{\textbf{(b.)} The function $F_2$ for various values of $\kappa$ with $\alpha = \alpha_c(\kappa)$.}
    \end{minipage}
    \caption{The second moment profile function for some example choices of parameters.}
    \label{fig:F2}
\end{figure}

However, \cref{hyp:maxima} actually fails for some values of $\kappa < 2$ when $\alpha = \alpha_c(\kappa)$; see Figure \ref{fig:F2}. Hence, while the assumption $\kappa\geq 10$ in \cref{thm:unique-max} is certainly not best possible, it is necessary to bound $\kappa$ from below. A natural next direction is to study the count and geometry of the solution space for the regime of $(\kappa, \alpha)$ where \cref{hyp:maxima} fails.

\begin{question}\label{q}
    Fix $\kappa>0$.
    \begin{enumerate}
        \item Does strong freezing hold for all $\alpha$ below the satisfiability threshold?
        \item Is $\alpha_c(\kappa)$ the satisfiability threshold for all $\kappa$? 
    \end{enumerate}
\end{question}
Our second moment analysis does not answer these questions
when \cref{hyp:maxima} fails; however, the failure of \cref{hyp:maxima} does not necessarily
imply a negative answer. We are especially interested in the values of $\kappa$ for which \cref{hyp:maxima} fails in the entire range $1 < \alpha \le \alpha_c(\kappa)$. 

We conjecture a positive answer to the first question and a negative answer to the second. For the latter, there are two mechanisms that cause the first moment prediction $\alpha_c(\kappa)$ to sometimes be an overestimate. The first mechanism is that, for $\kappa$ small and $\alpha$ close to one, large deviations in the lower singular spectral edge of $A$, the $\alpha n \times n$ matrix obtained by stacking the rows $(A_i)_{i \in [\alpha n]}$, can cause a disproportionate increase in the expected number of solutions. 

\begin{remark}
    As a concrete example, fix $(\kappa,\alpha) = (.5,1.01)$. The smallest eigenvalue of $A^TA$ is typically approximately $2.5\times 10^{-5}$. Consider the rare event $\mathcal{E} := \{\lambda_{\mathrm{min}}(A^TA) \le 10^{-5}\}$. It can be shown $\E{\calZ\, \ind(\mathcal{E}^c)}/\E{\calZ} < e^{-0.0006n}$.
\end{remark}

This can potentially be resolved by conditioning on the high probability event that $A$ has a typical singular spectrum and then adjusting the definition of $\alpha_c(\kappa)$ accordingly. However, we conjecture that the first moment prediction, even after conditioning on the spectrum of $A$, is still an overestimate for some (narrower) regime of $\kappa$ and $\alpha$. This is due to the second mechanism: we suspect the existence of a regime of $(\kappa,\alpha)$ in which the solution space is nonempty and exhibits ``full replica symmetry breaking'' (FRSB), a complicated form of hierarchical clustering also conjectured to occur in the $U$-perceptron \cite{APZ} and negative spherical perceptron (see \cite{negsphere} and references therein for the formal definition). More boldly, we conjecture that after the appropriate spectral conditioning, the model exhibits only frozen 1-RSB or FRSB, with no intermediate replica symmetry breaking regimes. 

\begin{question}
    Does there exist $(\kappa,\alpha)$ with $\alpha$ below the satisfiability threshold such that the solution space exhibits FRSB, in the sense of the overlap distribution\footnote{The empirical distribution proportional to $\sum_{(x,y) \in \sol^2}\delta(|\la x/\|x\|,y/\|y\|\ra|)$.} converging weakly to a probability measure with infinite support? If so, is it the case that for all $\kappa > 0$ and all $\alpha > 1$ below the satisfiability threshold that the solution space only either exhibits frozen 1-RSB or else FRSB?
\end{question}

\subsection{Related work}

The study of solution space geometry of random constraint satisfaction problems has a rich history in physics, computer science and mathematics, including the seminal picture put forth by Krzaka{\l}a, Montanari,  Ricci-Tersenghi,  Semerjian, and Zdeborov{\'a}~\cite{krzakala2007gibbs}.  The mathematical understanding of the solution space is closely linked to second moment methods, as shown by Achlioptas and Coja-Oghlan~\cite{achlioptas2008algorithmic}.

The value of symmetry in aiding second moment methods was first exploited by Achlioptas and Moore in the context of the random $k$-SAT problem~\cite{achlioptas2006random} and explains the relatively rich literature of rigorous results on the symmetric binary perceptron in comparison to, say, the (asymmetric) binary perceptron.

Some of the structural properties studied here were first predicted in the physics literature on the binary perceptron.  This is a toy model of a neural network considered in the 1980's by Gardner and Derrida~\cite{gardner1988optimal}. Precise predictions for the satisfiability threshold were put forth by Krauth and Mezard~\cite{krauth1989storage} along with a form of the frozen $1$-RSB phenomenon which was then conjectured in a strong form by Huang, Wong, and Kabashima~\cite{huang2013entropy}.  For this model, the location of the satisfiability threshold was only recently established by Ding and Sun~\cite{ding2019capacity} (lower bound) and Huang~\cite{huang2024capacity} (upper bound); the structural predictions remain open. 

The SBP was introduced in~\cite{APZ}, though it also has a natural interpretation as a random model of the combinatorial discrepancy problem~\cite{matouvsek1999geometric}.  The symmetry of the model makes direct first and second moment calculations effective. Combined with
small subgraph conditioning and some concentration arguments, the moment method yields the satisfiability threshold,
frozen $1$-RSB behavior, and the width of the scaling window for the SBP
\cite{PX24,ALS22,dja-ejp,abbe2022binary,ss}. Links between the structure of the solution space and algorithms were explored in~\cite{gamarnik2022algorithms,barbier2024atypical,li2025discrepancy,el2025hardness,gong2026stable,vafa2025symmetric}.  We leave as a future direction understanding algorithmic aspects of finding solutions under the current model of random integer programs. 

Finally, numerous similar models of random integer programs have been studied under the name ``chance-constrained integer feasibility programs''~\cite{chandrasekaran2014integer,charnes1958cost,miller1965chance,prekopa1970probabilistic}. For example, our model is roughly the same studied by Chandrasekaran and Vempala in~\cite{chandrasekaran2014integer}, except that our hyperplanes come in symmetric pairs. 

\subsection*{Organization}
In \cref{sec:ov}, we give an extended technical overview of the article. In \cref{sec:1profile}, we prove the lattice-counting estimates and first-moment
truncation. In \cref{sec:2profile} (and the appendix), we verify the second-moment condition \cref{hyp:maxima} for a large regime of $(\kappa,\alpha)$. \cref{sec:cycle counts,sec:second moment method} assume $(\kappa,\alpha)$ satisfy \cref{hyp:maxima} and establish the
cycle-count correction, log-normal limit, and freezing.
Finally, the remaining claims in the main theorems are completed in \cref{sec:main theorem proofs}.

\subsection*{Statement on AI use}

This project spanned several years and was almost entirely completed in the ``pre-AI era.'' LLMs were involved in the final writing stages. During proof reading, GPT 5.6 suggested and drafted improvements for the range of $\alpha$ considered in \cref{thm:unique-max}, which we verified and adopted.   

\section{Notation and proof overview}
\label{sec:ov}

\subsection{Notation}

We use the standard asymptotic notation $o$, $O$, $\omega$, $\Omega$, and
$\Theta$, with asymptotics taken as $n\to\infty$ unless otherwise stated.
The notation $\widetilde O$ suppresses factors of order $n^{o(1)}$. We work
throughout in the overdetermined regime $\alpha=m/n>1$. Indeed, when $m<n$
the polytope contains infinitely many integer points almost surely, whereas
for $m\ge n$ it contains only finitely many.

All logarithms are natural, and $0\log0:=0$. For a probability distribution
$\nu$ on a countable set, its Shannon entropy is
\[
    H(\nu):=-\sum_x\nu(x)\log\nu(x).
\]
We also write $H(X)$ for the entropy of the law of a random variable $X$.
For distributions $\nu$ and $\mu$ on the same countable set, their relative
entropy is
\[
    D(\nu\|\mu):=\sum_x\nu(x)\log\frac{\nu(x)}{\mu(x)}.
\]

Next, it will be convenient to work with the following representation of the planar discrete Gaussian.

\begin{proposition}\label{prop:dg2}
    Let $r,\hat{r} > 1$ and $\theta \in (-1,1)$ satisfy
    \begin{equation}\label{eq:dg2-condition}
        r\hat{r}|\theta| < 1 + \sqrt{(r^2-1)(\hat{r}^2-1)}\,.
    \end{equation}
    Then $\mu_{r,\hat r,\theta}^{(2)}$ exists and there is also a unique tuple of parameters $(a,b,c)$ and a normalizing constant $\Psi$ with
    \[
        \mu_{r,\hat r,\theta}^{(2)}(x,y) =\exp\{ax^2+bxy+cy^2-\Psi(a,b,c)\}.
    \]
\end{proposition}

The proof is elementary and deferred to the appendix. Note that this condition is satisfied if $r = \hat{r}$, for every $r \ge 1$ and $\theta \in (-1,1)$. Fortunately, as discussed further in \cref{sec:2profile}, we only consider regimes of $(r,\hat{r},\theta)$ in this article for which this condition is satisfied.

\begin{remark}[Parameterization]
    Our parameterization of the discrete Gaussian distribution $\mu_r$ is through its variance $r^2$. In the    computer science literature, especially in lattice cryptography, the discrete Gaussian is often parameterized by $\sigma^2$, the parameter appearing in the exponential. Another parameterization frequently employed in this paper and arising from viewing the discrete Gaussian as an exponential family solving a maximum entropy problem, is to set $a := a(r) = -1/(2\sigma^2)$, $b = -\log Z$, and $\mu_r(\ell) = e^{a \ell^2 +  b}$. 
\end{remark}

The central objects are $\sol := \sol_{n,m,\kappa}$ and
$\calZ=\calZ(n,m,\kappa):=|\sol|$. For these variables as well as those defined below, we suppress dependence on $n,m,\kappa,\alpha$ whenever there is no ambiguity.
\begin{definition}[Solution spheres]
For $r\ge1$, define
\[
    \Lambda_n(r):=\cb{x\in\Lambda_n:\|x\|_2^2=nr^2},
    \qquad
    \sol(r) := \sol \cap\Lambda_n(r),
    \qquad
    \calZ_r:=|\sol(r)|,
\]
and let
\[
    N_n(r) := |\Lambda_n(r)|, \qquad \calR_n:=\cb{r\ge1:\Lambda_n(r)\ne\emptyset}.
\]
For $J\subset[1,\infty)$, set
\[
    \calZ_J:=\sum_{r\in J\cap\calR_n}\calZ_r
    =\left|\bigcup_{r\in J\cap\calR_n}
      \sol(r)\right|.
\]
We also define an analogous count of lattice pairs. For $r,\hat{r} \in \calR_n$ and $\theta \in [-1,1]$ with $\theta r\hat{r} n \in \ZZ$,
\[
    \Lambda_n(r,\hat{r},\theta):=\cb{(x,y)\in\Lambda_n(r) \times \Lambda_n(\hat{r}): \la x,y \ra = \theta r \hat{r} n},
    \qquad
    N_n(r,\hat{r},\theta):= |\Lambda_n(r,\hat{r},\theta)|\,.
\]
\end{definition}
We henceforth will suppress the subscript $n$ in the above variables, wherever possible.

For easy reference, \cref{table:notation} lists some notation with short descriptions. 
Whenever $\Lambda_n(r,\hat r,\theta)$ or
$\Lambda_n^{\typ}(r,\hat r,\theta)$ appears, we implicitly assume that
$r,\hat r\in\calR_n$ and $\theta r\hat r n\in\ZZ$. Similarly, a
moment-constrained distribution is considered only when its prescribed
moments can be realized on the stated support.

\begin{table}[htb!]
\centering
\small
\renewcommand{\arraystretch}{1.08}
\begin{tabular}{>{\centering\arraybackslash}p{4.1cm} p{10cm}}
\toprule
\textbf{Notation} & \textbf{Description} \\
\midrule
$n,m,\alpha=m/n,\kappa$
    & Dimension, number and density of constraints, and slab margin \\
$\calP_{n,m,\kappa}$, $\Lambda_n$, $\sol_{n,m,\kappa}$, $\calZ$
    & Random polytope, punctured lattice, solution space, solution count \\
$\mu_r$, $\mu_{r,\hat r,\theta}^{(2)}$
    & Scalar and planar ``punctured'' (no zero coordinates) discrete Gaussians with the prescribed
      second moments \\
$\calR_n$, $\Lambda_n(r)$, $\calZ_r$
    & Admissible radii and restrictions to squared Euclidean norm $nr^2$ \\
$\Lambda_n(r,\hat r,\theta)$, $N(r,\hat{r},\theta)$ & Set of and cardinality of lattice pairs with given length and angle. \\
$p(u)$, $p^{(2)}(u,v,\theta)$
    & One- and two-dimensional Gaussian constraint probabilities \\
    $F_1$, $F_2$ &
    First
      and second moment profiles \\
$\alpha_c(\kappa)$, $r_*(\kappa,\alpha)$; $r_c$, $p_c$
    & First-moment capacity and maximizing radius; critical capacity specializations, $r_*(\kappa,\alpha_c(\kappa))$ and $p(\kappa/r_*)$ \\
$I=I_n(\kappa,\alpha)$
    & Radii within $(\log\log n)/\sqrt n$ of $r_*$
    \\
$d(u)$, $B_{u,\alpha}$
    & $d(u)=1-\E{W^2\mid |W|\le u}$ and
      $B_{u,\alpha}=-\sqrt\alpha\,d(u)/2$ \\
    $\eta$, $m_0$ & Small constant and ``precritical'' time $m_0 := \alpha_c(\kappa)n - \eta \log n$. \\
$\lambda,T,L$
    & Cycle parameters:
      $T=\lfloor\lambda\log n\rfloor$, $L=\sqrt{\log n}$,
      $\lambda$ sufficiently small and independent of $\eta$ \\
\bottomrule
\end{tabular}
\caption{Notation used throughout to denote relevant parameters and  sets.}\label{table:notation}
\end{table}

\newpage

\subsection{Technical overview}

\subsubsection{The symmetric binary perceptron}

The SBP provides a partial template for our
strategy as well as an interesting contrast. The SBP considers the same random polytope $\calP_{n,m,\kappa}$ but restricts
solutions to $\Lambda_n(1)=\{-1,+1\}^n$, so the object of interest is
$\calZ_1$. For fixed $x\in\Lambda_n(1)$, $Ax$ is an $m$--dimensional
standard Gaussian. Consequently, 
\[
    \E{\calZ_1} = 2^n p(\kappa)^{\alpha n},\, \qquad \alpha_{\rm bin}(\kappa) = - \frac{\log 2}{\log p(\kappa)}\,.
\]
Moreover, the Gaussian process $(Ax)_{x\in\Lambda_n(1)}$ has covariance entirely determined by the inner products of its indices (since every vector $x \in \Lambda_n(1)$ has the same norm). Consequently, grouping pairs $(x,y)$ by their normalized inner product
reduces computing the exponential growth rate of $\E{\calZ_1^2}$ to maximizing a one-variable profile function:
\[
    \E{\calZ_1^2} = \sum_{\theta} N(1,1,\theta) \,p^{(2)}(\kappa,\kappa,\theta)^{\alpha n} =: \sum_{\theta} e^{nF_2(\theta)}\,,
\]
and so $\frac{1}{n} \log  \E{\calZ_1^2} = \max_\theta F_2 (\theta) + o(1)  $.
Note that $N(1,1,\theta)$ admits an exact identity in terms of binomial coefficients, and hence sharp asymptotics using Stirling's formula. For every $\kappa>0$ and $\alpha\le\alpha_{\rm bin}(\kappa)$, the profile $F_2^{\rm bin}$ is strictly maximized at zero on $(-1,1)$.
For fixed $0<\alpha<\alpha_{\rm bin}(\kappa)$, the endpoint contributions are also exponentially smaller, so the second
moment is dominated by nearly orthogonal pairs. A saddle-point computation then yields, for some constant $C>1$,

\[
    \E{\calZ_1^2}=(C+o(1))\,\E{\calZ_1}^2
\]
This provides a constant-order lower bound on the probability of the existence of solutions below the satisfiability threshold, but it does not yet establish a sharp threshold since $C$ is strictly larger than one \cite{APZ}.

The remaining fluctuations are encoded by the emergence of small correlations in the matrix $A$. The ``dense'' small subgraph conditioning method quantifies this by counting certain short weighted cycle counts and gives a log-normal limit for $\calZ_1/\E{\calZ_1}$ \cite{ALS22,ss}. Hence, $\alpha_{\rm bin}(\kappa)$ is a sharp threshold for satisfiability \cite{PX24,dja-ejp}. Moreover, this log-normal convergence implies concentration of $\log \calZ_1$, which can be combined with analysis of  $\lim_{\theta \to \pm 1}F_2'(\theta)$ to establish strong freezing of the solution space.

The above arguments extend to within a logarithmic number of constraints from criticality; beyond this point, $\calZ_1$ is typically sub-polynomial and certain concentration arguments break down. Nonetheless, for $\eta$ some small positive universal constant and $m_0 = \lfloor \alpha_{\rm bin}(\kappa)n - \eta \log n\rfloor$, a different argument suffices. For $m \ge m_0$, it holds with high probability that \textit{all} non-antipodal pairs of solutions are nearly orthogonal, and this property persists among the surviving solutions as further constraints are added. The Gaussian projections governing these constraints can then be compared with independent Gaussians, yielding the critical window and capacity distribution \cite{ss}.

\subsubsection{Integer points in random polytopes}

Our analysis of  integer points in random polytopes  follows a similar outline, but there are substantial technical challenges to overcome as well as genuinely different behaviors compared to the SBP.
For example, solutions can have different Euclidean norms and coordinate
profiles, so even computing $\E{\calZ}$ is nontrivial. In particular, computing both first and second moments requires solving higher-dimensional optimization problems. Even more fundamentally, for some
choices of $(\kappa,\alpha)$ there is an exponential mismatch between the second moment and squared first moment of $\calZ$.

Our first main contribution is the isolation of a user-friendly condition, formalized as \cref{hyp:maxima}, which is necessary and sufficient for the first and second moments to match at the exponential scale. The condition concerns the maximum of an analytic function on a one-dimensional bounded interval. We rigorously establish in \Cref{sec:2profile} that \cref{hyp:maxima} holds for $(\kappa,\alpha)$ satisfying
\[
    \kappa\ge10,
    \qquad
    1+\kappa^{-2}\le\alpha\le\alpha_c(\kappa).
\]

Our second main contribution is an essentially complete description of the model for all $(\kappa,\alpha)$ satisfying \cref{hyp:maxima}. We obtain sharp estimates on the solution count and solution-space geometry at subcritical parameters $(\kappa,\alpha)$, as well as the satisfiability threshold, the critical window, and the capacity distribution when the hypothesis holds at $\alpha=\alpha_c(\kappa)$.

\subsubsection{Gaussian process view and shell entropies}
Fix $(\kappa,\alpha)$. The collection $(Ax)_{x\in\Lambda_n}$ is a centered $m$--dimensional
Gaussian process. Hence, the probability that a pair of vectors $(x,y) \in \Lambda_n(r) \times \Lambda_n(\hat{r})$ are both solutions is just a function of $r$, $\hat{r}$, and $\theta = \la x,y\ra/(\|x\|_2\|y\|_2)$. In particular, 
\begin{align}
    \E{\calZ^2} 
    &=: \sum_{(r,\hat r, \theta) \in \calR_n \times \calR_n  \times [-1,1]} N(r,\hat r,\theta)\,
      p^{(2)}(\kappa/r,\kappa/\hat r,\theta)^m\,.\label{eq:ov-2mm}
\end{align}
The task of estimating $N_n(r,\hat{r},\theta)$ was trivial in the binary setting, corresponding to $r = \hat{r} = 1$. In our setting, even the simpler task of obtaining tight asymptotics for $N_n(r)$ is already non-trivial.

Our information theoretic approach for approximating $N_n(r)$ and $N_n(r,\hat{r},\theta)$ is particularly transparent for the following example, where we consider solutions whose coordinates are bounded in absolute value by a constant. Let
$\mathcal A_k=\{-k,\ldots,-1,1,\ldots,k\}$ and let
$U_1,\ldots,U_n$ be uniform on $\mathcal A_k$. Then
\[
    \left|\cb{x\in\mathcal A_k^n:\|x\|_2^2=nr^2}\right|
    =(2k)^n\PP{\sum_{i=1}^nU_i^2=nr^2}.
\]
Sanov's theorem identifies the
exponential rate of this large-deviation probability as
\[
    \max\left\{H(\nu):
       \nu\text{ is supported on }\mathcal A_k,\ 
       \sum_xx^2\nu(x)=r^2\right\}.
\]
A Lagrange multiplier computation shows that the unique maximizer has mass proportional to $e^{ax^2}$, i.e., is a discrete Gaussian supported on $\mathcal{A}_k$. Letting $k$ diverge, this suggests $n^{-1}\log N_n(r)$ is given by the entropy of
the punctured discrete Gaussian. In fact, no limiting argument is needed. The maximum-entropy counting
identity of \cite{mckinley2023maximum} gives the exact formula
\[
    |\Lambda_n(r)|
    =e^{nH(\mu_r)}~
      \PP[X_i\stackrel{\mathrm{iid}}\sim\mu_r]
      {\sum_{i=1}^nX_i^2=nr^2}.
\]
The probability term can be evaluated using a classical lattice local central limit theorem, and provides an explicit polynomial correction to the exponential rate suggested by Sanov's theorem. A formula for the count of pairs is derived similarly:
\[
    N_n(r,\hat r,\theta)
    =e^{nH(\mu_{r,\hat r,\theta}^{(2)})}
      \times(\text{explicit local central limit correction}).
\]
These counting formulas are developed in \cref{prop:1d-lattice-count,prop:2d-lattice-count} and also yield as corollaries that a typical vector or pair of vectors with prescribed Euclidean data has empirical coordinate distribution close to the corresponding discrete Gaussian.

\begin{remark}
    A reasonable, but absolutely wrong, low-dimensional intuition is that $N_n(r)$ scales proportionally to the Euclidean volume of the $n$-ball of radius $r\sqrt{n}$ \cite{mazo1990lattice}. The above computations establish that this is off by an exponential factor. The intuition is that the discrepancy between the volume and number of integer points in a body is largely due to boundary effects, whereas almost all of the volume of a high-dimensional sphere lies near its boundary. However, as $r$ diverges, the entropy of $\mu_r$ approaches $\frac12\log(2\pi e r^2)$, the differential entropy of $N(0,r^2)$. For fixed $r$, this also equals the limiting normalized log-volume of the ball, $n^{-1}\log \mathrm{Vol} \,B_n(0,r\sqrt{n})$. Hence, as $r \to 1$ and $r \to \infty$, the integer geometry of our model approaches the discrete cube and continuum respectively, and is non-trivial everywhere in between. 
\end{remark}

\subsubsection{First-moment truncation}

Since a vector in $\Lambda_n(r)$ satisfies one constraint with probability
$p(\kappa/r)$, we have 
\[
    \E{\calZ_r}
    =\exp\{nF_1(\kappa,\alpha,r)+O(\log n)\},
    \quad \text{ where }
    F_1(\kappa,\alpha,r)
    =H(\mu_r)+\alpha\log p(\kappa/r).
\]
The two terms express a simple competition: larger radii provide more
lattice vectors but make each constraint harder to satisfy.

In \cref{sec:1profile} we prove that, for every $\kappa>0$ and
$\alpha\in(1,\alpha_c(\kappa)]$, $F_1$ has a unique maximizer
$r_*=r_*(\kappa,\alpha)>1$ and strictly negative curvature there. Hence
the first moment is concentrated on
\[
    I=\cb{r\in\mathcal R_n:
       \sqrt n\,|r-r_*|\le w_n},
\]
where $w_n$ can be chosen as any slowly diverging window; for concreteness, we take $w_n = \log\log n$. The first moment is also concentrated on ``flat'' vectors satisfying the delocalization condition $\|x\|_\infty\le\log n$. We denote
the restriction of the lattice and solution set to typical vectors by
$\Lambda_n^{\typ},\sol_{n,m,\kappa}^{\typ}$, and $\calZ^{\typ}$. \cref{thm:truncated 1mm} proves that the expected number of discarded
solutions is negligible relative to $\E{\calZ}$. We therefore apply the second-moment method to $\calZ^\typ$. Its contributing pairs have $r,\hat r=r_*+o(1)$,  essentially reducing the three-dimensional saddle-point computation in
\eqref{eq:ov-2mm} to a one-dimensional maximization over the overlap $\theta$.

\subsubsection{Second moment profile function} After truncation, the relevant saddle-point computation in \eqref{eq:ov-2mm} is the maximization over $0 \le |\theta| < 1$ of the second moment profile function,
\[
    F_2(\theta) := H(\mu_{\theta}^{(2)})
    +\alpha\log p^{(2)}(\theta),
\]
with $\mu_\theta^{(2)} := \mu_{r_*,r_*,\theta}^{(2)}$ and $p^{(2)}(\theta) := \PP{|W| \le \kappa/r_*,\,|W_\theta| \le \kappa/r_*}$ where $(W,W_\theta)$ are standard Gaussians with covariance $\theta$. 
As $\mu_0^{(2)}$ factors as a product distribution, we have the identity
\begin{equation}\label{eq:overview F2 zero}
    F_2(0) = 2\,F_1\,.
\end{equation}
Hence, the first and second moments matching at the exponential scale requires $F_2$ to be strictly maximized at $\theta=0$; this property is the first condition in \cref{hyp:maxima}.

We also need the curvature of $F_2$ at $\theta = 0$ to be non-degenerate, which is the second condition of \cref{hyp:maxima}. This condition can be made explicit. Let $\phi$ denote the standard Gaussian density, and define
\[
    d(u):=1-\E{W^2\mid |W|\le u}
         =\frac{2u\,\phi(u)}{p(u)},
    \qquad
    B_{u,\alpha}:=-\frac{\sqrt\alpha}{2}d(u),
\]
where the equality for $d(u)$ follows from integration by parts. Then, as will be proven in \cref{prop:F2'' central}, for $K = \kappa/r_*$,
\[
    -\partial_{\theta\theta}^2 F_2(\kappa,\alpha,r_*,0) = 1-4B_{K,\alpha}^2 = 1 - \frac{2\alpha}{\pi}\frac{K^2e^{-K^2}}{p(K)^2}.
\]

The main content of \cref{sec:2profile} is \cref{thm:unique-max}, which provides an analytic proof that \cref{hyp:maxima} holds for all $\kappa$ modestly large and $\alpha$ bounded away from the linear-algebraic transition at $\alpha = 1$. More precisely, we consider:
\begin{equation}\label{eq:ov-hyp-region}
    \kappa \ge 10, \qquad 1 + \kappa^{-2} \le \alpha \le \alpha_c(\kappa)\,.
\end{equation}
While this characterization is not sharp, it is necessary to lower bound both $\kappa$ and $\alpha$. Indeed, \cref{hyp:maxima} fails at some values of $\kappa < 2$ (see Figure \ref{fig:F2}), and $\calZ = +\infty$ almost surely for $\alpha < 1$ and any $\kappa$. 

We sketch an overview of the proof of \cref{thm:unique-max}. Define the shorthand:
\[
    D_r(\theta)
    :=2H(\mu_r)-H(\mu_{r,r,\theta}^{(2)})
    =
    D\left(\mu_{r,r,\theta}^{(2)}
       \,\middle\|\,\mu_r\otimes\mu_r\right), \qquad  G_K(\theta)
    :=\log\frac{p^{(2)}(K,K,\theta)}{p(K)^2}.
\]
Here, $D_r(\theta)$ is the entropy lost by correlating the two lattice
vectors, while $G_K(\theta)$ is the gain in their joint probability of
satisfying one constraint. Next, define
\[
    q(r):=\frac{r^2}{\sigma(r)^2}, \qquad d(K):=1-\E{W^2\mid |W|\le K},
     \qquad K := \kappa/r_*\,.
\]
As $r_*$ is a critical point of $F_1$, stationarity implies $\alpha d(K)=q(r_*)$. Hence, we can eliminate $\alpha$ by treating $r_*$ as a parameter:
\begin{equation}\label{eq:overview profile gap}
    F_2(0)
    -F_2(\theta)
    =
    D_{r_*}(\theta)
    -\frac{q(r_*)}{d(K)}~G_K(\theta).
\end{equation}

Our goal is to show the
right-hand side of \eqref{eq:overview profile gap} is strictly positive on
$0<|\theta|<1$, for $(\kappa,\alpha)$ within \eqref{eq:ov-hyp-region}. We consider three regimes. When $r_*\ge1.2$ and $K\ge2$, basic convexity estimates
on $D_{r_*}$ and $G_K$ suffice.  When $K\le 2$, the assumption
$\kappa\ge10$ forces $r_*\ge5$, so that the relevant discrete Gaussians are close to their continuous counterparts, yielding convenient approximation of $D$ and $q$, implemented via Poisson summation. Finally, when
$1<r_*<1.2$, the scalar discrete Gaussian is close to Rademacher, and
we can appeal to existing estimates on the SBP
\cite{ALS22}.

In practice, we need to compute the second moment of $\calZ^\typ$ up to vanishing, rather than just sub-exponential, multiplicative error. This creates a number of additional technical obstacles, but none are fundamental. The restriction $r \approx \hat{r} \approx r_*$ needs to be carefully quantified in a separate saddle-point argument. Then, we need to show $F_2''(0) < 0$, which is the second condition in \cref{hyp:maxima}. Finally, to prove strong freezing and to handle the critical case $\alpha = \alpha_c(\kappa)$, we require estimates on the singularities $\lim_{\theta \to \pm 1} F_2'(\theta)$; these are developed in \cref{lemma:F2-endpoint-uniform}. While analogous computations are needed for the SBP, our setting has the substantial additional complication that $\mu^{(2)}$ is defined in terms of non-elementary fixed-point equations, and hence is not as easily worked with.

\subsubsection{Cycle correction and conditional second moment}

Assume \cref{hyp:maxima}, and let $m\le m_0 := \alpha_c(\kappa)n-\eta\log n$. A saddle-point computation gives
\[
    \frac{\E{(\calZ^{\typ})^2}}
         {\E{\calZ^{\typ}}^2}
    = \frac1{\sqrt{1-4B^2}} + o(1)>1,
    \qquad
    B:=B_{\kappa/r_*,\alpha}.
\]
Thus $\calZ^\typ$ has multiplicative fluctuations of order one. These are removed by a dense analogue of the small
subgraph conditioning method \cite{ALS22,ss}. Let $C_k$ be the normalized sum of products of matrix entries
around alternating cycles with $k$ row vertices and $k$ column
vertices, defined formally in \cref{sec:cycle counts}, and consider cycles up to length $T=\lfloor\lambda\log n\rfloor$ for a sufficiently small constant $\lambda$. Let
\[
    Y_T(r)
    =
    \sum_{k=1}^T
    \frac{
      2(2B_{\kappa/r,\alpha})^kC_k
      -(2B_{\kappa/r,\alpha})^{2k}
    }{4k}.
\]
The variables $C_k/\sqrt{2k}$ are
approximately jointly independent standard Gaussians. Conditioning on one
solution of radius $r$ maintains approximate joint Gaussianity of $(C_k)_k$, but shifts the mean of $C_k$ by
$(2B_{\kappa/r,\alpha})^k$.  Conditioning on two nearly orthogonal
solutions of radii $r,\hat r$ adds the two shifts. The main content of \cref{sec:cycle counts} is the quantification of these Gaussian approximations.

These conditioning calculations are made particularly transparent by the Gaussianity of the model.
Conditioning on $x$ changes each row of $A$ only in the direction spanned by
$x$, while conditioning on $(x,y)$ changes it only in their two-dimensional
span. We couple the conditional and unconditional matrices by resampling
these projections. The leading change in $C_k$ is the trace of the $k$th
power of the resulting one- or two-dimensional covariance perturbation,
which gives the shifts above. The delocalization estimates 
$\|x\|_\infty,\|y\|_\infty\le\log n$ ensure that the sub-leading terms are
negligible.

In \cref{sec:second moment method}, we utilize these cycle counts to remove the remaining fluctuations from $\calZ$. Let $L=\sqrt{\log n}$ and define the reweighted solution count:
\[
    Q:=\sum_{r\in I}\calZ_r^{\typ}
       e^{-Y_T(r)}\ind_{\{Y_T(r)\ge-L\}}.
\]
The quantitative Gaussian approximations for $C_k$ under the null, single planted, and pair planted distributions allow us to compute high moments of $C_k$, $\ind(x \in \sol)\,C_k$, and $\ind(x,y \in \sol)\,C_k$. In particular, we can Taylor expand $\exp(-Y_T(r))$ as a polynomial of $(C_k)_k$ and compute $\E{Q}$ and $\E{Q^2}$. Doing so yields
\[
    \E{Q}=(1+o(1))\E{\calZ^{\typ}},
    \qquad
    \E{Q^2}=(1+o(1))\E{Q}^2.
\]
Hence $Q/\E{Q}\to1$ in probability. Moreover, $Y_T(r)=Y_T(r_*)+o(1)$ uniformly over $r \in I$, and $Y_T(r_*)$ converges to a Gaussian.
It follows that
\[
    \log\frac{\calZ^{\typ}}{\E{\calZ^{\typ}}}
    \ \xrightarrow{\ d\ }\
    \mathcal N\left(
       \frac14\log(1-4B^2),
       -\frac12\log(1-4B^2)
    \right).
\]
During the second moment computation for $Q$, we naturally obtain strong control over the solution geometry: with high probability, only an $\exp(-\tilde{O}(\sqrt n))$--fraction of solutions are not completely isolated, establishing strong freezing. Moreover, only an inverse-polynomial fraction of pairs of solutions have absolute inner product larger than
$\sqrt n\log n$. In particular, for $m \ge m_0$, concentration of $Q$ implies that with high probability  $\calZ \approx \E{\calZ} = n^{O(\eta)}$, and hence for $\eta$ sufficiently small, \textit{all} (non-collinear) pairs of solutions are nearly orthogonal. This ``complete freezing'' property will be used for the critical window.

\subsubsection{Critical window}
Reveal the first $m_0$ rows of $A$. As just described, taking $\eta$ small yields complete freezing: with high probability, every pair of
solutions $(x,y)$ (with $x \neq \pm y$) satisfies
\[
    |\la x,y\ra|\le\sqrt n\log n.
\]
After accounting for antipodal pairs of vectors, the normalized Gram matrix of $\sol_0 := \sol_{n,m_0,\kappa}^\typ$ is $o(1)$-close to the identity in Frobenius norm. For any bounded $t \in \ZZ$, the Gaussian projections 
\[
    (\la \mathrm{row}_s(A),\,x \ra)_{x \in \sol_0,~m_0 < s \le \alpha_c n +t}
\]
are jointly vanishingly close in total variation to independent Gaussians. Hence, if $\ell$ further constraints are revealed, each representative in $\sol_0$
survives with probability $(1+o(1))\,p_c^\ell$ (uniformly over the
slightly different radii in the thin shell of typical solutions).  Conditional on $\calZ_0 := |\sol_0|$,
\begin{equation}\label{eq:overview conditional extinction}
    \PP{\text{no solution survives the next $\ell$ constraints}}
    =
    \exp\left\{-\frac12p_c^\ell\calZ_0\right\}+o(1).
\end{equation}
The factor $1/2$ accounts for antipodal pairs. The log-normal characterization of $\calZ_0$ yields the distributional limit for the satisfiability threshold claimed in \cref{thm:main}, with the specific definition of the function $f_\kappa$ given in \cref{thm:main2}.

\section{First moment profile function}\label{sec:1profile}

This section considers the profile of typical vectors in $\sol_{n,m,\kappa}$. In particular, we develop large deviation bounds on the Euclidean norm and infinity norm of such vectors. We begin with the maximizing radius, which we show below is well-defined:
\[
    r_*=r_*(\kappa,\alpha):=\arg\max_{r\ge 1}F_1(\kappa,\alpha,r).
\]
Set $w_n:=\log\log n$ and
\begin{align*}
    I=I_n(\kappa,\alpha)
      &:=\cb{r\in\calR_n:\sqrt n\,|r-r_*|\le w_n},\\
    B_\infty(\log n)&:=\cb{x\in\RR^n:\|x\|_\infty\le\log n}.
\end{align*}
Define
\begin{align*}
    \Lambda_n^\typ(r)&:=\Lambda_n(r)\cap B_\infty(\log n),\\
    \sol_{n,m,\kappa}^\typ(r)&:=\sol_{n,m,\kappa}(r)\cap B_\infty(\log n),\\
    \calZ_r^\typ&:=|\sol_{n,m,\kappa}^\typ(r)|,
\end{align*}
and
\begin{align*}
    \Lambda_n^\typ&:=\bigcup_{r\in I}\Lambda_n^\typ(r),\\
    \sol_{n,m,\kappa}^\typ&:=\bigcup_{r\in I}\sol_{n,m,\kappa}^\typ(r),\\
    \calZ^\typ=\calZ_I^\typ&:=|\sol_{n,m,\kappa}^\typ|=\sum_{r\in I}\calZ_r^\typ.
\end{align*}

\begin{theorem}[First-moment truncation]\label{thm:truncated 1mm}
Fix $\kappa>0$, let $m=m(n)$, and set $\alpha=m/n$. Suppose that
$\alpha$ ranges over a compact subset of $(1,\alpha_c(\kappa)]$. Then
$r_*$ exists and is uniquely defined for each $n$. Moreover, there is a constant
$c>0$ such that, uniformly in this range,
\begin{equation}\label{eq:F1 dominated by r*}
    \E{\calZ}=\pa{1+O\pa{e^{-cw_n^2}}}\E{\calZ_I^\typ},
\end{equation}
and
\[
    \PP{0\le \calZ-\calZ_I^\typ
        \le e^{-cw_n^2}\E{\calZ_I^\typ}}
    \ge 1-e^{-cw_n^2}.
\]
\end{theorem}
The particular choice $w_n=\log\log n$ is arbitrary; the same proof works for any $w_n\to\infty$ with $w_n=o(\sqrt{\log n})$. The rest of this section is devoted to the proof of \cref{thm:truncated 1mm}.

\subsection{Integer geometry of the sphere}

For $r,\hat r\ge1$ and $\theta\in[-1,1]$, define
\begin{align*}
    \Lambda_n(r,\hat r,\theta)
    &:=\cb{(\sigma,\tau)\in\Lambda_n(r)\times\Lambda_n(\hat r):
        \la\sigma,\tau\ra=\theta r\hat r n},\\
    N(r)&:=|\Lambda_n(r)|,
    \qquad
    N(r,\hat r,\theta):=|\Lambda_n(r,\hat r,\theta)|.
\end{align*}

We will use the following identity, relating maximum entropy distributions to counting. 
\begin{lemma}[{\cite[Lemma~3.2]{mckinley2023maximum}}]
\label{lemma:max entropy}
Let $\Omega$ be a countable set, let $f:\Omega\to\RR^d$, and fix
$b\in\RR^d$. Suppose that, for some $\lambda\in\RR^d$,
\[
    Z:=\sum_{x\in\Omega}e^{\la\lambda,f(x)\ra}<\infty,
    \qquad
    \mu(x):=Z^{-1}e^{\la\lambda,f(x)\ra}
\]
satisfies $\E[\mu]{f}=b$. Then, writing
$\Omega_b:=\{x\in\Omega:f(x)=b\}$ and assuming that this set is finite,
\[
    |\Omega_b|=e^{H(\mu)}\mu(\Omega_b).
\]
\end{lemma}

\begin{proof}
Since $H(\mu)=\log Z-\la\lambda,b\ra$, for $x\in\Omega_b$ one has
$\mu(x)=Z^{-1}e^{\la\lambda,b\ra}=e^{-H(\mu)}$. Then $
\PP[\mu]{\Omega_b} = |\Omega_b|e^{-H(\mu)}$. Rearranging gives the claim.
\end{proof}

By nonnegativity of relative entropy, the assumed exponential form implies that $\mu$ is the maximum-entropy distribution on $\Omega$ satisfying $\E{f} = b$.

\begin{proposition}[Asymptotic count of lattice points]\label{prop:1d-lattice-count}
For every $r\ge1$,
\[
    N(r)
    =e^{nH(\mu_r)}
      \PP[X\sim\mu_r^{\otimes n}]{\|X\|_2^2=nr^2}
    \le e^{nH(\mu_r)}.
\]
Moreover, for every $\delta>0$ and $C<\infty$, uniformly over
$r\in[1+\delta,C]$ such that $nr^2\in\ZZ$,
\[
    N(r)
    =e^{nH(\mu_r)}
      \frac{1+o(1)}{\sqrt{2\pi n\Var_{\mu_r}(X_1^2)}}
    =\Theta(n^{-1/2})e^{nH(\mu_r)}.
\]
For $r=1$, one has $N(1)=2^n$.
\end{proposition}

\begin{proof}[Proof of \cref{prop:1d-lattice-count}]
Apply \cref{lemma:max entropy} with
$\Omega=(\ZZ\setminus\{0\})^n$ and $f(x)=\|x\|_2^2$; take $\mu_r^{\otimes n}$ as the distribution $\mu$ in the Lemma, as its log-density is
affine in $f$ and $\E{f}=nr^2$. This gives the first identity. The remaining assertion follows from the lattice-valued local central limit theorem \cite[Ch.~5, \S\,22]{lclt}, uniformly when $r$ ranges over a compact subset of $(1,\infty)$. Indeed, in this regime the moments are uniformly bounded, $\Var_{\mu_r}(X_1^2)$ is bounded away from zero, and the lattice span of $X_1^2$ is one.
\end{proof}

The analogous pair count will be used in the second moment analysis. Let
$\Gamma_{r,\hat r,\theta}$ denote the covariance matrix of
$(X^2,XY,Y^2)$ under $\mu_{r,\hat r,\theta}^{(2)}$. Say a triple $(r,\hat{r},\theta)$ is \textit{feasible} if $N(r,\hat r,\theta)>0$ and, as described right after \cref{def:DG 2}, the corresponding discrete Gaussian has density proportional to the exponential of a negative-definite quadratic form.

\begin{proposition}[Asymptotic lattice-pair count]\label{prop:2d-lattice-count}
For every feasible $(r,\hat r,\theta)$ with $r,\hat r>1$ and $\theta\in(-1,1)$,
\begin{align*}
    N(r,\hat r,\theta)
    &=e^{nH(\mu_{r,\hat r,\theta}^{(2)})}\,
      \PP[(X_i,Y_i)\stackrel{\mathrm{iid}}\sim\mu_{r,\hat r,\theta}^{(2)}]
      {\sum_{i=1}^n(X_i^2,X_iY_i,Y_i^2)
       =(nr^2,n\theta r\hat r,n\hat r^2)}\le e^{nH(\mu_{r,\hat r,\theta}^{(2)})}.
\end{align*}
Fix $\delta,\eps>0$ and $C<\infty$. There exists $\eta>0$
such that, uniformly over feasible triples satisfying
\[
    r,\hat r\in[1+\delta,C],\qquad
    |r-\hat r|\le\eta,\qquad
    |\theta|\le1-\eps,
\]
one has
\begin{equation}\label{eq:2d lattice count}
    N(r,\hat r,\theta)
    =e^{nH(\mu_{r,\hat r,\theta}^{(2)})}
      \frac{1+o(1)}{(2\pi n)^{3/2}\sqrt{\det\Gamma_{r,\hat r,\theta}}}
    =e^{nH(\mu_{r,\hat r,\theta}^{(2)})+O(\log n)},
\end{equation}
and also
\begin{equation}\label{eq:N2 expansion}
    \frac1n\log N(r,\hat r,\theta)
    =\frac1n\log\frac{N(r)N(\hat r)}{\sqrt{2\pi n r^2\hat r^2}}
      +\frac{\theta^2}{2}
       \pa{\partial_{\theta\theta}^2H(\mu_{r,\hat r,0}^{(2)})+O(n^{-1})}
      +O(\theta^4)+o(n^{-1}).
\end{equation}
\end{proposition}
The expansion \eqref{eq:N2 expansion} is later made explicit by computing
$\partial_{\theta\theta}^2H(\mu_{r,\hat r,0}^{(2)})=-1$ in \eqref{eq:H''}. The constraints on $(r,\hat{r},\theta)$ assumed in the last two formulas ensure for sufficiently small $\eta$ that the
coefficients $(a,b,c)$ are in a compact subset of
$\{(a,b,c):a,c<0,\ b^2<4ac\}$.
In particular, $4ac-b^2$ is bounded away from zero.

\begin{proof}[Proof of \cref{prop:2d-lattice-count}]
Apply \cref{lemma:max entropy} with
\[
    \Omega=((\ZZ\setminus\{0\})^2)^n,
    \qquad
    f((x_i,y_i)_{i\le n})=\sum_i(x_i^2,x_iy_i,y_i^2).
\]
The associated distribution is
$(\mu_{r,\hat r,\theta}^{(2)})^{\otimes n}$, which gives the first exact
identity and upper bound. Next, on the compact parameter set, the covariance matrix
$\Gamma_{r,\hat r,\theta}$ is uniformly non-degenerate, all moments are
uniformly bounded, and the lattice generated by differences of support
points is $\mathbb Z^3$. The multivariate lattice local central limit
theorem therefore gives \eqref{eq:2d lattice count}.

Finally, at $\theta=0$, note that
$\mu_{r,\hat r,0}^{(2)}=\mu_r\otimes\mu_{\hat r}$ and
\[
    \det\Gamma_{r,\hat r,0}
    =\Var_{\mu_r}(X^2)\,r^2\hat r^2\,
      \Var_{\mu_{\hat r}}(Y^2).
\]
Taking logarithms in \eqref{eq:2d lattice count}, using
\cref{prop:1d-lattice-count}, and Taylor expanding the two even analytic
functions
$H(\mu_{r,\hat r,\theta}^{(2)})$ and
$\log\det\Gamma_{r,\hat r,\theta}$ at $\theta=0$ gives
\eqref{eq:N2 expansion}.
\end{proof}

\begin{proposition}[Entropy of the discrete Gaussian]\label{prop:H(DG) properties}
Let $r>1$ and write $v=v(r):=r^2$. Viewing $H(\mu_r)$ as a function of $v$, one has
\[
    \partial_{vv}^2H(\mu_r)<0<\partial_vH(\mu_r)
\]
and
\[
    -\frac{\partial_vH(\mu_r)}{\partial_{vv}^2H(\mu_r)}
    =\frac1{2\sigma^2}\pa{\E[\mu_r]{X^4}-v^2}.
\]
Moreover, as $r\to\infty$,
\[
    \partial_vH(\mu_r)=\frac{1+o(1)}{2r^2},
    \qquad
    H(\mu_r)=\log r+O(1).
\]
\end{proposition}

\begin{proof}
Set $a:=a(v)$ and $b:=b(v)$, with
\[
    \mu_r(x)=e^{ax^2+b}.
\]
Thus $a=-1/(2\sigma^2)$ and $b=-\log Z$. The relations
\begin{equation}\label{eq:relations-mu_r}
    v=\sum_xe^{ax^2+b}x^2,
    \qquad
    1=\sum_xe^{ax^2+b}
\end{equation}
give $a'v+b'=0$. Since $H(\mu_r)=-(av+b)$,
\[
    -\partial_vH(\mu_r)=a,
    \qquad
    -\partial_{vv}^2H(\mu_r)=a'.
\]
Here $a<0$, while differentiating the first identity in
\eqref{eq:relations-mu_r} gives
\begin{equation}\label{eq:implicitdiff-mu_r}
    1=a'\E{X^4}+b'v.
\end{equation}
Combining this with $a'v+b'=0$ gives
\[
    1=a'\pa{\E{X^4}-v^2}.
\]
Hence, $a'>0$. It follows that
\[
    -\frac{\partial_vH(\mu_r)}{\partial_{vv}^2H(\mu_r)}
    =-\frac a{a'}
    =-a\pa{\E{X^4}-v^2}
    =\frac1{2\sigma^2}\pa{\E{X^4}-v^2}.
\]

For the asymptotics, let $t:=-a=1/(2\sigma^2)$. By Poisson summation,
\[
    Z(t):=2\sum_{k\ge1}e^{-tk^2}
    =\sqrt{\frac\pi t}\pa{1+O(e^{-\pi^2/t})}-1.
\]
Hence
\[
    r^2=-\frac{d}{dt}\log Z(t)
    =\frac1{2t}\pa{1+O(\sqrt t)}.
\]
Since $\partial_vH(\mu_r)=t$, the first asymptotic follows. The identity
$H(\mu_r)=tr^2+\log Z(t)$ then gives $H(\mu_r)=\log r+O(1)$.
\end{proof}

\begin{proposition}[Discrete Gaussian moment inequality]
\label{prop:DG moments}
Let $X\sim\mu_r$, where $\mu_r$ is the punctured discrete Gaussian with variance $r^2$. For all $r>1$,
\[
    \frac1{2\sigma^2}\pa{\E{X^4}-r^4}<r^2.
\]
\end{proposition}

Note that if $\mu_r$ were replaced by the continuous Gaussian density, this would hold with equality.

\begin{proof}
Let $t=1/(2\sigma^2)$, and define
\[
    Z(t)=2\sum_{j\ge1}e^{-tj^2},
    \qquad
    R(t)=-\frac{d}{dt}\log Z(t).
\]
Then $R(t)=r^2$, $R'(t)=-\Var(X^2)$, and the desired inequality is
equivalent to $(tR(t))'>0$.
First, consider $0<t\le1/2$. Poisson summation gives
\[
    Z(t)=\bigl(\sqrt{\pi/t}-1\bigr)(1+\delta(t)),
    \qquad
    \delta(t)=
    \frac{2\sum_{m\ge1}e^{-\pi^2m^2/t}}{1-\sqrt{t/\pi}},
\]
and hence
\[
    tR(t)=\frac12+\frac1{2(\sqrt{\pi/t}-1)}
           -t\frac{d}{dt}\log(1+\delta(t)).
\]
The derivative of the first two terms is greater than $.53$ for $0 < t \le 1/2$, using the standard truncation formulas for geometric series. On the other hand, differentiating and using the inequality
$m^2\ge1+3(m-1)$ gives an absolute bound of, e.g., $10^{-4}$ for the
derivative of the last term. Hence $(tR(t))' > .5 > 0$, as desired. 

It remains to consider $t\ge1/2$. Let $Y=X^2-1$. Since
$\Var(Y)\le\E{Y^2}$,
\[
    (tR(t))'
    \ge
    \frac{1+\sum_{j\ge2}
       \bigl((j^2-1)-t(j^2-1)^2\bigr)e^{-t(j^2-1)}}
         {1+\sum_{j\ge2}e^{-t(j^2-1)}}.
\]
Keeping the terms $j=2,3$ and bounding the remaining geometric tail,
the numerator is at least
\[
    1+(3-9t)e^{-3t}+(8-64t)e^{-8t}-300te^{-15t}
    >1-3e^{-2}-24e^{-4}-150e^{-15/2}>0.
\]
Thus $(tR(t))' > 0$ for $t \ge 1/2$, completing the proof of the proposition in all cases. 
\end{proof}

\subsection{Unique maximizing radius}

The main result of this subsection is that the first moment profile
$F_1(r):=F_1(\kappa,\alpha,r)$ has a unique maximizer, without any assumption that $\kappa$ is large.

\begin{lemma}[Shape of first moment]\label{lemma:F1 shape}
Fix $\kappa>0$ and $\alpha>1$, and set
\[
    f(v):=F_1(\kappa,\alpha,\sqrt v),
    \qquad v>1.
\]
For every $v>1$, at least one of the following holds:
\begin{enumerate}
    \item $f''(v)<0$;
    \item $f'(v)<0$.
\end{enumerate}
Moreover, $f'(v)>0$ for $v$ sufficiently close to $1$, while
$f'(v)<0$ for all sufficiently large $v$.
\end{lemma}

\begin{corollary}[Unique maximum for first moment]\label{cor:F1 unique max}
Fix $\kappa>0$ and $\alpha>1$. There exist a unique
$r_*\in(1,\infty)$ and constants $c,c'>0$, depending only on $\kappa$ and $\alpha$, such that for all $r\ge1$,
\[
    F_1(\kappa,\alpha,r)
    \le F_1(\kappa,\alpha,r_*)-
       \min\{c',c|r-r_*|^2\}.
\]
\end{corollary}

\begin{proof}[Proof of \cref{lemma:F1 shape}]
We first prove the dichotomy. In what follows, primes denote derivatives with respect to $v=r^2$. A computation gives
\[
    \partial_{vv}^2\log p(\kappa/r)
    =-\frac{\kappa e^{-\kappa^2/r^2}
      \left(\sqrt{2\pi}e^{\kappa^2/(2r^2)}(\kappa^2-3r^2)p(\kappa/r)
      +2\kappa r\right)}
      {4\pi r^7p(\kappa/r)^2},
\]
which is strictly negative for $r\le\kappa/2$.

Fix $v$ and suppose that $f''(v)\ge0$. Then necessarily $r>\kappa/2$ and
\[
    \partial_{vv}^2\log p(\kappa/r)
    \ge-\frac1\alpha\partial_{vv}^2H(\mu_r)>0.
\]
Also,
\begin{equation}\label{eq:E1}
    \partial_v\log p(\kappa/r)
    =-\frac{\kappa e^{-\kappa^2/(2r^2)}}
      {\sqrt{2\pi}r^3p(\kappa/r)}<0.
\end{equation}
Thus it is enough to show
\begin{equation}\label{eq:dichotomy goal}
    \frac{H(\mu_r)'(\log p(\kappa/r))''}
         {H(\mu_r)''(\log p(\kappa/r))'}<1.
\end{equation}
A direct computation yields
\[
    \frac{(\log p(\kappa/r))''}{(\log p(\kappa/r))'}
    =\frac1{r^2}\pa{
       \frac1{\sqrt{2\pi}}\frac\kappa r
       \frac{e^{-\kappa^2/(2r^2)}}{p(\kappa/r)}
       +\frac{\kappa^2}{2r^2}-\frac32}.
\]
The left-hand side is negative under our assumption. Using
$p(x)\le x\sqrt{2/\pi}$ and $e^{-x^2/2}\ge1-x^2/2$, we obtain
\[
    \left|\frac{(\log p(\kappa/r))''}{(\log p(\kappa/r))'}\right|
    \le\frac1{r^2}\pa{1-\frac{\kappa^2}{4r^2}}.
\]
On the other hand, \cref{prop:H(DG) properties,prop:DG moments} give
\[
    \left|\frac{H(\mu_r)'}{H(\mu_r)''}\right|
    =\frac1{2\sigma^2}\pa{\E{X^4}-r^4}<r^2.
\]
Combining the last two displays proves \eqref{eq:dichotomy goal}, and hence the dichotomy.

We next show that $f'(v)>0$ near $v=1$. Write
$\mu_r(x)=e^{ax^2+b}$ and let $q=\PP{|X|>1}$. Since
$v\ge1+3q$, one has $q\to0$ as $v\downarrow1$, while
\[
    e^{3a}=\frac{\PP{|X|=2}}{\PP{|X|=1}}\le\frac q{1-q}.
\]
Thus $a\to-\infty$. Since $\partial_vH(\mu_r)=-a$ and
$\partial_v\log p(\kappa/r)$ remains bounded, $f'(v)>0$ near $1$.
Finally, \cref{prop:H(DG) properties} and \eqref{eq:E1} give, as
$r\to\infty$,
\[
    f'(v)=\frac{1-\alpha+o(1)}{2r^2}<0.
\]
\end{proof}

\begin{proof}[Proof of \cref{cor:F1 unique max}]
Let $f$ be as in \cref{lemma:F1 shape}. Since $f'>0$ near $1$ and
$f'<0$ eventually,
\[
    v_*:=\inf\{v>1:f'(v)\le0\}
\]
is finite, satisfies $v_*>1$, and obeys $f'(v_*)=0$. The dichotomy gives
$f''(v_*)<0$, so $f'<0$ immediately to the right of $v_*$. If $f'$ returned to zero, let $w>v_*$ be its first subsequent zero. Then
$f'<0$ on $(v_*,w)$, and hence $f''(w)\ge0$, contradicting the dichotomy. Therefore $f'>0$ on $(1,v_*)$ and $f'<0$ on $(v_*,\infty)$, so $v_*$ is the unique maximizer.

Since $f''(v_*)<0$, there are $c_0,\delta>0$ such that
\[
    f(v)\le f(v_*)-c_0(v-v_*)^2
    \qquad\text{for }|v-v_*|\le\delta.
\]
Monotonicity on either side of $v_*$ gives a constant $c_1>0$ such that
$f(v)\le f(v_*)-c_1$ whenever $|v-v_*|\ge\delta$. Writing
$v=r^2$ and $v_*=r_*^2$ now gives the claimed global estimate.
\end{proof}

\begin{remark}[Capacity]
Recall that 
\[
    \alpha_c(\kappa) := \sup_{r\ge 1} \frac{H(\mu_r)}{-\log p(\kappa/r)}\,.
\]
It is a minor subtlety that, \textit{a priori}, the supremum in this definition may only be attained in the limit as $r$ diverges so that $F_1(\kappa,\alpha_c(\kappa),r)$ could potentially be negative for every finite $r$. This is a danger if the margin $\kappa$ is extremely small, so that the model goes from almost surely having infinite solutions when $\alpha < 1$ and $A$ has a kernel, to suddenly having no solutions when $\alpha = 1$. Fortunately, this turns out to be the only obstruction: the supremum is attained as long as $\alpha_c(\kappa)>1$. A Taylor expansion gives
\[
    p(\kappa/r)
    =\sqrt{\frac2\pi}\,\frac{\kappa}{r}+O(r^{-3}).
\]
Combining this estimate with \cref{prop:H(DG) properties}, we have $-H(\mu_r)/\log p(\kappa/r) \to 1$ as $r \to \infty$. In particular, $\alpha_c(\kappa)\ge1$. Thus, if $\alpha_c(\kappa)>1$, continuity of $H(\mu_r)/\log p(\kappa/r)$ implies that the supremum above is attained at some finite $r$. Together with \cref{cor:F1 unique max}, we have for $\alpha>1$,
\begin{equation}\label{eq:F1-capacity-sign}
    \max_{r\ge1}F_1(\kappa,\alpha,r)
    =
    \begin{cases}
        >0, & \alpha<\alpha_c(\kappa),\\
        0,  & \alpha=\alpha_c(\kappa),\\
        <0, & \alpha>\alpha_c(\kappa).
    \end{cases}
\end{equation}
\end{remark}

\begin{proof}[Proof of \cref{thm:truncated 1mm}]
The existence and uniqueness of $r_*$ follow from
\cref{cor:F1 unique max}. Write
\[F_1(r):=F_1(\kappa,\alpha,r), \qquad  F_*:=F_1(r_*).\]
Throughout the proof, $c,C>0$ denote constants depending only on
$(\kappa,\alpha)$ which may change from line to line.

We first control the infinity norm. For all sufficiently large $n$, the interval $I$ lies in a fixed compact subset of $(1,\infty)$. Uniformly over $r\in I$, \cref{prop:1d-lattice-count} gives
\[
    \PP[X\sim\mu_r^{\otimes n}]{\|X\|_2^2=nr^2}=\Theta(n^{-1/2}).
\]
Writing $\mu_r(x)\propto e^{a(r)x^2}$, continuity gives $a(r)\le-c<0$ uniformly on this compact set. A union bound gives
\[
    \PP[X\sim\mu_r^{\otimes n}]{\|X\|_\infty>\log n}
    \le e^{-c(\log n)^2}.
\]
Since the uniform measure on $\Lambda_n(r)$ is exactly 
$\mu_r^{\otimes n}$ conditioned on $\{\|X\|_2^2=nr^2\}$, it follows that
\[
    |\Lambda_n(r)\setminus\Lambda_n^\typ(r)| \le |\Lambda_n(r)|\, \frac{\PP[X\sim\mu_r^{\otimes n}]{\|X\|_\infty>\log n}}{\PP[X\sim\mu_r^{\otimes n}]{\|X\|_2^2=nr^2}}
    \le e^{-c(\log n)^2}|\Lambda_n(r)|.
\]
Rotational invariance of the Gaussian constraints therefore gives
\begin{equation}\label{eq:E[Z] - typical}
    \E{\calZ_I-\calZ_I^\typ}
    \le e^{-c(\log n)^2}\E{\calZ_I},
    \qquad
    \E{\calZ_I}
    =\pa{1+O(e^{-c(\log n)^2})}\E{\calZ_I^\typ}.
\end{equation}

It remains to truncate the radius. For sufficiently
large $n$, the uniform local central limit
theorem and a Taylor expansion of $F_1$ around $r_*$ give
\begin{equation}\label{eq:lower bound EZ I}
    \E{\calZ_I}\ge c\,e^{nF_*}.
\end{equation}

Fix $\delta\in(0,r_*-1)$. For $|r-r_*|\le\delta$, the local central
limit theorem and \cref{cor:F1 unique max} give
\[
    \E{\calZ_r}
    \le Cn^{-1/2}e^{nF_*}e^{-cn(r-r_*)^2}.
\]
Writing $q:=nr^2$ and $q_*:=nr_*^2$, we have in this range
\[
    n(r-r_*)^2=\frac{(q-q_*)^2}{n(r+r_*)^2}
    \ge c\frac{(q-q_*)^2}{n}.
\]
Summing the resulting Gaussian tail over $q\in\mathbb Z$ yields
\begin{equation}\label{eq:F1 trunc mid}
    \sum_{\substack{r\in\calR_n\setminus I\\|r-r_*|\le\delta}}
        \E{\calZ_r}
    \le e^{-cw_n^2}\,\E{\calZ_I}.
\end{equation}

Fix a sufficiently large constant $R>r_*+\delta$. By continuity and uniqueness of the maximizer,
$F_1(r)\le F_*-c$ when $r\in[1,R]$ and $|r-r_*|>\delta$. The exact counting identity in
\cref{prop:1d-lattice-count} gives
$\E{\calZ_r}\le e^{nF_1(r)}$, and there are only $O(n)$ possible values of $nr^2$ in this bounded range. Hence
\[
    \sum_{\substack{r\in\calR_n\cap[1,R]\\|r-r_*|>\delta}}
        \E{\calZ_r}
    \le e^{-cn}\E{\calZ_I}.
\]
Finally, \cref{prop:H(DG) properties} and
$p(\kappa/r)\le\sqrt{2/\pi}\,\kappa/r$ imply, for $r$ sufficiently large,
\[
    \E{\calZ_r}\le\pa{C_{\kappa,\alpha}r^{1-\alpha}}^n.
\]
Grouping the possible values of $nr^2$ according to $k\le r<k+1$ gives
\begin{equation}\label{eq:F1 trunc extreme}
    \sum_{r\in\calR_n:\,r\ge R}\E{\calZ_r}
    \le\sum_{k\ge R}n(2k+1)
       \pa{C_{\kappa,\alpha}k^{1-\alpha}}^n
    \le e^{-cn}\E{\calZ_I},
\end{equation}
where $R$ is chosen sufficiently large and we used $F_*\ge0$ from \eqref{eq:F1-capacity-sign}, together with \eqref{eq:lower bound EZ I}. Combining the above computations gives
\[
    \E{\calZ-\calZ_I}
    \le e^{-cw_n^2}\E{\calZ_I}.
\]
Together with \eqref{eq:E[Z] - typical}, this proves
\eqref{eq:F1 dominated by r*}. Applying Markov's inequality also yields
\[
    \PP{0\le\calZ-\calZ_I^\typ
        \le e^{-cw_n^2}\E{\calZ_I^\typ}}
    \ge1-e^{-cw_n^2}.
\]

Finally, let $\alpha$ range over a compact subset of
$(1,\alpha_c(\kappa)]$. Recalling the strict inequality
\[
    \left.\partial_{vv}^2F_1(\kappa,\alpha,\sqrt v)
    \right|_{v=r_*^2}<0,
\]
the implicit function theorem implies that $r_*(\kappa,\alpha)$ is continuous in $\alpha$ and hence ranges over
a compact subset of $(1,\infty)$. Thus, the local central limit theorem,
curvature, compactness, and tail estimates above may be chosen uniformly.
\end{proof}

\section{Second moment profile function}
\label{sec:2profile}

The main purpose of this section is to prove \cref{thm:unique-max} by verifying \cref{hyp:maxima} for $\kappa$
large throughout the relevant range of $\alpha$ away from the
linear-algebraic transition at $\alpha=1$. 

Define the following second moment profile function, generalizing that in~\eqref{eq:F2-definitionintro}, by
\begin{equation}\label{eq:F2-definition}
    F_2(\kappa,\alpha,r,\hat r,\theta)
    :=H(\mu^{(2)}_{r,\hat r,\theta})
      +\alpha\log p^{(2)}(\kappa/r,\kappa/\hat r,\theta).
\end{equation}

\subsection{Basic identities and the profile reduction}
In order to analyze $F_2$, we need access to representation of $\mu^{(2)}$ in \cref{prop:dg2}. Fortunately, \eqref{eq:dg2-condition} is automatically satisfied after truncating to typical solutions.

\begin{proposition}
    Let $r,\hat{r} \in I$ and $\theta \in (-1,1)$. If $\Lambda_n^\typ(r,\hat{r},\theta)\neq \emptyset$, then for all $n$ sufficiently large, \eqref{eq:dg2-condition} is satisfied.
\end{proposition}

\begin{proof}
If $r=\hat r$, then \eqref{eq:dg2-condition} is immediate, so assume $r \neq \hat{r}$. Assume without loss of generality that $\theta \in [0,1)$ and then fix some
$(x,y)\in\Lambda_n^\typ(r,\hat{r},\theta)$. Since the coordinates of $x$ and $y$ are integers of absolute
value at most $\log n$,
\begin{align*}
    |r^2-\hat r^2|
    \le\frac{1}{n}\sum_{i=1}^n|x_i-y_i|\,|x_i+y_i|&\le\frac{2\log n}{n}\sum_{i=1}^n|x_i-y_i| \le 2\log n\,(r^2+\hat r^2-2r\hat r\theta).
\end{align*}
Here we used $|k|\le k^2$ for every integer $k$.
By \cref{sec:1profile}, $r_*$ stays in a compact subset of
$(1,\infty)$. Since $r,\hat r\in I$,
\[
    \log n\,|r^2-\hat r^2|
    =O\left(\frac{\log n\log\log n}{\sqrt n}\right)=o(1).
\]
Consequently, uniformly in these parameters, for all sufficiently
large $n$,
\begin{align*}
    \left(\sqrt{r^2-1}-\sqrt{\hat r^2-1}\right)^2
    =\frac{(r^2-\hat r^2)^2}
        {\left(\sqrt{r^2-1}+\sqrt{\hat r^2-1}\right)^2} &<\frac{|r^2-\hat r^2|}{2\log n} \le r^2+\hat r^2-2r\hat r\theta.
\end{align*}
Rearranging gives \eqref{eq:dg2-condition} as required. 
\end{proof}

In particular, for all ranges of $(r,\hat{r},\theta)$ considered in the remainder of this article (with the exception of the endpoints $\theta = \pm 1$, which are handled separately), we may use the exponential representation of the discrete planar Gaussian in \cref{prop:dg2}.

We now turn to the proof of \cref{thm:unique-max}, beginning with computing derivatives of $H(\mu_{r,\hat r, \theta}^{(2)})$ with respect to $\theta$. 

\begin{proposition}[Entropy derivatives]\label{prop:F2 entropy derivatives}
Fix $r,\hat r>1$ and restrict $\theta\in(-1,1)$
to the domain \eqref{eq:dg2-condition}. Then
\begin{equation}\label{eq:H-prime}
    \partial_\theta H(\mu^{(2)}_{r,\hat r,\theta})
    =-br\hat r.
\end{equation}
Moreover, 
\begin{equation}\label{eq:H''}
    \left.
    \partial_{\theta\theta}^2
    H(\mu^{(2)}_{r,\hat r,\theta})
    \right|_{\theta=0}=-1,
\end{equation}
and $H(\mu_{r,\hat{r},\theta}^{(2)})$ is strictly concave in $\theta$ on its domain.
\end{proposition}

The proof uses standard ideas from the theory of exponential families. In case these manipulations are opaque, we note that tedious direct computation of the derivatives of $H$ gives the same result. 

\begin{proof}
Let
\[
    T(x,y)=(x^2,xy,y^2),\qquad
    m(\theta)=(r^2,\theta r\hat r,\hat r^2),
\]
and write
\[
    \mathbf a(\theta)=(a(\theta),b(\theta),c(\theta)),
    \qquad
    h(\theta)=H(\mu^{(2)}_{r,\hat r,\theta}).
\]
Differentiating the log-partition function gives
\[
    \nabla\Psi(\mathbf a)=\E{T(X,Y)}=m(\theta),
    \qquad
    \nabla^2\Psi(\mathbf a)
      =\mathrm{Cov}(T(X,Y))
      =\Gamma_{r,\hat r,\theta}.
\]
We claim that this covariance matrix is positive definite. Indeed, for
$u=(u_1,u_2,u_3)\in\RR^3$,
\[
    u^\top\Gamma_{r,\hat r,\theta}u
      =\Var\bigl(u_1X^2+u_2XY+u_3Y^2\bigr).
\]
If this vanished for some $u\ne0$, then since the law has full support on
$(\ZZ\setminus\{0\})^2$, the polynomial
$u_1x^2+u_2xy+u_3y^2$ would be constant there. Fixing $y=1$ and varying
$x$ gives $u_1=u_2=0$, after which varying $y$ gives $u_3=0$, a
contradiction. 

Since
\[
    \log\mu^{(2)}_{r,\hat r,\theta}(x,y)
      =\mathbf a(\theta)\cdot T(x,y)-\Psi(\mathbf a(\theta)),
\]
we have
\begin{equation}\label{eq:entropy-Legendre}
    h(\theta)
      =\Psi(\mathbf a(\theta))
         -\mathbf a(\theta)\cdot m(\theta)
      =\inf_{\mathbf u:\,\Psi(\mathbf u)<\infty}
         \{\Psi(\mathbf u)-\mathbf u\cdot m(\theta)\}.
\end{equation}
The infimum is uniquely attained at $\mathbf u=\mathbf a(\theta)$ since
the gradient of the expression being minimized is
$\nabla\Psi(\mathbf u)-m(\theta)$ and its Hessian is positive definite. Differentiating the attained infimum in \eqref{eq:entropy-Legendre} and using
$\nabla\Psi(\mathbf a)=m(\theta)$, we obtain
\[
\begin{aligned}
    h'(\theta)
      &=\nabla\Psi(\mathbf a)\cdot\mathbf a'
        -\mathbf a'\cdot m(\theta)
        -\mathbf a\cdot m'(\theta) =-\mathbf a\cdot m'(\theta)
       =-b(\theta)r\hat r,
\end{aligned}
\]
which proves \eqref{eq:H-prime}. Differentiating
$m(\theta)=\nabla\Psi(\mathbf a(\theta))$ gives
\[
    m'(\theta)=\Gamma_{r,\hat r,\theta}\mathbf a'(\theta).
\]
Since $m'(\theta)=(0,r\hat r,0)$ is constant,
\[
    h''(\theta)
      =-\mathbf a'(\theta)\cdot m'(\theta)
      =-m'(\theta)^\top
         \Gamma_{r,\hat r,\theta}^{-1}m'(\theta)<0.
\]
This proves that the entropy is strictly concave. Moreover, as $\mu_{r,\hat{r},0}^{(2)}= \mu_r\otimes\mu_{\hat r}$, independence yields \eqref{eq:H''}:
\[
    h''(0)
      =-(r\hat r)^2
        \bigl(\Gamma_{r,\hat r,0}^{-1}\bigr)_{22}
      = - (r\hat{r})^2 \, \E{X^2Y^2}^{-1} 
      =-1. \qedhere
\]
\end{proof}
We now compute the second derivative of the other term of the second moment profile function to obtain the following.
\begin{proposition}[Central estimate]\label{prop:F2'' central}
For every $\kappa>0$, $\alpha>0$, and $r,\hat r>1$,
\begin{equation}\label{eq:F2-center}
    \left.\partial_{\theta\theta}^2
      F_2(\kappa,\alpha,r,\hat r,\theta)\right|_{\theta=0}
    =-1+4B_{\kappa/r,\alpha}B_{\kappa/\hat r,\alpha}.
\end{equation}
\end{proposition}

\begin{proof}
    
Let $\varphi_\theta$ denote the density of a pair of standard Gaussian random variables with covariance $\theta$.  Direct differentiation yields
\[
    \partial_\theta\varphi_\theta(x,y) =\partial_x\partial_y\varphi_\theta(x,y).
\]
Integrating this identity over $[-u,u]\times[-v,v]$, and then dividing by $p^{(2)}(u,v,\theta)$, gives
\begin{equation}\label{eq:energy_derivative}
\begin{split}
    \partial_\theta\log p^{(2)}(u,v,\theta)
    =\frac{1}{\pi p^{(2)}(u,v,\theta)\sqrt{1-\theta^2}} \left(\exp\left(-\frac{u^2-2\theta uv+v^2}{2(1-\theta^2)}\right)-\exp\left(-\frac{u^2+2\theta uv+v^2}{2(1-\theta^2)}\right)\right).
\end{split}
\end{equation}
In particular,
\[
    \left.\partial_{\theta\theta}^2
      \log p^{(2)}(u,v,\theta)\right|_{\theta=0}
    =d(u)d(v).
\]
Combining this with \eqref{eq:H''}, the result follows. 
\end{proof}

We now split the proof of \cref{thm:unique-max} into three ranges of parameters. 
Differentiation with respect to $r^2$ gives
\[
    \partial_{r^2}H(\mu_r)=\frac1{2\sigma(r)^2}, \qquad \partial_{r^2}\log p(\kappa/r) =-\frac{d(\kappa/r)}{2r^2}.
\]
For ease of notation, for $r>1$ and $\kappa > 0$ define
\[
    q(r):=\frac{r^2}{\sigma(r)^2} \quad\text{ and }\quad K := \kappa/r.
\]
From the derivative computations, we see that if $r$ is a stationary point of $F_1$, we have
\begin{equation}\label{eq:F1-stationarity-B}
    \alpha=\frac{q(r)}{d(K)},
    \qquad
    4B_{K,\alpha}^2=q(r)d(K).
\end{equation}

For $r>1$, $K>0$, and $0\le\theta\le1$, define
\begin{equation}\label{eq:Dr-relative-entropy}
    \begin{aligned}
        D_r(\theta) :=2H(\mu_r)-H(\mu^{(2)}_{r,r,\theta}) =D\!\left(\mu^{(2)}_{r,r,\theta} \,\middle\|\,\mu_r\otimes\mu_r\right), \qquad
        G_K(\theta) :=\log\frac{p^{(2)}(K,K,\theta)}{p(K)^2}.
    \end{aligned}
\end{equation}
The above equality for $D_r(\theta)$ follows from a direct computation, using that the marginals of $\mu_r \otimes \mu_{\hat r}$ and $\mu^{(2)}_{r, \hat r, \theta}$ have the same second moments. If $r = r_*(\kappa,\alpha)$, then $r$ is a stationary point of $F_1$ and \eqref{eq:F1-stationarity-B} gives
\begin{equation}\label{eq:profile-gap}
    F_2(\kappa,\alpha,r,r,0)-F_2(\kappa,\alpha,r,r,\theta) = D_r(\theta)-\frac{q(r)}{d(K)} G_K(\theta).
\end{equation}
Thus $D_r(\theta)$ is the entropy cost of correlation, while $G_K(\theta)$ is the logarithmic gain in the Gaussian constraint probability, and stationarity removes $\alpha$ from their comparison. The lower bound $\alpha \ge 1+\kappa^{-2}$ will be used only when $K\le 2$.

\subsection{The three regimes}

We begin with collecting some basic Gaussian estimates that will be used throughout. We relegate proofs which depend on unenlightening computations to the appendix. Numerical conventions and setup for these computations can be found in \cref{app:numer-conventions}.

\begin{lemma}[Gaussian overlap estimates]\label{lem:Gaussian-overlap-estimates}
For $K>0$ and $0<\theta<1$,
\begin{align}
    G_K(\theta)
       &\le\frac{1-p(K)}{p(K)}\,\theta^2,
       \label{eq:Gaussian-Hermite-coarse}\\
    G_K(\theta)
       &\le\frac{d(K)^2}{2}\theta^2
       +\left(\frac{1-p(K)}{p(K)}
                    -\frac{d(K)^2}{2}\right)\theta^4.
       \label{eq:Gaussian-Hermite-refined}
\end{align}
Moreover, if $K\le2$, then
\begin{equation}\label{eq:Gaussian-small-window}
    G_K(\theta)
       \le-\frac{d(K)^2}{2}\log(1-\theta^2).
\end{equation}
\end{lemma}
The proof is given in \cref{app:Gaussian-small-window}; it is based on a Hermite expansion, a standard tool for computing Gaussian expectations. 

\begin{lemma}[Entropy interpolation]\label{lem:entropy-interpolation}
If $r\ge1.2$ and $0\le\theta\le1$, let $Z_r:=Z(r)$ denote the
normalizing constant in \cref{def:DG}, and set $c_r:=1.1/Z_r.$
Then $0<c_r<1$ and
\begin{equation}\label{eq:entropy-interpolation}
    D_r(\theta)
       \ge-\log\bigl(c_r+(1-c_r)\sqrt{1-\theta^2}\bigr)
       \ge\frac{1-c_r}{2}\theta^2.
\end{equation}
\end{lemma}
We give the proof in \cref{app:entropy-interpolation}, using the relative entropy variational formula followed by Poisson summation.  Together with \cref{lem:Gaussian-overlap-estimates}, it yields the first of our three regimes:

\begin{lemma}[Bulk profile]\label{lem:bulk-profile}
If
\[
    r\ge1.2,\qquad K\ge2,\qquad rK\ge10,
\]
then $q(r)\,d(K)<1$, and for $0 < \theta < 1$ we have
\[
    D_r(\theta)>\frac{q(r)}{d(K)}G_K(\theta).
\]
\end{lemma}

\begin{proof}
Let $\overline\Phi(K)=\PP{W>K}$, so that $2\overline\Phi(K) = \PP{|W|>K}= 1-p(K)$. Using \eqref{eq:Gaussian-Hermite-coarse}, the definition of $d(K)$, and the fact that
$\overline\Phi(K)\le\phi(K)/K$, we obtain
\[
    \frac{q(r)}{d(K)}G_K(\theta)
      \le q(r)\frac{\overline\Phi(K)}{K\phi(K)}\theta^2
      \le\frac{q(r)}{K^2}\theta^2.
\]
For the parameter range assumed in the lemma, we verify the scalar estimate
\begin{equation}\label{eq:bulk-scalar-estimate}
    \frac{q(r)}{K^2}
       <\frac{1-1.1/Z_r}{2}
\end{equation}
in
\cref{app:finite-profile-estimates}.  The inequality $D_r(\theta) > \frac{q(r)}{d(K)} G_K(\theta)$ then follows from
\eqref{eq:entropy-interpolation}. Furthermore, \eqref{eq:F2-center}, the positive quadratic gap, and \eqref{eq:F1-stationarity-B} give $q(r)d(K)<1$.
\end{proof}

The second regime considers small values of $K = \kappa/r_*$. 

\begin{lemma}[Small-window profile]\label{lem:small-window-profile}
Suppose
\[
    r\ge5,\qquad 0<K\le2,\qquad rK\ge10,\qquad
    \frac{q(r)}{d(K)}\ge1+\frac1{r^2K^2}.
\]
Then
\begin{equation}\label{eq:small-window-profile}
    D_r(\theta)>\frac{q(r)}{d(K)}G_K(\theta)
    \qquad(0<\theta<1),
\end{equation}
and $q(r)d(K)<1$.
\end{lemma}

The proof is given in \cref{app:small-window-profile}. Briefly, since $K\leq 2$, \eqref{eq:Gaussian-small-window} yields 
\[
    \frac{q(r)}{d(K)}G_K(\theta)\le-\frac{q(r)d(K)}{2}\log(1-\theta^2),
\]
while Poisson summation gives the reverse strict bound $D_r(\theta)>-q(r)d(K)/2 \cdot \log(1-\theta^2)$ up to an explicit value of $\theta$. An easy argument using the monotonicity of $D_r$ covers the remaining values.

The final remaining regime is $r_* \approx 1$, which is where existing estimates on the SBP can be reused.
\begin{lemma}[Near-binary profile]\label{lem:near-binary-profile}
Suppose
\[
    1<r\le1.2,\qquad rK\ge10,\qquad
    H(\mu_r)+\frac{q(r)}{d(K)}\log p(K)\ge0.
\]
Then
\[
    D_r(\theta)>\frac{q(r)}{d(K)}G_K(\theta)
    \qquad(0<\theta<1),
\]
and $q(r)d(K)<1$.
\end{lemma}

The proof is given
in \cref{app:near-binary-profile}.
Roughly, when $q(r)\le8$, the proof follows from convexity estimates.
When $q(r)>8$, $r$ must be very close to $1$: the series defining $\mu_r$ gives
$r^2-1<7.3\cdot10^{-4}$.  Projecting a correlated pair to its signs
then reduces the entropy comparison to the binary
profile inequality of Abbe, Li, and Sly~\cite{ALS22}, up to an explicit error. We need separate computations to handle the interval within $1/(3K^2)$ of the endpoint $\theta =1$.

\begin{proof}[Proof of \cref{thm:unique-max}]
Since $\kappa\geq 10$, we have $p(\kappa)>.99$, so
\[
    F_1(\kappa,1+\kappa^{-2},1)
      =\log2+(1+\kappa^{-2})\log p(\kappa)>0,
\]
and hence $\alpha_c(\kappa)>1+\kappa^{-2}$. Fix $\alpha$ in the range \eqref{eq:profile-main-range}.
Let $r=r_*(\kappa,\alpha)$ and $K=\kappa/r$. By \eqref{eq:F1-stationarity-B} and
\eqref{eq:F1-capacity-sign}, respectively,
\[
    \alpha=\frac{q(r)}{d(K)},
    \qquad
    F_1(\kappa,\alpha,r)\ge0.
\]
By evenness, it is enough to consider $0<\theta<1$. If $r\ge1.2$ and $K\ge2$, apply
\cref{lem:bulk-profile}. If $K\le2$, then $r\ge5$, and
\eqref{eq:profile-main-range} gives
\[
    \frac{q(r)}{d(K)}=\alpha
       \ge1+\frac1{\kappa^2}
       =1+\frac1{r^2K^2},
\]
so \cref{lem:small-window-profile} applies. The only remaining case is
$1<r<1.2$, which is handled by \cref{lem:near-binary-profile} (noting that the condition $H(\mu_r)+\frac{q(r)}{d(K)}\log p(K) \geq 0$ is equivalent to $F_1(\kappa, \alpha, r) \geq 0$). Thus, in all cases, \eqref{eq:profile-gap} is strictly positive for
$0<\theta<1$, and $1-4B_K^2=1-q(r)d(K)>0$, completing the proof.
\end{proof}

\section{Conditioned cycle counts}\label{sec:cycle counts}

\subsection{Cycle counts and planted ensembles}

Fix $\kappa>0$ and let $m=m(n)$, $\alpha=m/n$, and
$I=I_n(\kappa,\alpha)$. In this section, $\alpha$ is assumed to
remain in a fixed compact subset of $(1,\alpha_c(\kappa)]$ on which
\cref{hyp:maxima} holds. In particular, we will apply the results with
$m=\lfloor\alpha_c(\kappa)n-\eta\log n\rfloor$, where $\eta>0$ is fixed. For the remainder of this section, it will be convenient to work with the normalization
\[
    G:=\sqrt n A.
\]
That is, $G$ has iid standard Gaussian entries and, if
$\|x\|_2=r\sqrt n$, then
\begin{equation}\label{eq:rescaled planting event}
    x\in\sol_{n,m,\kappa}
    \quad\Longleftrightarrow\quad
    |\la G_j,\hat x\ra|\le \kappa/r
    \quad\text{for every }j\in[m],
\end{equation}
where $\hat x=x/\|x\|_2$. For later convenience, set
\[
    \psi_a^2:=\E{W^2\,\big|\,|W|\le a},
    \qquad\text{so that}\qquad
    2B_{a,\alpha}=\sqrt\alpha(\psi_a^2-1).
\]
By \cref{hyp:maxima} and continuity, there is a
constant $q<1$ such that, for all sufficiently large $n$,
\begin{equation}\label{eq:B contraction}
    |2B_{\kappa/r,\alpha}|\le q
    \qquad\text{for every }r\in I.
\end{equation}

The aim of this section is to control weighted cycle counts up to
$T=\lfloor\lambda\log n\rfloor$, where $\lambda>0$ is a sufficiently
small constant. The cycle-count correction was introduced for the
symmetric binary perceptron by Sah and Sawhney \cite{ss}. Here
Gaussianity allows us to replace their coordinate-wise planted
calculation by the low-rank resampling argument below. This comparison
uses the $\ell_\infty$ bound on typical solutions to control the
coherence of the planted directions; the only input we import directly from \cite{ss} is the limiting distribution of the cycle counts under the ``null'' (no planted solutions) distribution for $G$. 

\begin{definition}[Cycle counts]
For $k\ge1$, define the $2k$-cycle count of the first $m$ rows of a
matrix $M$ by
\[
    C_k(m,M):=(nm)^{-k/2}
    \sum_{\substack{i_1,\ldots,i_k\in[n]\ \mathrm{distinct}\\
                    j_1,\ldots,j_k\in[m]\ \mathrm{distinct}}}
    \prod_{\ell=1}^k M_{j_\ell,i_\ell}M_{j_\ell,i_{\ell+1}}
    -\ind_{\{k=1\}}\sqrt{mn},
\]
where $i_{k+1}:=i_1$. Also define
\[
    Y_T(m,M,a):=
    \sum_{k=1}^T
    \frac{2(2B_{a,m/n})^k C_k(m,M)-(2B_{a,m/n})^{2k}}{4k}.
\]
\end{definition}

We now turn to the construction of the null and planted distributions on $G$. Set
\[
    \Lambda_n^{\typ}(r,\hat r,\theta)
    :=\cb{(x,y)\in\Lambda_n^{\typ}(r)\times\Lambda_n^{\typ}(\hat r):
          \la x,y\ra=\theta r\hat r n}.
\]
Fix $x\in\Lambda_n^{\typ}(r)$ and
$(x,y)\in\Lambda_n^{\typ}(r,\hat r,\theta)$ with $|\theta|<1$. Let
\[
    \hat y:=\frac{y}{\|y\|_2},
    \qquad
    \hat z:=\frac{\hat y-\theta\hat x}{\sqrt{1-\theta^2}},
    \qquad
    U:=\begin{pmatrix}\hat x&\hat z\end{pmatrix}.
\]
Then $(\hat x,\hat z)$ is an orthonormal basis of
$\operatorname{span}(x,y)$ and
\begin{equation}\label{eq:y in xz basis}
    \hat y=\theta\hat x+\sqrt{1-\theta^2}\,\hat z.
\end{equation}

For each $j\in[m]$, define
\begin{align*}
    \tau_j^{(0)}&:=\la G_j,\hat x\ra\hat x,
    &G_j^{(x)}&:=G_j-\tau_j^{(0)},\\
    \tau_j^{(0,0)}&:=UU^\top G_j,
    &G_j^{(x,y)}&:=(I-UU^\top)G_j.
\end{align*}
Let $\xi_j^{(x)}$ be iid copies of a standard Gaussian conditioned on
$|\xi_j^{(x)}|\le\kappa/r$, and let
$\xi_j^{(x,y)}=(\xi_{j,1}^{(x,y)},\xi_{j,2}^{(x,y)})$ be iid copies of
$(W_1,W_2)\sim N(0,I_2)$ conditioned on
\begin{equation}\label{eq:pair conditioning event}
    |W_1|\le\kappa/r,
    \qquad
    |\theta W_1+\sqrt{1-\theta^2}\,W_2|\le\kappa/\hat r.
\end{equation}
Take these variables independent of the Gaussian projections above and
set
\[
    \tau_j^{(x)}:=\xi_j^{(x)}\hat x,
    \qquad
    \tau_j^{(x,y)}:=U\xi_j^{(x,y)}.
\]
All variables are sampled independently across rows. By orthogonal
invariance and \eqref{eq:rescaled planting event},
\begin{align}
    \calL(G)&=\calL(G^{(x)}+\tau^{(0)})
             =\calL(G^{(x,y)}+\tau^{(0,0)}),\label{eq:null couplings}\\
    \calL(G\mid x\in\sol)&=\calL(G^{(x)}+\tau^{(x)}),
        \label{eq:one planted coupling}\\
    \calL(G\mid x,y\in\sol)&=\calL(G^{(x,y)}+\tau^{(x,y)}).
        \label{eq:two planted coupling}
\end{align}

\subsection{Conditional cycle-count limits}\label{sec:conditional cycle count limits}

For $g\in\mathcal C^3(\RR^T)$ and $s\in[3]$, write
\[
    M_s(g):=\sup_{u\in\RR^T}\|D^s g(u)\|_{\mathrm{op}}.
\]
For a matrix $M$, define
\[
    V(M):=\pa{\frac{C_k(m,M)}{\sqrt{2k}}}_{k\in[T]},
\]
and, for $r\in I$, set
\[
    \rho_r:=
    \pa{\frac{(2B_{\kappa/r,\alpha})^k}{\sqrt{2k}}}_{k\in[T]}.
\]
We aim to track the distribution of $V(M)$ and show that the shifts under one and two planted solutions are $\rho_r$ and
$\rho_r+\rho_{\hat r}$, respectively. The main result of this section is the following set of quantitative limit laws:

\begin{lemma}[Distribution limits of shifted cycle counts]
\label{lemma:joint normal limits}
There is a sufficiently small $\lambda>0$ such that the following holds.
Let $T=\lfloor\lambda\log n\rfloor$, let $r,\hat r\in I$, and suppose
$|\theta|\le n^{-1/2}\log n$. Uniformly over
$x\in\Lambda_n^{\typ}(r)$ and
$(x,y)\in\Lambda_n^{\typ}(r,\hat r,\theta)$, for every
$g\in\mathcal C^3(\RR^T)$ and $W_T\sim N(0,I_T)$,
\begin{align}
    \ba{\E{g(V(G))}-\E{g(W_T)}}
    &\le 3n^{-1/5}\sum_{s=1}^3M_s(g),
        \label{eq:joint normal null}\\
    \ba{\E{g(V(G)-\rho_r)\mid x\in\sol}-\E{g(W_T)}}
    &\le 3n^{-1/5}\sum_{s=1}^3M_s(g),
        \label{eq:joint normal x}\\
    \ba{\E{g(V(G)-\rho_r-\rho_{\hat r})\mid x,y\in\sol}-\E{g(W_T)}}
    &\le 3n^{-1/5}\sum_{s=1}^3M_s(g).
        \label{eq:joint normal xy}
\end{align}
\end{lemma}

The effect of conditioning is isolated in the following proposition.

\begin{proposition}[Conditional cycle counts]
\label{prop:cycle-shifts}
After decreasing $\lambda>0$ if necessary, uniformly for the parameters
in \cref{lemma:joint normal limits} and $1\le k\le T$,
\begin{align}
    \E{C_k(m,G^{(x)}+\tau^{(x)})-C_k(m,G)}
    &=(2B_{\kappa/r,\alpha})^k+\widetilde O(n^{-1}),
        \label{eq:cycles-x-mean}\\
    \mathrm{Var}\pa{C_k(m,G^{(x)}+\tau^{(x)})-C_k(m,G)}
    &\le n^{-3/4}\mathrm{Var}\pa{C_k(m,G)},
        \label{eq:cycles-x-var}\\
    \E{C_k(m,G^{(x,y)}+\tau^{(x,y)})-C_k(m,G)}
    &=(2B_{\kappa/r,\alpha})^k+(2B_{\kappa/\hat r,\alpha})^k
      +\widetilde O(n^{-1}),
        \label{eq:cycles-xy-mean}\\
    \mathrm{Var}\pa{C_k(m,G^{(x,y)}+\tau^{(x,y)})-C_k(m,G)}
    &\le n^{-3/4}\mathrm{Var}\pa{C_k(m,G)}.
        \label{eq:cycles-xy-var}
\end{align}
Moreover, for $k \ge 2$,
\begin{equation}\label{eq:null cycle variance}
    \mathrm{Var}(C_1(m,G))=2,
    \qquad
    \mathrm{Var}(C_k(m,G))
    =2k\frac{(m)_k(n)_k}{m^kn^k}
    =2k\pa{1+O\pa{\frac{k^2}{m}+\frac{k^2}{n}}}.
\end{equation}
\end{proposition}

Assuming this proposition, the main lemma is readily deduced. 

\begin{proof}[Proof of \cref{lemma:joint normal limits}]
The null estimate \eqref{eq:joint normal null} is
given by the proof of \cite[Proposition~3.7]{ss}. That proof applies
verbatim when $m/n$ ranges over a fixed compact subset of
$(0,\infty)$ and $T\le\lambda\log n$, provided $\lambda>0$ is
sufficiently small.

For the one-planted estimate, couple the null matrix as
$G=G^{(x)}+\tau^{(0)}$ and define
\[
    \Delta^{(x)}:=V(G^{(x)}+\tau^{(x)})-\rho_r-V(G).
\]
By \cref{prop:cycle-shifts} and \eqref{eq:null cycle variance},
\[
    \E{\|\Delta^{(x)}\|_2^2}
    =\left\|\E{\Delta^{(x)}}\right\|_2^2
      +\sum_{k=1}^T\Var(\Delta_k^{(x)})
    \le n^{-2/3}
\]
for all sufficiently large $n$. Consequently,
\[
    \ba{\E{g(V(G^{(x)}+\tau^{(x)})-\rho_r)}-\E{g(V(G))}}
    \le M_1(g)\E{\|\Delta^{(x)}\|_2}
    \le M_1(g)n^{-1/3}.
\]
Combining this with \eqref{eq:joint normal null} and
\eqref{eq:one planted coupling} proves \eqref{eq:joint normal x}. The
pair-planted estimate is proven identically, using
\eqref{eq:two planted coupling} in place of \eqref{eq:one planted coupling}.
\end{proof}

\subsection{Low-rank resampling}
Towards developing the proof of \cref{prop:cycle-shifts}, we isolate the comparison estimate used for both planted ensembles.
For an $n\times d$ matrix $U$ with orthonormal columns, define its
coherence by
\[
    \mu(U):=\max_{i\in[n]}\|U^\top e_i\|_2^2.
\]
In our applications, $d=1$ and $U=\hat x$, or $d=2$ and
$U=(\hat x,\hat z)$; thus, the columns of $U$ span exactly the
directions resampled when one or two solutions are planted. The vectors
$G_j^\perp$ below are the projections of the Gaussian rows onto the
orthogonal complement of this span. Adding $U\xi_j^{(0)}$ restores an
iid standard Gaussian row, while replacing $\xi_j^{(0)}$ by
$\xi_j^{(1)}$ gives the planted row. 

\begin{lemma}[Cycle counts under low-rank resampling]
\label{lem:low-rank-cycle-resampling}
Fix $d\le2$. Let $U\in\RR^{n\times d}$ satisfy $U^\top U=I_d$ and
$\mu(U)\le n^{-1/2}$. Let $G_j^\perp$, $j\in[m]$, be iid
$N(0,I_n-UU^\top)$ vectors. Independently, let
$\xi_j^{(0)}\sim N(0,I_d)$ and let $\xi_j^{(1)}$ be iid centered random
vectors such that, for some $K<\infty$,
\[
    \sup_{\|v\|_2=1}\E{|\la v,\xi_j^{(1)}\ra|^4}\le K.
\]
Set
\[
    M_j^{(b)}:=G_j^\perp+U\xi_j^{(b)},
    \qquad
    \Sigma:=\E{\xi_j^{(1)}(\xi_j^{(1)})^\top},
    \qquad
    D:=\sqrt{m/n}\,(\Sigma-I_d).
\]
There is a constant $C$, depending only on $K$ and the compact range of
$m/n$, such that, uniformly for
$1\le k\le\lambda\log n$,
\begin{align}
    \E{C_k(m,M^{(1)})-C_k(m,M^{(0)})}
    &=\tr(D^k)
      +O\pa{k^2\pa{\mu(U)+m^{-1}}\|D\|_{\mathrm{op}}^k},
        \label{eq:low rank mean}\\
    \mathrm{Var}\pa{C_k(m,M^{(1)})-C_k(m,M^{(0)})}
    &\le C^k k^C\pa{\mu(U)+m^{-1}}
      \mathrm{Var}(C_k(m,M^{(0)})).
        \label{eq:low rank variance}
\end{align}
\end{lemma}

\begin{proof}
Write $u_i:=U^\top e_i$ and $Q:=\Sigma-I_d$. Since
$M^{(0)}$ is an iid standard Gaussian matrix, its cycle counts are
centered. For $k=1$, direct calculation gives
\[
    \E{C_1(m,M^{(1)})-C_1(m,M^{(0)})}
    =\sqrt{m/n}\,\tr Q=\tr D.
\]
For $k\ge2$, the row indices in a cycle are distinct, and the two column
indices incident to each row are distinct. Row independence therefore
gives
\begin{align*}
    \E{C_k(m,M^{(1)})}
    &=\frac{(m)_k}{(mn)^{k/2}}
      \sum_{i_1,\ldots,i_k\ \mathrm{distinct}}
      \prod_{\ell=1}^k u_{i_\ell}^\top Q u_{i_{\ell+1}}.
\end{align*}
Without the distinctness restriction, the sum is $\tr(Q^k)$. To see
the cost of the restriction, fix $1\le a<b\le k$ and impose
$i_a=i_b$. Summing successively over the remaining column indices along
the two arcs from $a$ to $b$ gives
\[
    \sum_{i=1}^n
       (u_i^\top Q^{b-a}u_i)
       (u_i^\top Q^{k-b+a}u_i).
\]
The absolute value of this expression is at most
\[
    \|Q\|_{\mathrm{op}}^k\sum_{i=1}^n\|u_i\|_2^4
    \le d\,\mu(U)\|Q\|_{\mathrm{op}}^k.
\]
There are $O(k^2)$ possible equalities among the column indices.
Excluding them and using $(m)_k/m^k=1+O(k^2/m)$ proves
\eqref{eq:low rank mean}.

We next prove \eqref{eq:low rank variance}. Suppose first that $k\ge2$.
For distinct tuples $\bi=(i_1,\ldots,i_k)$ and
$\bj=(j_1,\ldots,j_k)$, write
\[
    P_b(\bi,\bj)
    :=\prod_{\ell=1}^k
       M^{(b)}_{j_\ell,i_\ell}M^{(b)}_{j_\ell,i_{\ell+1}},
    \qquad i_{k+1}:=i_1,
\]
and let
\[
    \overline P(\bi,\bj)
    :=P_1(\bi,\bj)-P_0(\bi,\bj)
      -\E{P_1(\bi,\bj)-P_0(\bi,\bj)}.
\]
Then, writing $Z_k:=C_k(m,M^{(1)})-C_k(m,M^{(0)})$,
\begin{equation}\label{eq:low rank variance expansion}
    \Var(Z_k)
    =\frac1{(mn)^k}
      \sum_{\bi,\bj,\bi',\bj'}
      \E{\overline P(\bi,\bj)\overline P(\bi',\bj')},
\end{equation}
where each of the four tuples is restricted to have distinct entries.

We estimate this sum by expanding every matrix entry as
\[
    M^{(b)}_{j,i}=G^\perp_{j,i}+u_i^\top\xi_j^{(b)}.
\]
The term using $G^\perp$ at all $2k$ edges is identical for $b=0$ and
$b=1$, so it cancels from $P_1-P_0$. Since the row indices within each
cycle are distinct, a fixed row occurs at most twice in the product of
the two cycle monomials and hence contributes at most four random
factors. The fourth-moment assumption therefore bounds every moment of
the $\xi$ variables which occurs here by a constant depending only on
$K$.

Only the following elementary Gaussian formulas are needed. If
$Z\sim N(0,P)$, then
\begin{align}
    \E{Z_aZ_b}&=P_{ab},\label{eq:Gaussian second moment}\\
    \E{Z_aZ_bZ_cZ_d}
      &=P_{ab}P_{cd}+P_{ac}P_{bd}+P_{ad}P_{bc},
        \label{eq:Gaussian fourth moment}
\end{align}
and a product of an odd number of coordinates has mean zero. In the
present case, independence across rows and
$P=I_n-UU^\top$ give
\begin{equation}\label{eq:projected Gaussian covariance}
    \E{G^\perp_{j,i}G^\perp_{j',i'}}
    =\ind_{\{j=j'\}}\pa{\ind_{\{i=i'\}}-u_i^\top u_{i'}}.
\end{equation}
Thus \eqref{eq:Gaussian second moment}--\eqref{eq:Gaussian fourth moment},
followed by expanding each covariance in
\eqref{eq:projected Gaussian covariance}, writes every summand in
\eqref{eq:low rank variance expansion} as a sum of at most $C^k$ terms.
Each such term consists of bounded coefficients, factors involving the
vectors $u_i$, and explicit requirements that certain row indices and
column indices in the two cycles be equal.

We now bound one of these terms. Say that an edge of the first cycle is
matched to an edge of the second if the first term in
\eqref{eq:projected Gaussian covariance} is used for their two
$G^\perp$ factors. The indicator in that formula then requires both the
row and the column index of the two edges to be equal. Such a term
cannot match two $G^\perp$ factors within one cycle: their row indices
are either different, or they are the two edges incident to the same
row and their column indices are different. Hence all matched edges are
shared by the two cycles.

Unless all $2k$ edges are matched, the matched edges form disjoint
paths. Indeed, every vertex has degree two in either cycle, so a
connected set of shared edges is either a path or the entire cycle; the
latter is exactly the all-$G^\perp$ term already cancelled above. Let
$s$ be the number of these paths. A path of $r$ edges has $r+1$
vertices. Once the labels in the first cycle are chosen, the equality
requirements prescribe the corresponding $r+1$ alternating row and
column labels in the second cycle.

We now sum all terms having fixed locations and orientations for these
paths. Let $r$ be the number of matched edges, and let $a$ and $b$ be
the numbers of row vertices and column vertices, respectively, lying
in the paths. Since a path has one more vertex than edge,
\begin{equation}\label{eq:path vertices}
    a+b=r+s.
\end{equation}
The row labels have at most $m^{2k-a}$ possible values: without the
paths the two cycles have $2k$ row labels, while the value of each of
the $a$ rows in the second cycle is prescribed by the first.

It remains to sum the column labels. Every occurrence of a matrix entry
which is not part of a matched pair carries one factor $u_i$. After the
$b$ shared column vertices are merged, there are $2k-b$ column
positions. Exactly $r-b$ of them are internal column vertices of the
paths, incident only to matched edges, and hence carry no $u_i$ factor;
every other column position carries two. A position carrying no such
factor has at most $n$ possible labels. A position carrying two is
summed using
\begin{equation}\label{eq:orthonormal row sum}
    \sum_{i=1}^n u_i u_i^\top=I_d.
\end{equation}
The requirement that the columns within either cycle be distinct does
not enlarge this bound: after some labels have been fixed, the relevant
sum has the form
\[
    \sum_{i\notin S}u_i u_i^\top
    =I_d-\sum_{i\in S}u_i u_i^\top,
\]
whose operator norm is at most one. Since $d\le2$ and every row
contributes a moment of at most four bounded random variables, the
remaining finite-dimensional sums contribute at most $C^k$. The full
column sum is therefore at most $C^k n^{r-b}$.

Including the normalization in
\eqref{eq:low rank variance expansion}, the contribution of the fixed
paths is at most
\begin{align}
    \frac{C^k m^{2k-a}n^{r-b}}{(mn)^k}
    &=C^k\pa{\frac mn}^{k-a}n^{r-a-b}
      \le C^k n^{-s},\label{eq:path contribution bound}
\end{align}
where we used \eqref{eq:path vertices} and the compact range of $m/n$.
This is the entire counting step: the exponent $s$ appears because each
path has one more vertex than edge.

There are at most $C^k$ choices made when expanding the matrix factors
and the Gaussian moments, and at most $k^{4s}$ choices for the endpoints
and orientations of $s$ paths. The total contribution from terms with
$s\ge1$ is consequently at most
\[
    C^k\sum_{s\ge1}
       \pa{\frac{k^4}{n}}^s.
\]
If $s=0$, the two centered monomials in
\eqref{eq:low rank variance expansion} are independent unless their row
tuples share a value. For each of the $k^2$ possible positions of such
a shared value, choosing the first row tuple leaves at most $m^{k-1}$
choices for the second. The same column summation as above bounds these
terms by
\[
    \frac{C^k k^2m^{2k-1}}{(mn)^k}
    =\frac{C^k k^2}{m}\pa{\frac mn}^k.
\]
Since $m/n$ remains in a compact subset of $(0,\infty)$ and
$k=O(\log n)$, we obtain the stronger bound
\begin{equation}\label{eq:absolute low rank variance}
    \Var(Z_k)\le \frac{C^k k^C}{m}.
\end{equation}
Finally, a direct calculation gives
\[
    \Var(C_k(m,M^{(0)}))
    =2k\frac{(m)_k(n)_k}{m^kn^k}\ge1
\]
for all sufficiently large $n$. After increasing $C$,
\eqref{eq:absolute low rank variance} proves
\eqref{eq:low rank variance} for $k\ge2$.

For $k=1$, the same estimate follows directly from
\[
    C_1(m,M^{(1)})-C_1(m,M^{(0)})
    =(mn)^{-1/2}\sum_{j=1}^m
      \pa{\|\xi_j^{(1)}\|_2^2-\|\xi_j^{(0)}\|_2^2}. \qedhere
\]
\end{proof}

\begin{proof}[Proof of \cref{prop:cycle-shifts}]
The variance identity \eqref{eq:null cycle variance} is immediate for
$k=1$. For $k\ge2$, two Gaussian cycle monomials have nonzero covariance
only when they represent the same unoriented cycle, which has $2k$
choices of starting point and orientation. This gives the first equality
in \eqref{eq:null cycle variance}; the second follows from
$(s)_k/s^k=1+O(k^2/s)$.

For a single planted solution, the scalar $\xi_j^{(x)}$ is centered by
symmetry and has fourth moment bounded by $(\kappa/r)^4$ and hence by $\kappa^4$. Thus \cref{lem:low-rank-cycle-resampling} applies with
$d=1$, $U=\hat x$, and $\Sigma=\psi_{\kappa/r}^2$. Since
\[
    \mu(\hat x)=\|\hat x\|_\infty^2
    \le \frac{\log^2 n}{nr^2},
    \qquad
    D=\sqrt\alpha(\psi_{\kappa/r}^2-1)
      =2B_{\kappa/r,\alpha},
\]
\eqref{eq:B contraction} and \eqref{eq:low rank mean} together give
\eqref{eq:cycles-x-mean}. Moreover, uniformly for $k\le T$,
\[
    C^k k^C\bigl(\mu(\hat x)+m^{-1}\bigr)
    \le n^{-1+\lambda\log C+o(1)}.
\]
After decreasing $\lambda$ so that $\lambda\log C<1/4$, then
\eqref{eq:low rank variance} gives \eqref{eq:cycles-x-var}.

For a planted pair, let
\[
    \Sigma_\theta
    :=\E{\xi_j^{(x,y)}(\xi_j^{(x,y)})^\top},
    \qquad
    D_\theta:=\sqrt\alpha(\Sigma_\theta-I_2).
\]
The full row covariance is
\begin{equation}\label{eq:pair planted covariance}
    \mathrm{Cov}(G_j^{(x,y)}+\tau_j^{(x,y)})
    =I_n+U(\Sigma_\theta-I_2)U^\top.
\end{equation}
This identity includes the mixed $\hat x\hat z^\top$ terms. At
$\theta=0$, the two conditioning events are independent, and hence
\begin{equation}\label{eq:Sigma zero}
    \Sigma_0=\begin{pmatrix}
        \psi_{\kappa/r}^2&0\\
        0&\psi_{\kappa/\hat r}^2
    \end{pmatrix},
    \qquad
    D_0=\begin{pmatrix}
        2B_{\kappa/r,\alpha}&0\\
        0&2B_{\kappa/\hat r,\alpha}
    \end{pmatrix}.
\end{equation}
The map $\theta\mapsto\Sigma_\theta$ is analytic near zero, uniformly
for $r,\hat r\in I$. Let $J=\operatorname{diag}(1,-1)$. The change of
variables $W_2\mapsto-W_2$ in \eqref{eq:pair conditioning event} gives
\[
    \Sigma_{-\theta}=J\Sigma_\theta J,
    \qquad D_{-\theta}=JD_\theta J.
\]
Hence $\tr(D_\theta^k)$ is even in $\theta$. By
\eqref{eq:B contraction}, $\|D_\theta\|_{\mathrm{op}}$ is uniformly
bounded above by a constant smaller than one for the present range of
$\theta$. Taylor's theorem therefore gives, uniformly for $k\le T$,
\begin{equation}\label{eq:pair trace expansion}
    \tr(D_\theta^k)
    =(2B_{\kappa/r,\alpha})^k
     +(2B_{\kappa/\hat r,\alpha})^k+O(k^2\theta^2).
\end{equation}

Finally, \eqref{eq:y in xz basis} and the definition of the typical set
give
\[
    \mu(U)=\max_i(\hat x_i^2+\hat z_i^2)
    \le C\frac{\log^2 n}{n}.
\]
The conditional vector in \eqref{eq:pair conditioning event} has
uniformly bounded fourth moments. Applying
\cref{lem:low-rank-cycle-resampling}, then using
$|\theta|\le n^{-1/2}\log n$ and \eqref{eq:pair trace expansion}, proves
\eqref{eq:cycles-xy-mean}. After decreasing $\lambda$ once more, the
variance estimate \eqref{eq:low rank variance} proves
\eqref{eq:cycles-xy-var}.
\end{proof}

\section{Conditioned second moment method}\label{sec:second moment method}

We now combine the profile estimates of \cref{sec:2profile} with the
conditional cycle-count limits of \cref{sec:cycle counts}. Fix
$\kappa>0$ and a compact interval $J\subset(1,\alpha_c(\kappa)]$ such
that \cref{hyp:maxima} holds for every $\alpha\in J$. Let $\eta>0$ be
fixed, let $m=m(n)$, and set $\alpha=m/n$. Throughout this section we
assume that
\[
    \alpha\in J,
    \qquad
    m\le \alpha_c(\kappa)n-\eta\log n.
\]
All constants below are uniform in this range. As in
\cref{sec:cycle counts}, let $T=\lfloor\lambda\log n\rfloor$, where
$\lambda>0$ is sufficiently small. The choice of $\lambda$ depends only
on the compact parameter range in the cycle-count estimates and is
independent of $\eta$. Set
\[
    L:=\sqrt{\log n}
\]
and define
\begin{equation}\label{eq:truncated cycle weight}
    \mathcal W_r(M)
    :=e^{-Y_T(m,M,\kappa/r)}
      \ind_{\{Y_T(m,M,\kappa/r)\ge-L\}},
    \qquad
    Q:=\sum_{r\in I}\calZ_r^{\typ}\mathcal W_r(G).
\end{equation}
The indicator enforces $0\le\mathcal W_r\le e^L$; we show below that the indicator equals one with high probability. We now present the main result of this section.

\begin{lemma}[Second moment and solution-space geometry]
\label{lemma:conditional 2mm}
The following hold as $n$ diverges.
\begin{enumerate}
    \item (Conditional second moment.)
    \begin{equation}\label{eq:conditional 2mm - goal}
        \E{Q^2}=(1+o(1))\E{Q}^2,
        \qquad
        \E{Q}=(1+o(1))\E{\calZ_I^{\typ}}.
    \end{equation}

    \item (Log-normal number of solutions.) Let
    \begin{equation}\label{eq:limiting cycle Gaussian}
        W_\alpha\sim\mathcal N\left(
        \frac14\log\bigl(1-4B_{\kappa/r_*,\alpha}^2\bigr),
        -\frac12\log\bigl(1-4B_{\kappa/r_*,\alpha}^2\bigr)
        \right),
    \end{equation}
    where $r_*=r_*(\kappa,\alpha)$. Then
    \begin{equation}\label{eq:conditional 2mm - scaling limit goal}
        \sup_{t\in\RR}
        \ba{
        \PP{\calZ_I^{\typ}>e^t\E{\calZ_I^{\typ}}}
        -\PP{W_\alpha>t}}
        =o(1).
    \end{equation}
    In particular,
    \begin{equation}\label{eq:conditional 2mm - tail goal}
        \PP{
        \frac{\calZ_I^{\typ}}{\E{\calZ_I^{\typ}}}
        \notin\bigl((\log n)^{-1/2},(\log n)^{1/2}\bigr)}
        =o(1).
    \end{equation}

    \item (Strong freezing.) There is a constant $\delta>0$ such
    that, with high probability,
    \begin{equation}\label{eq:strong freezing conclusion}
        \frac{1}{\calZ_I^{\typ}}
        \ba{\cb{x\in\sol_{n,m,\kappa}^{\typ}:
        (\sol_{n,m,\kappa}^{\typ}\setminus\{x\})
        \cap B_{\|\cdot\|_2}(x,\delta\sqrt n)\ne\emptyset}}
        \le e^{-n^{.48}}.
    \end{equation}
    Here the left-hand side is defined to be zero when
    $\calZ_I^{\typ}=0$.

    \item (Complete freezing.) If, in addition,
    $\alpha_c(\kappa)\in J$ and
    $m=\lfloor\alpha_c(\kappa)n-\eta\log n\rfloor$, then, with
    probability at least $1-n^{-1}$, every
    $x,y\in\sol_{n,m,\kappa}^{\typ}$ with $y\notin\{x,-x\}$ satisfy
    \begin{equation}\label{eq:complete freezing conclusion}
        |\la x,y\ra|\le\sqrt n\log n.
    \end{equation}
\end{enumerate}
\end{lemma}

As a first consequence, we obtain a logarithmic bound on the 
critical window for the satisfiability transition. This will later be sharpened to a constant-order bound
.

\begin{corollary}[Sharp threshold]\label{cor:sharp threshold log window}
Fix $\kappa>0$, and let $J\subset(1,\alpha_c(\kappa)]$ be a compact
interval on which \cref{hyp:maxima} holds. For every fixed $\eta>0$, if
$\alpha=m/n\in J$ and
$m\le\alpha_c(\kappa)n-\eta\log n$, then
$\sol_{n,m,\kappa}$ is nonempty with high probability. For any integer
sequence satisfying
$m\ge\alpha_c(\kappa)n+\omega(1)$,
$\sol_{n,m,\kappa}$ is empty with high probability.
\end{corollary}

\begin{proof}
By \eqref{eq:conditional 2mm - goal} and Paley--Zygmund,
\[
    \PP{Q>0}\ge\frac{\E{Q}^2}{\E{Q^2}}=1-o(1).
\]
Since $Q>0$ implies $\calZ_I^{\typ}>0$, this proves the first claim.
For the second claim, the first-moment estimates of
\cref{sec:1profile} give $\E{\calZ}=o(1)$, and Markov's inequality
finishes the proof.
\end{proof}

The rest of this section is dedicated to the proof of \cref{lemma:conditional 2mm}. We now isolate and prove some of the main technical components. 

\subsection{Endpoint estimate}
The proof of \cref{lemma:conditional 2mm} requires an estimate of $F_2(\theta)$ near $|\theta|=1$. Roughly, we seek to quantitatively bound from below the singularity of $|F_2'(\theta)|$ as $\theta \to \pm 1$, which will then be used to show that the contribution of $\theta$ near $\pm 1$ is negligible in the saddle-point computation for $\E{\calZ^2}$. 

\begin{lemma}[Endpoint profile estimate]
\label{lemma:F2-endpoint-uniform}
There are constants $c_0,c>0$ such that the following holds uniformly
for $\alpha\in J$ and $r,\hat r\in I$. If
\[
    0<\eps:=1-|\theta|<c_0,
    \qquad
    (r-\hat r)^2\le c_0\eps\log(1/\eps),
\]
then
\begin{equation}\label{eq:uniform endpoint profile estimate}
    F_2(\kappa,\alpha,r,\hat r,\theta)
    \le
    \min\{F_1(\kappa,\alpha,r),F_1(\kappa,\alpha,\hat r)\}
    -c\eps^{.51}.
\end{equation}
\end{lemma}

\begin{proof}
As $F_2$ is even in $\theta$, it suffices to take
$\theta=-1+\eps$, so that $\theta \in (-1,0)$. Let $a=\kappa/r$ and $b=\kappa/\hat r$. Note that $p^{(2)}(a,b,\theta)$ is bounded away from zero for all $\theta$. Hence, applying the
derivative formula \eqref{eq:energy_derivative}, it holds uniformly for
$t\in[\eps/2,\eps]$,
\begin{align*}
    \partial_\theta\log p^{(2)}(a,b,-1+t) 
    \le-c \,t^{-1/2} \,\exp\left\{-\frac{(a-b)^2}{4t}\right\}\,.
\end{align*}
Recalling the assumption on $(r-\hat r)^2$, the exponent above is at most of order $-c_0\log(1/\eps)$. Hence, taking $c_0$ sufficiently small,
\begin{equation}\label{eq:p2'-bound}
    \partial_\theta\log p^{(2)}(a,b,-1+t) \le -t^{-.49}.
\end{equation}
Since
$p^{(2)}(a,b,-1+\eps/2)\le\min\{p(a),p(b)\}$,
integrating \eqref{eq:p2'-bound} from $\eps/2$ to $\eps$
gives:
\begin{equation}\label{eq:uniform endpoint probability estimate}
   \log p^{(2)}(a,b,-1+\eps) \le \log\min\cb{p(a),\,p(b)} -c\,\eps^{.51}.
\end{equation}

It remains to control the entropy. Let $(X,Y)$ have law
$\mu_{r,\hat r,-1+\eps}^{(2)}$, and set
\[
    \delta:=\PP{X+Y\ne0},
    \qquad
    v:=\Var(X+Y)=(r-\hat r)^2+2r\hat r\eps.
\]
Since $X+Y$ is integer-valued, $\delta\le v$. Moreover, by the fact that the Gaussian maximizes entropy among real-valued random variables with fixed variance, it holds for any random variable $Z$,
\[
    H(Z)\le\frac12\log\left(2\pi e\,\Var(Z)\right)\,.
\]
This will be used in conjunction the crude upper bound $\Var(X+Y\mid X+Y = 0) \le v/\delta$.
In total, by the chain rule for entropy,
\begin{align*}
    H(X,Y)
    &\le H(X)+H(\ind_{\{X+Y\ne0\}})
       +\delta\, H(X+Y\mid X+Y\ne0)\\
    &\le H(X)+H(\ind_{\{X+Y\ne0\}})
       +\frac\delta2\log\left(2\pi e
          \,\frac v\delta\right).
\end{align*}
Since
$H(X)\le H(\mu_r)$ and $\delta\le v$, the last display is at most
\[
    H(\mu_r)+Cv\log(1/v).
\]
Interchanging $X$ and $Y$ gives the same estimate with
$H(\mu_{\hat r})$ in place of $H(\mu_r)$. Our assumption implies
$v=O(\eps\log(1/\eps))$, and hence
\begin{equation}\label{eq:uniform endpoint entropy estimate}
    H(\mu_{r,\hat r,-1+\eps}^{(2)})
    \le\min\{H(\mu_r),H(\mu_{\hat r})\}
       +O(\eps\log^2(1/\eps)).
\end{equation}
The error in \eqref{eq:uniform endpoint entropy estimate} is
$o(\eps^{.51})$. Moreover,
\[
    \min\{H(\mu_r),H(\mu_{\hat r})\}
    +\alpha\log\min\{p(a),p(b)\}
    \le\min\{F_1(\kappa,\alpha,r),F_1(\kappa,\alpha,\hat r)\}.
\]
Combining the last two displays with
\eqref{eq:uniform endpoint probability estimate} proves the lemma.
\end{proof}

\subsection{The correction factor and the first moment}\label{subsec:correction} We next compute the moments of the correction factor $\mathcal{W}_r(G)$ under the single and pair planted distributions. 
Recall the vectors $\rho_r$ and $V(M)$ from \cref{sec:conditional cycle count limits},
and set $s_r:=\|\rho_r\|_2^2$. We have the useful identity
\begin{equation}\label{eq:Y rho representation}
    Y_T(m,M,\kappa/r)
    =\la\rho_r,V(M)\ra-\frac12s_r.
\end{equation}
By \eqref{eq:B contraction}, there is a constant $q<1$ such that
$|2B_{\kappa/r,\alpha}|\le q$ for every $r\in I$. Hence,
\[
    s_r=\sum_{k=1}^T\frac{(2B_{\kappa/r,\alpha})^{2k}}{2k}
    \le\sum_{k=1}^\infty\frac{q^{2k}}{2k}
    =-\frac12\log(1-q^2).
\]
In particular,
$\sup_{r\in I}s_r<\infty$. Finally, we recall \eqref{eq:truncated cycle weight} here for convenience:
\[
    \mathcal W_r(M)
    :=e^{-Y_T(m,M,\kappa/r)}
      \ind_{\{Y_T(m,M,\kappa/r)\ge-L\}}\,.
\]

\begin{prop}\label{prop:correction-moments}
Uniformly over $r,\hat r\in I$, $x\in\Lambda_n^\typ(r)$, and
\[
    (x,y)\in\Lambda_n^\typ(r,\hat r,\theta),
    \qquad |\theta|\le n^{-1/2}\log n,
\]
it holds
\begin{align}
    \E{\mathcal W_r(G)\mid x\in\sol}
    &=1+o(1),
        \label{eq:one planted weight}\\
    \E{\mathcal W_r(G)\mathcal W_{\hat r}(G)\mid x,y\in\sol}
    &\le(1+o(1))e^{-\la\rho_r,\rho_{\hat r}\ra},
        \label{eq:two planted weight}\\
    \E{\mathcal W_r(G)^2\mid x\in\sol}
    &=O(1).
        \label{eq:one planted squared weight}
\end{align}
\end{prop}

\begin{proof}[Proof of \cref{prop:correction-moments}]
Under the conditional law given $\{x\in\sol\}$,
\eqref{eq:Y rho representation} gives the identity
\begin{equation}\label{eq:Y-single-Gaussian}
     Y_T(m,G,\kappa/r)
    =\la\rho_r,V(G)-\rho_r\ra+\frac12s_r.
\end{equation}
Similarly, under the conditional law given $\{x,y\in\sol\}$, we have
\begin{equation}\label{eq:Y-pair-Gaussian}
\begin{aligned}
    Y_T(m,G,\kappa/r)
    &=\la\rho_r,V(G)-\rho_r-\rho_{\hat r}\ra
      +\frac12s_r+\la\rho_r,\rho_{\hat r}\ra,\\
    Y_T(m,G,\kappa/\hat r)
    &=\la\rho_{\hat r},V(G)-\rho_r-\rho_{\hat r}\ra
      +\frac12s_{\hat r}+\la\rho_r,\rho_{\hat r}\ra.
\end{aligned}
\end{equation}
Letting $W_T\sim N(0,I_T)$, the classical formula for the moment-generating function of a Gaussian gives
\[
    \E{\exp\left\{-\la\rho_r,W_T\ra-\frac12s_r\right\}}=1,
    \qquad
    \E{\exp\left\{-2\la\rho_r,W_T\ra-s_r\right\}}
    =e^{s_r}=O(1).
\]
This will be used shortly in conjunction with \eqref{eq:Y-single-Gaussian} and the normal approximation \eqref{eq:joint normal x}, which asserts $V(G) - \rho_r\approx W_T$ in distribution.
We also have
\[
    \E{\exp\left\{
        -\la\rho_r+\rho_{\hat r},W_T\ra
        -\frac{s_r+s_{\hat r}}2
        -2\la\rho_r,\rho_{\hat r}\ra
    \right\}}
    =e^{-\la\rho_r,\rho_{\hat r}\ra}.
\]
Similarly, this will be used shortly in conjunction with \eqref{eq:Y-pair-Gaussian} and \eqref{eq:joint normal xy}, which asserts $V(G) - \rho_r - \rho_{\hat{r}} \approx W_T$.

We next reintroduce into our approximation the cutoff $\ind( \cdot \ge -L)$ from the definition of $\mathcal{W}_r$. Recall that $L := \sqrt{\log n}$; moreover, as noted above, $s_r$ and hence $\|\rho_r\|_2$ are uniformly bounded for all $r \in I$. Letting $g: x\mapsto  \exp(-x) \ind(x \ge -L)$, these estimates and the formula for the Gaussian density imply that, for some constant $c := c_{\kappa} > 0$,
\begin{align*}
    \E{g\big(\la\rho_r,W_T\ra+\tfrac12s_r\big)}
    &=1+O(n^{-c}),\\
    \E{g\big(\la\rho_r,W_T\ra+\tfrac12s_r\big)^2}
    &=O(1),\\
    \E{
    g\big(\la\rho_r,W_T\ra+\tfrac12s_r
    +\la\rho_r,\rho_{\hat r}\ra\big)
    g\big(\la\rho_{\hat r},W_T\ra+\tfrac12s_{\hat r}
    +\la\rho_r,\rho_{\hat r}\ra\big)}
    &=e^{-\la\rho_r,\rho_{\hat r}\ra}+O(n^{-c}).
\end{align*}

It remains to transfer these Gaussian calculations to the actual cycle counts. We first need a smooth approximation for $g$. Fix a nondecreasing
$\chi\in\mathcal C^\infty(\RR;[0,1])$ that is zero on $(-\infty,0]$ and one on $[1,\infty)$. Set
\[
    \chi_-(x):=\chi(x+L),\qquad
    \chi_+(x):=\chi(x+L+1),\qquad
    g_\pm(x):=e^{-x}\chi_\pm(x).
\]
Then
\[
    \chi_-(x)\le\ind(x\ge-L)\le\chi_+(x),
    \qquad g_-(x)\le g(x)\le g_+(x).
\]
Replacing $g$ by $g_\pm$ changes the Gaussian expectations above by $O(n^{-c})$, after decreasing $c$ if necessary, since these functions differ only on $[-L-1,-L+1]$.
We then apply \eqref{eq:joint normal x} to
\[
    u\longmapsto
    g_\pm\left(
        \la\rho_r,u\ra + \frac12s_r\right), \qquad \text{ and } u\longmapsto
    g_\pm\left(
        \la\rho_r,u\ra  + \frac12s_r\right)^2\,.
\]
Similarly, we apply
\eqref{eq:joint normal xy} to
\[
\begin{aligned}
    u\longmapsto\;&
    g_\pm\left(
        \la\rho_r,u\ra+\frac12s_r
        +\la\rho_r,\rho_{\hat r}\ra
    \right) \times
    g_\pm\left(
        \la\rho_{\hat r},u\ra+\frac12s_{\hat r}
        +\la\rho_r,\rho_{\hat r}\ra
    \right).
\end{aligned}
\]
We claim that the first three derivatives of all three of the above maps are bounded by $O(e^{2L})$. Indeed, the functions $\chi_\pm(\cdot)$ have uniformly bounded derivatives by construction, and the cutoffs $\chi_\pm(\cdot)$ in the definition of $g_\pm$ enforce the uniform bounds $0 \le g_\pm(\cdot) \le e^{L+1}$. Hence, the error incurred by the above invocations of the Gaussian
approximations \eqref{eq:joint normal x} and \eqref{eq:joint normal xy} is at most
\[
    O(n^{-1/5}e^{2L})
    =O\left(
        \exp\left\{-\frac15\log n+2\sqrt{\log n}\right\}
      \right)
    =o(1).
\]
\end{proof}

\begin{remark}[On the choice of the parameter $L$]
    The definition $L:=\sqrt{\log n}$ was somewhat arbitrary and taken for concreteness. Nonetheless, the above argument enforces that $L$ should satisfy $\omega(1) \le L \le o(\log n)$, as two tensions must be (loosely) balanced. First, we need $L$ small enough that $e^{2L}=n^{o(1)}$, so that the cutoff $\chi_\pm$ is sufficiently smooth for approximation; at the same time, $L$ should be large enough that $e^{-cL^2}=o(1)$, so that the cutoff $\chi$ does not affect the limiting Gaussian expectation. 
\end{remark}

With \cref{prop:correction-moments} established, we next compute $\E{Q}$.

\begin{prop}
    For some $c := c_\kappa > 0$, 
    \begin{equation}\label{eq:conditional first moment}
        \E{Q} = (1+o(1))\,\E{\calZ_I^\typ} \ge n^{c\eta}.
    \end{equation}
\end{prop}

\begin{proof}
Since $\PP{x\in\sol}=p(\kappa/r)^m$ for
$x\in\Lambda_n^\typ(r)$, summing \eqref{eq:one planted weight} gives the first equality:
\begin{align}
    \E{Q}
    &=\sum_{r\in I}\sum_{x\in\Lambda_n^\typ(r)}
        p(\kappa/r)^m
        \E{\mathcal W_r(G)\mid x\in\sol}\notag\\
    &=(1+o(1))
        \sum_{r\in I}|\Lambda_n^\typ(r)|\,p(\kappa/r)^m
      =(1+o(1))\,\E{\calZ_I^\typ}.
\end{align}

Next, we develop the lower bound. Let $r_c:=r_*(\kappa,\alpha_c(\kappa))$. By
\eqref{eq:F1-capacity-sign},
\(
    F_1(\kappa,\alpha_c(\kappa),r_c)=0.
\)
Since $\alpha=m/n\le\alpha_c(\kappa)-\eta(\log n)/n$ by assumption,
\[
    \max_{r\ge1}F_1(\kappa,\alpha,r)
    \ge F_1(\kappa,\alpha,r_c)=(\alpha-\alpha_c(\kappa))\log p(\kappa/r_c) \ge
      -\log p(\kappa/r_c)\,\frac{\eta\log n}{n}.
\]
Consequently, \eqref{eq:E[Z] - typical} and \eqref{eq:lower bound EZ I} yield, for example with
$c=-\frac12\log p(\kappa/r_c)>0$,
\[
    \E{\calZ_I^\typ}
    \ge C^{-1}
       \exp\left\{n\max_{r\ge1}F_1(\kappa,\alpha,r)\right\} \ge n^{c\eta}\,.
\]
\end{proof}

\subsection{The second moment}

Fix $r,\hat r\in I$, and define
\[
    f_{r,\hat r}(\theta)
    :=\frac1n\log\left(
      N(r,\hat r,\theta)\,
      p^{(2)}(\kappa/r,\kappa/\hat r,\theta)^m
      \right).
\]
Our goal is to estimate $\E{Q^2}$ from above. We have
\begin{align*}
\E{Q^2}
&=\sum_{r,\hat r\in I}
\E{\calZ_r^\typ\calZ_{\hat r}^\typ
\mathcal W_r(G)\mathcal W_{\hat r}(G)}\\
&=\sum_{r,\hat r\in I}\sum_\theta
\sum_{(x,y)\in\Lambda_n^\typ(r,\hat r,\theta)}
\PP{x,y\in\sol}\,
\E{\mathcal W_r(G)\mathcal W_{\hat r}(G)
\mid x,y\in\sol} \\
&= \sum_{r,\hat{r} \in I}\sum_\theta e^{nf_{r,\hat{r}}(\theta)} \,\E{\mathcal{W}_r(G)\mathcal{W}_{\hat r}(G)\mid x,y \in \sol}\,.
\end{align*}

We begin by expanding $f_{r,\hat{r}}(\theta)$. Combining \eqref{eq:N2 expansion} with
\cref{prop:F2'' central}, and then with \eqref{eq:E[Z] - typical} to truncate to typical solutions: for all $|\theta|$ sufficiently
small, it holds uniformly
\begin{align}
    nf_{r,\hat r}(\theta)
    &\le
    \log\frac{\E{\calZ_r^{\typ}}\E{\calZ_{\hat r}^{\typ}}}
                   {\sqrt{2\pi n}\,r\hat r}
    -\frac n2\bigl(1-4B_{\kappa/r,\alpha}
                         B_{\kappa/\hat r,\alpha}\bigr)\theta^2
    +O(n\theta^4+\theta^2)+o(1).
    \label{eq:f(theta)-taylor}
\end{align}

We divide the possible values of $\theta$ into three disjoint regions. It will be convenient to consider the parameterization 
$
    \theta=\frac{u}{nr\hat r},
$
where $u$ ranges over integers for which
$\Lambda_n^{\typ}(r,\hat r,u/(nr\hat r))$ is nonempty. Fix
$c_2>0$ smaller than the constant $c_0$ in
\cref{lemma:F2-endpoint-uniform}, and let
\begin{align*}
    \Omega_{\mathrm{cen}}&:=\cb{u:|u|\le\sqrt n\log n},\\
    \Omega_{\mathrm{mid}}&:=\cb{u:\sqrt n\log n<|u|
                         \le(1-c_2)nr\hat r},\\
    \Omega_{\mathrm{end}}&:=\cb{u:(1-c_2)nr\hat r<|u|\le nr\hat r}.
\end{align*}

We first consider $u\in\Omega_{\mathrm{cen}}$. Recalling the definition of $\rho_r$ from \cref{sec:conditional cycle count limits}, we have
\[
    \la\rho_r,\rho_{\hat r}\ra
    =\sum_{k=1}^T
      \frac{(4B_{\kappa/r,\alpha}B_{\kappa/\hat r,\alpha})^k}{2k}\,.
\]
Using the fact that $\sum_{k=1}^\infty (4x)^k/(2k)$ is the power-series for $-\log(1-4x)/2$, we have
\[
    -\la \rho_r,\,\rho_{\hat{r}} \ra = \frac{1}{2}\log\pa{1 - 4B_{\kappa/r,\alpha}B_{\kappa/\hat r,\alpha}} + O(n^{-c})\,,
\]
where the error comes from truncating the power series. Convergence of the power series is justified by recalling $1-4B_{\kappa/r_*,\alpha}^2$ is bounded uniformly from below away from zero by \cref{hyp:maxima}, and hence compactness of $J$ and continuity then imply that $1-4B_{\kappa/r,\alpha}B_{\kappa/\hat r,\alpha}$ is also uniformly bounded away from zero for all $n$ sufficiently large. Then, \eqref{eq:two planted weight} yields
\[
    \E{\mathcal W_r(G)\mathcal W_{\hat r}(G)\mid x,y\in\sol}
    \le (1+o(1))
    \sqrt{1-4B_{\kappa/r,\alpha}B_{\kappa/\hat r,\alpha}}.
\]
Finally, \eqref{eq:f(theta)-taylor} and comparison with a Gaussian integral give
\begin{align}
    &\sum_{u\in\Omega_{\mathrm{cen}}}
      \sum_{(x,y)\in\Lambda_n^{\typ}(r,\hat r,u/(nr\hat r))}
      \PP{x,y\in\sol}
      \E{\mathcal W_r(G)\mathcal W_{\hat r}(G)\mid x,y\in\sol}
    \nonumber\\
    &\qquad\le(1+o(1))
      \frac{\E{\calZ_r^{\typ}}\E{\calZ_{\hat r}^{\typ}}}
           {\sqrt{2\pi n}\,r\hat r}
      \sqrt{1-4B_{\kappa/r,\alpha}
B_{\kappa/\hat r,\alpha}}
      \sum_{u\in\ZZ}
      \exp\left\{-\frac{\pa{1-4B_{\kappa/r,\alpha}
B_{\kappa/\hat r,\alpha}}u^2}
                              {2nr^2\hat r^2}\right\}
    \nonumber\\
    &\qquad=(1+o(1))
      \E{\calZ_r^{\typ}}\E{\calZ_{\hat r}^{\typ}},
    \label{eq:2mm-central}\,.
\end{align}

This completes our analysis of $\Omega_{\mathrm{cen}}$, which we claim provides the dominant contribution to $\E{Q^2}$. Our analysis of the remaining regions uses the crude bound that, almost surely,
\begin{equation}\label{eq:WWcrude}
    \mathcal W_r\mathcal W_{\hat r}\le e^{2L}=n^{o(1)}
\end{equation}
We split the
intermediate annular region $\Omega_{\mathrm{mid}}$ at $|u|=c_1nr\hat r$, where $c_1>0$ is sufficiently
small relative to the implicit constant in the error term in \eqref{eq:f(theta)-taylor}. For example, it suffices to take $c_1$ small enough that
\eqref{eq:f(theta)-taylor} holds for $|\theta|\le c_1$ and so that the remainder term denoted by
``$O(n\theta^4)$'' is at most
$(1-4B_{\kappa,r_*,\alpha}^2)n\theta^2/100$. Then, uniformly for $|\theta|\le c_1$,
\[
    nf_{r,\hat r}(\theta)
    \le
    \log\frac{\E{\calZ_r^{\typ}}\E{\calZ_{\hat r}^{\typ}}}
                   {\sqrt{2\pi n}\,r\hat r}
    -c n\theta^2 \le \log\frac{\E{\calZ_r^{\typ}}\E{\calZ_{\hat r}^{\typ}}}
                   {\sqrt{2\pi n}\,r\hat r}
    - c'\log(n)^2.
\]
The second inequality follows from construction of $\Omega_{\mathrm{mid}}$. Using \eqref{eq:WWcrude} and summing over
$u,r,\hat r$, we obtain
\begin{equation}\label{eq:2mm-near intermediate}
    \sum_{r,\hat r\in I}
    \sum_{\substack{u\in\Omega_{\mathrm{mid}}:\\
                    |u|\le c_1nr\hat r}}\,
    \sum_{(x,y)\in\Lambda_n^{\typ}(r,\hat r,u/(nr\hat r))}
    \PP{x,y\in\sol}\,
    \E{\mathcal W_r\mathcal W_{\hat r}\mid x,y\in\sol}
    \le n^{-\omega(1)}\E{\calZ_I^{\typ}}^2.
\end{equation}

We now turn to the remaining portion within $\Omega_{\mathrm{mid}}$. \cref{hyp:maxima} implies that there exists a constant $c>0$ such
that, for all sufficiently large $n$,
\[
    F_2(\kappa,\alpha,r,\hat r,\theta)
    \le F_2(\kappa,\alpha,r,\hat r,0)-c
\]
whenever $c_1\le|\theta|\le1-c_2$. As before, 
continuity of $F_2$, compactness of $J$, and the fact that
$r,\hat r=r_*+o(1)$ implies that this bound holds uniformly over $\alpha \in J$ and $r,\hat{r} \in I$. Finally, \cref{prop:2d-lattice-count} and \eqref{eq:WWcrude} yield
\begin{equation}\label{eq:2mm-bulk}
    \sum_{r,\hat r\in I}
    \sum_{\substack{u\in\Omega_{\mathrm{mid}}:\\
                    |u|>c_1nr\hat r}}
    \,\sum_{(x,y)\in\Lambda_n^{\typ}(r,\hat r,u/(nr\hat r))}
    \PP{x,y\in\sol}\,
    \E{\mathcal W_r\mathcal W_{\hat r}\mid x,y\in\sol}
    \le e^{-cn}\,\E{\calZ_I^{\typ}}^2.
\end{equation}

It remains to treat $\Omega_{\mathrm{end}}$. Fix
$r,\hat r\in I$ and $u\in\Omega_{\mathrm{end}}$, and let
$\theta:=u/(nr\hat r)$ and $\eps:=1-|\theta|<c_0$.
If $\eps>0$ and $N(r,\hat{r},\theta) \neq 0$, we claim: $|r-\hat r|^2\le c_0\eps\log(1/\eps)$. Indeed, take $(x,y)\in\Lambda_n^\typ(r,\hat r,\theta)$.
Since $x,y$ are not collinear, the ratios $x_i/y_i$ are not
all equal. Choose one value taken by these ratios, and let $a$
be the number of indices where it occurs, so $1\le a\le n-1$.
Pairing each of these indices with each of the remaining indices
gives $a(n-a)\ge n-1$ pairs with $x_i/y_i\ne x_j/y_j$.
For each such pair, $x_i y_j-x_j y_i$ is a nonzero integer,
so its square is at least one. Consequently,
\[
    2n^2r^2\hat r^2\eps
    \ge n^2r^2\hat r^2(1-\theta^2)
    =\sum_{i<j}(x_i y_j-x_j y_i)^2
    \ge n-1.
\]
Since $r,\hat r$ are uniformly bounded, $\eps\ge c/n$.
Together with $|r-\hat r|^2=o(n^{-1}\log n)$, this implies
$|r-\hat r|^2\le c_0\eps\log(1/\eps)$, as claimed.
Thus \cref{lemma:F2-endpoint-uniform} applies and gives
\[
    nF_2(\kappa,\alpha,r,\hat r,\theta)
    \le
    n\min\{F_1(\kappa,\alpha,r),F_1(\kappa,\alpha,\hat r)\}
    -cn\eps^{.51}
    \le
    n\min\{F_1(\kappa,\alpha,r),F_1(\kappa,\alpha,\hat r)\}
    -c n^{.49}.
\]
Combining this with \cref{prop:2d-lattice-count} and \eqref{eq:WWcrude} yields (after possibly decreasing the constant $c > 0$)
\begin{align}
    &\sum_{r,\hat r\in I}
    ~\sum_{\substack{u\in\Omega_{\mathrm{end}}: \\ \eps \neq 0}}
    ~\sum_{(x,y)\in
        \Lambda_n^{\typ}\left(r,\hat r,\frac{u}{nr\hat r}\right)}
    \PP{x,y\in\sol}
    \left(
        1+
        \E{\mathcal W_r(G)\mathcal W_{\hat r}(G)
        \mid x,y\in\sol}
    \right)
    &\le e^{-cn^{.49}}\E{\calZ_I^{\typ}}.
    \label{eq:2mm-extreme off diagonal}
\end{align}

Finally, consider $\eps=0$, so that $\theta=\pm1$.
By symmetry, suppose $\theta=1$. Equality in Cauchy--Schwarz
gives $x=(r/\hat r)y$. Choose a coordinate $i$ with
$|y_i|\le\|y\|_2/\sqrt n=\hat r$. Then
\[
    |x_i-y_i|
    =|r/\hat r-1|\,|y_i|
    \le |r-\hat r|=o(1).
\]
Since $x_i-y_i$ is an integer, it equals zero for all sufficiently
large $n$. As $y_i\ne0$, this forces $r=\hat r$ and $x=y$.
Thus the remaining contribution of $\Omega_{\mathrm{end}}$, from $\eps = 0$, consists of the diagonal
and antipodal pairs:

\begin{equation}\label{eq:2mm diagonal}
    2\sum_{r\in I}\sum_{x\in\Lambda_n^{\typ}(r)}
    \PP{x\in\sol}
    \E{\mathcal W_r(G)^2\mid x\in\sol}
    =O(\E{\calZ_I^{\typ}}) \stackrel{\eqref{eq:conditional first moment}}{=} n^{-\Omega(\eta)}\,\E{Q}^2.
\end{equation}

Combining the above estimates on $\Omega_{\mathrm{cen}}$, $\Omega_{\mathrm{mid}}$, and $\Omega_{\mathrm{end}}$ immediately yields \eqref{eq:conditional 2mm - goal}.

\subsection{Log-normal limit}
We begin by removing the dependence of $Q$ on $r$. That is, we will replace the radius-dependent cycle correction by the correction
at $r_*$ via Lipschitz estimates. Differentiating the definition of $Y_T$ and using
$|2B_{\kappa/r,\alpha}|\le q<1$ gives
\begin{equation}\label{eq:Y radius derivative}
    \sup_{r\in I}|\partial_rY_T(m,M,\kappa/r)|
    \le C\left(1+\max_{k\le T}|C_k(m,M)|\right).
\end{equation}
By \eqref{eq:null cycle variance} and a union bound,
\[
    \PP{\max_{k\le T}|C_k(m,G)|>n^{1/20}}\le n^{-c}.
\]
On the complementary event, \eqref{eq:Y radius derivative} and the
definition of $I$ imply
\begin{equation}\label{eq:Y radial comparison}
    \max_{r\in I}
    |Y_T(m,G,\kappa/r)-Y_T(m,G,\kappa/r_*)|=o(n^{-2/5}).
\end{equation}
A smoothing argument (essentially identical to those in the proof of \cref{prop:correction-moments}, but using the null limit law \eqref{eq:joint normal null} for $(C_k)_k$ rather than a planted law) gives
\[
    \PP{Y_T(m,G,\kappa/r_*)<-L+1}\le n^{-c}.
\]
Therefore, with probability at least $1-n^{-c}$, we have from
\eqref{eq:truncated cycle weight} that
\begin{equation}\label{eq:Event Y consequence}
    Q=\calZ_I^{\typ}e^{-Y_T(m,G,\kappa/r_*)}(1+o(1)).
\end{equation}

Next, Chebyshev's inequality and \eqref{eq:conditional 2mm - goal}
show that $Q/\E{\calZ_I^{\typ}}=1+o(1)$ in probability. On the other
hand, \eqref{eq:joint normal null} and
\eqref{eq:Y rho representation} together with the standard smoothing
argument imply
\begin{equation}\label{eq:kolmogorov distance}
    \sup_{x \in \RR} \left|\PP{Y_T(m,G,\kappa/r_*) \le x} - \PP{W \le x}\right| \le O(n^{-c}), \qquad \text{ where } W \sim \mathcal N(-s_{r_*}/2,s_{r_*}).
\end{equation}
Finally, as observed at the start of \cref{subsec:correction},
\begin{align*}
    -\frac12s_{r_*}
    &=\frac14\log(1-4B_{\kappa/r_*,\alpha}^2)+O(n^{-c}), \qquad
    s_{r_*}
    =-\frac12\log(1-4B_{\kappa/r_*,\alpha}^2)+O(n^{-c}).
\end{align*}
Hence, within the intersection of the events $\{Y_T(m,G,\kappa/r_*) \ge -L\}$ and
$\{Q/\E{\calZ_I^{\typ}}=1+o(1)\}$, it further holds with
probability $1-o(1)$ that $\calZ_I^{\typ}>0$ and
\[
    \log\frac{\calZ_I^{\typ}}{\E{\calZ_I^{\typ}}}
    =Y_T(m,G,\kappa/r_*)+o(1)\,.
\]
The normal distribution  $W$ above has variance bounded
above and away from zero, so its density is uniformly
bounded, and hence its distribution function is uniformly Lipschitz. The
last display and the Kolmogorov distance estimate \eqref{eq:kolmogorov distance} prove
\eqref{eq:conditional 2mm - scaling limit goal}. Taking
$t=\pm\frac12\log\log n$ and using the Gaussian tail bound gives
\eqref{eq:conditional 2mm - tail goal}.

\subsection{Solution-space geometry}
Define 
\[
    D = |\{(x,y)\in (\sol^\typ)^2\,:\,\|x-y\|_2\le \delta\sqrt{n},\, x \neq y\}|\,.
\]
Here, $\delta>0$ is a constant chosen sufficiently small, so that
every pair $(x,y)$ contributing to $D$ has correlation $\theta>1-c_2$, where $c_2$ is the constant in the definition of $\Omega_{\mathrm{end}}$ above. Recall from above that any collinear
pair $(x,y) \in \Lambda_n^\typ(r,\hat{r},1)$ necessarily satisfies $x = y$ and $r = \hat{r}$. Hence, \eqref{eq:2mm-extreme off diagonal} gives
\begin{equation}\label{eq:close pair first moment}
    \E{D}\le e^{-cn^{.49}}\E{\calZ_I^{\typ}}.
\end{equation}
By \eqref{eq:conditional 2mm - scaling limit goal},
\[
    \PP{\calZ_I^{\typ}<n^{-1}\E{\calZ_I^{\typ}}}=o(1).
\]
Markov's inequality applied to
\eqref{eq:close pair first moment} then yields
\[
    \PP{D/\calZ_I^{\typ}>e^{-n^{.48}}}
    \le \PP{\calZ_I^{\typ}<n^{-1}\E{\calZ_I^{\typ}}} +n\exp\{n^{.48}-cn^{.49}\}
    =o(1).
\]
This proves \eqref{eq:strong freezing conclusion}.

Finally, suppose that
$m=\lfloor\alpha_c(\kappa)n-\eta\log n\rfloor$. The first-moment
estimates of \cref{sec:1profile} give
$\E{\calZ_I^{\typ}}=n^{O(1)}$. Next, we claim
\[
    \sum_{r,\hat{r} \in I}\sum_{\substack{u \in \Omega_{\mathrm{mid}}  \cup \Omega_{\mathrm{end}} \\ \theta \neq \pm 1}} \sum_{(x,y) \in \Lambda^\typ_n(r,\hat{r},u/(nr\hat r))} \PP{(x,y) \in \sol^2} = n^{-\omega(1)}\,\E{\calZ^\typ}^2.
\]
The claim can be proven identically to \eqref{eq:2mm-near intermediate}, \eqref{eq:2mm-bulk}, and \eqref{eq:2mm-extreme off diagonal} upon recalling that the weights \[
\E{\mathcal{W}_r(G)\mathcal{W}_{\hat r}(G)\mid x,y \in \sol},
\]
which are the only difference between the claim and the referenced bounds, were dealt with using \eqref{eq:WWcrude}, which simply crudely dropped the weights. Hence, the same proofs go through with trivial changes, yielding
\[
    \E{\left|\left\{(x,y) \in (\sol^\typ_{n,m,\kappa})^2\,:\, |\la x,y\ra| \ge \sqrt n\log n,\, x \not\in \{-y,y\}\right\}\right|} \le n^{-\omega(1)}\,\E{\calZ_I^{\typ}}^2=n^{-\omega(1)}\,.
\]
Markov's inequality gives
\eqref{eq:complete freezing conclusion} with probability at least
$1-n^{-1}$. This completes all parts of \cref{lemma:conditional 2mm}. \qed \\

\section{Proofs of the main theorems}
\label{sec:main theorem proofs}

In this section, we deduce the main results from the estimates of the preceding
sections. 

We first state the more refined version of \cref{thm:main}, which includes the precise definition of the function $f_\kappa$.

\begin{theorem}[Sharp threshold and critical window]\label{thm:main2}
Fix $\kappa>0$ such that $\alpha_c(\kappa)>1$ and
\cref{hyp:maxima} holds with $\alpha=\alpha_c(\kappa)$. Abbreviate
$\alpha_c:=\alpha_c(\kappa)$,  let
\[
    r_c:=\arg\max_{r\ge1}F_1(\kappa,\alpha_c,r),
    \qquad p_c:=p(\kappa/r_c),
\]
and set
\begin{equation*}
    B_c:=\frac{\sqrt{\alpha_c}}2
       \left(\E{W^2\mid |W|\le\kappa/r_c}-1\right), \quad 
    \gamma_c:=-\left.
       \frac{\partial^2}{\partial v^2}
       F_1(\kappa,\alpha_c,\sqrt v)
       \right|_{v=r_c^2}, \quad 
    C_\kappa:=
       \left(\gamma_c\Var_{X\sim\mu_{r_c}}(X^2)\right)^{-1/2}.
\end{equation*}
Finally, define
\begin{equation}\label{eq:critical window distribution}
    f_\kappa(t):=\E{\exp\left\{-\frac{C_\kappa}{2}
                    p_c^t e^{W_c}\right\}}, \qquad  W_c\sim\mathcal N\left(
       \frac14\log(1-4B_c^2),
       -\frac12\log(1-4B_c^2)
       \right).
\end{equation}
Then, for each fixed $T>0$, uniformly over $t\in[-T,T]$ such that
$m=\alpha_cn+t\in\mathbb N$,
\begin{equation}\label{eq:main thm - scaling limit}
    \PP { \sol_{n,m,\kappa} = \emptyset } = f_\kappa(t)+o(1).
\end{equation}
Furthermore,
$\lim_{t\to-\infty}f_\kappa(t)=0$ and
$\lim_{t\to+\infty}f_\kappa(t)=1$. 
\end{theorem}

We record precise estimates for the first moment near the critical
point, which supply the normalization in the critical-window law.
The constants $r_c,p_c,\gamma_c$, and $C_\kappa$ are the same as those in
\cref{thm:main2} above.

\begin{lemma}[Critical first moment]
\label{lem:critical first moment}
Under the assumptions and notation of \cref{thm:main2}, for every fixed
$C<\infty$, uniformly over $s$ satisfying
$|s|\le C\log n$ and $m=\alpha_c n+s\in\mathbb N$,
\begin{equation}\label{eq:critical first moment}
    \E{\calZ(n,m,\kappa)}=(C_\kappa+o(1))p_c^s.
\end{equation}
\end{lemma}

\begin{proof}
Fix $C<\infty$ and suppose that $|s|\le C\log n$. Recalling
$w_n:=\log\log n$, we first claim that
\begin{equation}\label{eq:critical radius truncation}
    \sum_{\substack{r\in\mathcal R_n:\
                    \sqrt n\,|r-r_c|>w_n}}
       \E{\calZ_r}
    =o(p_c^s).
\end{equation}

Indeed, \eqref{eq:F1-capacity-sign} gives
\[
    F_1(\kappa,\alpha_c,r_c)=0.
\]
The radius-truncation argument in the proof of
\cref{thm:truncated 1mm}, with
$(\alpha,r_*)=(\alpha_c,r_c)$, applies mutatis mutandis. Hence, for $r$ near $r_c$,
\[
    N(r)p(\kappa/r)^{\alpha_c n}
    \le K n^{-1/2}e^{-cn(r-r_c)^2}.
\]
The only additional factor is $p(\kappa/r)^s$. A first-order expansion
gives
$
    p(\kappa/r)^s
    =p_c^s\exp\{O(|s||r-r_c|)\}.
$
Then,
\begin{align*}
    \E{\calZ_r}
    &\le K n^{-1/2}p_c^s
       \exp\{-cn(r-r_c)^2+K|s||r-r_c|\}\\
    &\le K' n^{-1/2}p_c^s
       \exp\{-c'n(r-r_c)^2\},
\end{align*}
where the second inequality follows by completing the square and using
$s^2/n=o(1)$. Thus the proof of \eqref{eq:F1 trunc mid} applies without
further change.

On any fixed compact set away from $r_c$, it holds uniformly that
$p(\kappa/r)^s/p_c^s= n^{O(1)}$, while $\E{\calZ_r}$ is already exponentially smaller than $\E{\calZ_I}$. Finally, for large $r$, we have $p(\kappa/r)^s\le (Kr)^{C\log n}$ for some $K<\infty$. Since $\alpha_c>1$, the
summation in \eqref{eq:F1 trunc extreme} still decays exponentially. Hence, arguing as in \cref{thm:truncated 1mm} yields
\eqref{eq:critical radius truncation}.

It therefore suffices to sum over radii satisfying
$\sqrt n\,|r-r_c|\le w_n$. Uniformly over these radii,
\cref{prop:1d-lattice-count} gives
\[
    \E{\calZ_r}
    =\frac{1+o(1)}
      {\sqrt{2\pi n\Var_{X\sim\mu_r}(X^2)}}
      \exp\{nF_1(\kappa,\alpha_c,r)\}
      p(\kappa/r)^s.
\]
Recalling that $F_1(\kappa,\alpha_c,r_c)=0$ and that $r_c$ is the
maximizer of $F_1$, the definition of $\gamma_c$ gives
\[
    F_1(\kappa,\alpha_c,r)
    =-\frac{\gamma_c}{2}(r^2-r_c^2)^2
      +O(|r^2-r_c^2|^3)
    =-\frac{\gamma_c}{2}(r^2-r_c^2)^2+o(n^{-1}).
\]
Moreover, first-order expansions give
\[
    p(\kappa/r)^s
    =p_c^s\exp\{O(|s||r-r_c|)\}
    =(1+o(1))p_c^s, \quad
    \Var_{X\sim\mu_r}(X^2)
    =(1+o(1))\Var_{X\sim\mu_{r_c}}(X^2).
\]
Consequently, 
\begin{align*} \E{\calZ} &=(1+o(1)) \frac{p_c^s} {\sqrt{2\pi n\Var_{X\sim\mu_{r_c}}(X^2)}} \sum_{\substack{r\in\calR_n\\ \sqrt n\,|r-r_c|\le\log\log n}} \exp\left\{-\frac{\gamma_c n}{2}(r^2-r_c^2)^2\right\} +o(p_c^s). \\ &= (1+o(1)) \frac{p_c^s} {\sqrt{\gamma_c\Var_{X\sim\mu_{r_c}}(X^2)}} =(C_\kappa+o(1))p_c^s. \end{align*}
Between the first and second lines, we have used that the Gaussian sum converges to $\sqrt{2\pi n/\gamma_c}$. \end{proof}
We are ready to prove our main result about the structure of the solution set.
\begin{proof}[Proof of \cref{thmStructure}] We prove the claims in \cref{thmStructure} in reverse order, beginning with item \ref{item:lognormal}. We begin by giving the more precise statement of the log-normal limit.

\begin{lemma}\label{lem:lognormal}
Let $\kappa, \alpha, n, m$ be as in \cref{thmStructure}. Let
\[
    B_*:=\frac{\sqrt\alpha}{2}
       \left(\E{W^2\mid |W|\le\kappa/r_*}-1\right)
\]
and
\[
    W_{\kappa,\alpha}\sim\mathcal N\left(
       \frac14\log(1-4B_*^2),
       -\frac12\log(1-4B_*^2)
       \right),
\]
then
\begin{equation}\label{eq:main lognormal limit}
    \sup_{s\in\RR}
    \left|
\PP{\calZ>e^s\E{\calZ}}-\PP{W_{\kappa,\alpha}>s}
    \right|=o(1).
\end{equation}
\end{lemma}
\begin{proof} 
Let $\alpha_n=m/n$. Note that \cref{hyp:maxima} is robust to sufficiently small perturbations in $\alpha$: item~(2) and the local expansion at
$\theta=0$ control small overlaps, \cref{lemma:F2-endpoint-uniform}
controls a neighborhood of $\theta=\pm1$, and continuity and compactness
control the remaining values of $\theta$.
Consequently, \cref{hyp:maxima} holds throughout some compact interval
$J\subset(1,\alpha_c(\kappa)]$ containing $\alpha$ in its interior.
For all sufficiently large $n$, we have $\alpha_n\in J$ and
$m\le\alpha_c(\kappa)n-\eta\log n$ for some fixed $\eta>0$, so the results of
\cref{sec:second moment method} apply. Let
$I=I_n(\kappa,\alpha_n)$ and write $\calZ^{\typ}=\calZ_I^{\typ}$. By
\cref{thm:truncated 1mm,lemma:conditional 2mm},
\begin{equation}\label{eq:full versus typical}
    \frac{\calZ-\calZ^{\typ}}{\calZ^{\typ}}=o_{\mathbb P}(1),
    \qquad
    \frac{\E{\calZ}}{\E{\calZ^{\typ}}}=1+o(1).
\end{equation}
Indeed, the numerator in the first ratio is at most
$e^{-c(\log\log n)^2}\E{\calZ^{\typ}}$ with high probability, while the denominator is lower bounded via
\eqref{eq:conditional 2mm - tail goal}. Hence, with high probability,
\begin{align*}
    \log\frac{\calZ}{\E{\calZ}}
    &=
    \log\frac{\calZ^{\typ}}{\E{\calZ^{\typ}}}
    +\log\frac{\calZ}{\calZ^{\typ}}
    +\log\frac{\E{\calZ^{\typ}}}{\E{\calZ}} \stackrel{\eqref{eq:full versus typical}}=
    \log\frac{\calZ^{\typ}}{\E{\calZ^{\typ}}}
    +o_{\mathbb P}(1),
\end{align*}
Since $\alpha_n\to\alpha$, continuity
gives $r_*(\kappa,\alpha_n)\to r_*$ and $B_{\kappa/r_*(\kappa,\alpha_n),\alpha_n}\to B_*$.
Thus, the mean and variance of the Gaussian $W_{\alpha_n}$ in
\eqref{eq:conditional 2mm - scaling limit goal} converge to those of
$W_{\kappa,\alpha}$. Since these Gaussian densities are uniformly
bounded, the $o_{\mathbb P}(1)$ perturbation above does not affect the
Kolmogorov distance. This establishes \eqref{eq:main lognormal limit}, and also item \ref{item:lognormal} of \cref{thmStructure}. \end{proof}

We next prove item \ref{item:typical empirical} of \cref{thmStructure}, the assertion about empirical distributions. For fixed
$r\in I$, consider an empirical distribution $\nu$ of a vector in
$\Lambda_n^{\typ}(r)$. 
Writing $\mu_r(k)=e^{a_rk^2+b_r}$, we have
\[
    \sum_k k^2\nu(k)=r^2=\sum_k k^2\mu_r(k).
\]
Hence,
\begin{align*}
    D(\nu\|\mu_r)
    =-H(\nu)-\sum_k\nu(k)(a_rk^2+b_r)&=-H(\nu)-a_rr^2-b_r =H(\mu_r)-H(\nu).
\end{align*}
On the other hand, there are at most $(n+1)^{2\log n}$ possible empirical
distributions $\nu$ since typical vectors $x$ satisfy $\|x\|_\infty\le \log n$ by construction. Moreover, standard counting considerations give that each empirical distribution $\nu$ corresponds to at most $e^{nH(\nu)}$
vectors. Hence, for every $s>0$,
\begin{equation}\label{eq:empirical type bound}
    \left|\left\{x\in\Lambda_n^{\typ}(r):
       D(\nu_x\|\mu_r)\ge s\right\}\right|
    \le \exp\{nH(\mu_r)-ns+O((\log n)^2)\}.
\end{equation}
We remark that this is the main computation in the proof of
Sanov's theorem; however, it is more convenient here to use the explicit bound above rather than quote the full theorem.

Let $\eps_n=n^{-1/2}(\log n)^{3}$ and define the set of solutions with ``atypical'' empirical distribution by
\[
    A_0:=\left\{x\in\sol_{n,m,\kappa}^{\typ}:
       \|\nu_x-\mu_{r_x}\|_{\mathrm{TV}}>\eps_n\right\},
    \qquad r_x:=\|x\|_2/\sqrt n.
\]
Combining this with \eqref{eq:empirical type bound} gives
\begin{align*}
    \E{|A_0|}
    &\le
    e^{-2n\eps_n^2+O((\log n)^2)}
    \sum_{r\in I}e^{nH(\mu_r)}p(\kappa/r)^m.
\end{align*}
Applying the local central limit
theorem to \cref{prop:1d-lattice-count}, and then using \eqref{eq:E[Z] - typical} to restrict to typical vectors, it holds uniformly over $r\in I$, 
\[
    e^{nH(\mu_r)}
    \le C\sqrt n\,|\Lambda_n^{\typ}(r)|.
\]
Hence, as $n\eps_n^2 = \omega((\log n)^2)$,
\[
    \E{|A_0|}
    \le e^{-cn\eps_n^2}\E{\calZ^{\typ}}.
\]
Consequently, Markov's inequality and
\eqref{eq:conditional 2mm - tail goal} give
\begin{align*}
    \PP{\frac{|A_0|}{\calZ^{\typ}}>
             e^{-cn\eps_n^2/2}}
    &\le
    \PP{\calZ^{\typ}<(\log n)^{-1/2}\E{\calZ^{\typ}}}
      +(\log n)^{1/2}e^{-cn\eps_n^2/2}=o(1).
\end{align*}
That is, with high probability, the proportion of solutions with atypical empirical distribution is vanishing:
\begin{equation}\label{eq:empirical good solutions}
    \frac{|A_0|}{\calZ^{\typ}}=o_{\mathbb P}(1).
\end{equation}
Of course, our definition of $A_0$ compares to $\mu_{r_x}$ rather than $\mu_{r_*}$; it remains to correct for this. 
Since $\alpha_n\to\alpha$ and $I$ is centered at
$r_*(\kappa,\alpha_n)$,
\[
    \sup_{r\in I}|r-r_*|
    \le\frac{w_n}{\sqrt n}
       +|r_*(\kappa,\alpha_n)-r_*|
    =o(1).
\]
Continuity of the discrete Gaussian in its variance therefore gives
\[
    \sup_{r\in I}\|\mu_r-\mu_{r_*}\|_{\mathrm{TV}}=o(1).
\]
Thus every $x\in\sol_{n,m,\kappa}^{\typ}\setminus A_0$ satisfies
\[
    \|\nu_x-\mu_{r_*}\|_{\mathrm{TV}}
    \le\eps_n+
       \sup_{r\in I}\|\mu_r-\mu_{r_*}\|_{\mathrm{TV}}
    =o(1).
\]
Combining the above estimates readily yields
item \ref{item:typical empirical} of \cref{thmStructure}; moreover, we remark that tracking the errors more carefully above leads to the quantitative rate $\|\nu_x - \mu_{r_*}\|_{\mathrm{TV}} \le n^{-1/2}(\log n)^{c}$. \\

We establish item \ref{item:typical orthogonal} next, which asserts the near-orthogonality of typical solutions. Let $P^{\typ}$ be the number of ordered pairs $(x,y)$ of typical
solutions $x,y\in\sol_{n,m,\kappa}^{\typ}$ such that
$y\notin\{x,-x\}$ and
$|\la x,y\ra|>\sqrt n\log n$. The unweighted estimates in the proof of
\cref{lemma:conditional 2mm}, obtained by removing the weights in the
arguments for
\eqref{eq:2mm-near intermediate}, \eqref{eq:2mm-bulk}, and
\eqref{eq:2mm-extreme off diagonal}, give
\[
    \E{P^{\typ}}=o\bigl(\E{\calZ^{\typ}}^2\bigr).
\]
It follows from \eqref{eq:conditional 2mm - scaling limit goal} and
Markov's inequality that $P^{\typ}/(\calZ^{\typ})^2=o_{\mathbb P}(1)$.
Except for pairs $(x,y)$ with $y\in\{x, -x\}$ (which are negligible in number compared to $\calZ^2$), all remaining ordered pairs have at least one endpoint outside the
typical set and hence number at most
$2\calZ(\calZ-\calZ^{\typ})$. Applying \eqref{eq:full versus typical}
yields item~(2).\\

It remains to prove item \ref{item:freezing}, namely freezing in Hamming distance. Let $\delta$ be the constant in \eqref{eq:strong freezing conclusion}; we will prove item \ref{item:freezing} with a modified constant $\delta'$. Since $I$
remains in a compact subset of $(1,\infty)$, there are $c,C>0$ such
that the sub-Gaussian estimate $\mu_r(k)\le Ce^{-ck^2}$ holds uniformly for $r\in I$. Thus, there exists
some fixed $M<\infty$ with
\[
    \sup_{r\in I}
    \E[\mu_r]{X^2\ind_{\{|X|>M\}}}\le\frac{\delta^2}{32}.
\]
For all $n$ sufficiently large, we then have for every $x\in\sol_{n,m,\kappa}^{\typ}\setminus A_0$,
\begin{align*}
    \frac1n\sum_{i:\,|x_i|>M}x_i^2 &\le \E[\mu_{r_x}]{X^2\ind_{\{|X|>M\}}}
      + 2(\log n)^2\eps_n \le\frac{\delta^2}{16}\,.
\end{align*}
That is, ``large'' coordinates do not contribute significantly to $\|x\|$. Next, choose $\delta'\in(0,1/2)$ so that
$2M^2\delta'\le\delta^2/16$. If
$x,x'\in\sol_{n,m,\kappa}^{\typ}\setminus A_0$ differ in at most
$2\delta' n$ coordinates, decomposing the set $\{i\,:\, x_i \neq x_i'\}$ based on whether $|x_i| > M$ or $|x_i| \le M$ (and similarly for $x_i'$) gives
\begin{align*}
    \|x-x'\|_2^2
    &\le2\sum_{i:x_i\ne x_i'}x_i^2
       +2\sum_{i:x_i\ne x_i'}(x_i')^2 \le4\pa{\frac{\delta^2n}{16}+2M^2\delta'n}
    \le\frac{\delta^2n}{2}.
\end{align*}
This gives freezing with respect to Hamming distance for pairs of typical solution vectors. It remains to briefly handle freezing in the case that at least one of the solution vectors is atypical. Define
\[
    A:=(\sol_{n,m,\kappa}\setminus\sol_{n,m,\kappa}^{\typ})\cup A_0,
\]
the set of solutions which are atypical in radius, infinity norm, or empirical distribution. Also define the set of typical solutions violating freezing as
\[
    B := \{x \in \sol_{n,m,\kappa}\setminus A\,:\, \exists\,y\in (\sol_{n,m,\kappa}\setminus A)\setminus\{x\},~\|x-y\|_2 \le \delta \sqrt{n}\}\,.
\]
By
\eqref{eq:full versus typical} and
\eqref{eq:empirical good solutions}, $|A|$ is negligible; by
\eqref{eq:strong freezing conclusion}, $|B|$ is also negligible. That is,
\[
    \frac{|A|}{\calZ}=o_{\mathbb P}(1),
    \qquad
    \frac{|B|}{\calZ}=o_{\mathbb P}(1).
\]

A fixed vector $y\in A$ can be within Hamming distance $\delta'n$ of at most one solution outside $A\cup B$.
Indeed, if two distinct solutions $x,x'\notin A\cup B$ were each within Hamming distance $\delta'n$ of $y$, then $x$ and $x'$ would differ in at most $2\delta'n$
coordinates. The preceding bound would give
$\|x-x'\|_2\le\delta\sqrt n$, so the definition of $B$ would imply $x,x'\in B$, which is a contradiction.

Therefore, the total
number of solutions $x \in \sol_{n,m,\kappa}$ having at least one (different) solution within Hamming distance
$\delta' n$ is at most
\[
    2|A|+|B|=o_{\mathbb P}(\calZ).
\]
This proves item \ref{item:freezing} and completes all parts of the proof of \cref{thmStructure}.
\end{proof}

Finally, we prove our main theorem about the satisfiability threshold and critical window. (Recall that \cref{thm:main2} subsumes \cref{thm:main}.)
\begin{proof}[Proof of \cref{thm:main2}]
We directly adapt the stopping-time argument of Sah and Sawhney
\cite[Section~2]{ss}, mutatis mutandis. There are some minor running differences due to the present model allowing for slight variation in the norms of the solutions. These are handled using the same type of Lipschitz estimates that have been used throughout the rest of the manuscript. We present the proof for completeness.

By assumption, \cref{hyp:maxima} holds at $\alpha_c$. The same stability argument used in the previous proof gives that \cref{hyp:maxima} also holds
throughout some interval $J=[\alpha_0,\alpha_c]$. We use
\cref{lemma:conditional 2mm} on this interval. Choose a sufficiently small fixed $\eta>0$ and reveal the first
\[
    m_0:=\lfloor\alpha_c n-\eta\log n\rfloor
\]
rows. Define $\mathcal F_{m_0}:=\sigma(G_1,\ldots,G_{m_0})$, and let $s_0 :=m_0-\alpha_c n$,
$I_0 :=I_n(\kappa,m_0/n)$, and
\[
    \mathcal T_0:=\bigcup_{r\in I_0}\Lambda_n^{\typ}(r),
    \qquad
    \calZ_0:=|\sol_{n,m_0,\kappa}\cap\mathcal T_0|.
\]
Note that $\mathcal{T}_0$, which should be thought of as the surviving typical solutions at the precritical time, is $\mathcal{F}_{m_0}$--measurable. By \cref{lem:critical first moment,thm:truncated 1mm},
\begin{equation}\label{eq:pretime first moment}
    \E{\calZ_0}=(C_\kappa+o(1))\,p_c^{s_0}.
\end{equation}
Choose $\eta$ small enough that
$\eta\log(1/p_c)<1/8$. Then
\cref{lemma:conditional 2mm} implies that with high probability, 
\begin{equation}\label{eq:pretime good event1}
    \frac{\calZ_0}{\E{\calZ_0}}
       \in\bigl((\log n)^{-1/2},(\log n)^{1/2}\bigr),
    \qquad
    \calZ_0\le n^{1/4},
\end{equation}
and also, simultaneously for all pairs $x,y\in\sol_{n,m_0,\kappa}\cap\mathcal T_0$ with
$y\notin\{x,-x\}$,
\begin{equation}\label{eq:pretime good event2}
    |\la x,y\ra|\le\sqrt n\log n\,.
\end{equation}

Fix $T<\infty$, let $m=\alpha_c n+t\in\mathbb N$ with $|t|\le T$, and
let $\ell=m-m_0$. Conditioned on $\mathcal F_{m_0}$, choose one
representative from each antipodal pair in
$\sol_{n,m_0,\kappa}\cap\mathcal T_0$, and call the resulting set $S$.
Thus $|S|=\calZ_0/2$. For a new standard Gaussian row $G_j=\sqrt nA_j$,
the vector
\[
    \pa{\la G_j,x/\|x\|_2\ra}_{x\in S}
\]
is centered Gaussian with covariance matrix $R$ given by
\[
    R_{x,y}=\frac{\la x,y\ra}{\|x\|_2\|y\|_2}.
\]
Within the intersection of the high-probability events in \eqref{eq:pretime good event1} and \eqref{eq:pretime good event2}, the covariance $R$ trivializes:
\[
    \|R-I\|_{\mathrm{F}}
    \le C|S|\frac{\log n}{\sqrt n}=o(1).
\]
Consequently, using the formula for the relative entropy of multidimensional Gaussians readily gives
\[
    D\left(
       N(0,R)^{\otimes\ell}\,\middle\|\,N(0,I)^{\otimes\ell}
    \right)
    \le C\ell\,\|R-I\|_{\mathrm F}^2=o(1).
\]
In particular, Pinsker's inequality gives that the joint law for the Gaussian processes corresponding to the inner products of $S$ with the next
$\ell=O(\log n)$ rows is vanishingly close in total variation from the idealized law in
which all these Gaussian variables are independent. This is
the decoupling step in
\cite[Section~2]{ss}.

For $x\in S$ with $\|x\|_2=r_x\sqrt n$, its survival probability under the addition of $\ell$ fresh constraints is
\[
    q_x:=p(\kappa/r_x)^\ell.
\]
Next, the dependence on $x$ is removed from $q$. We claim:
\begin{equation}\label{eq:critical maximizing radius comparison}
    r_*(\kappa,m_0/n)=r_c+O_{\kappa,\eta}((\log n)/n).
\end{equation}
Indeed, let $\Phi(v,\beta):=\partial_vF_1(\kappa,\beta,\sqrt v)$. Then
$\Phi(v_c,\alpha_c)=0$ and
$\partial_v\Phi(v_c,\alpha_c)=-\gamma_c<0$. The implicit function
theorem therefore gives a $C^1$ function
$v_*(\beta)=r_*(\kappa,\beta)^2$ near $\alpha_c$, and hence
\[
    |v_*(\beta)-v_c|\le C|\beta-\alpha_c|,
\]
where $C$ depends only on $\kappa$. Since $v_c>0$ and $\eta$ is fixed,
this yields the claim. Consequently, we obtain that uniformly over $x\in S$,
\begin{equation}\label{eq:varying radii survival}
    q_x=p_c^\ell
       \left(1+O\left(\frac{(\log n)(\log\log n)}{\sqrt n}\right)\right).
\end{equation}
Within the high probability events in \eqref{eq:pretime good event1} and \eqref{eq:pretime good event2},
$\calZ_0p_c^\ell=O(\sqrt{\log n})$ almost surely by
\eqref{eq:conditional 2mm - tail goal} and
\eqref{eq:pretime first moment}, while
$p_c^\ell=n^{-\eta\log(1/p_c)+o(1)}=o(1)$. Hence
\eqref{eq:varying radii survival} and the expansion
$\log(1-z)=-z+O(z^2)$ let us estimate the probability that no solutions from $\mathcal{T}_0$ survive after the $m^\text{th}$ constraint:

\begin{equation}\label{eq:conditional extinction}
    \PP{\mathcal T_0 \cap \sol_{n,m,\kappa} = \emptyset
          \mid \mathcal F_{m_0}}
    = \exp\left\{-\frac12p_c^\ell\calZ_0\right\}+o(1).
\end{equation}
The factor $1/2$ appears because solutions occur in antipodal pairs.

It remains only to remove the restriction to the set of typical solutions, $\mathcal T_0$. Apply
\eqref{eq:critical radius truncation} with
$w_n=\frac12\log\log n$. Together with \eqref{eq:critical maximizing radius comparison} and the infinity-norm truncation
argument leading to \eqref{eq:E[Z] - typical}, which is uniform when
$m/n$ ranges in a fixed neighborhood of $\alpha_c$, we have uniformly for $|t|\le T$,
\[
    \E{\left|\sol_{n,m,\kappa}\setminus\mathcal T_0\right|}=o(1).
\]
Thus Markov's inequality and \eqref{eq:conditional extinction} yield
\begin{equation}\label{eq:extinction from pretime}
    \PP{\sol_{n,m,\kappa}=\emptyset}
    =\E{\exp\left\{-\frac12p_c^\ell\calZ_0\right\}}+o(1).
\end{equation}
Finally, as $m_0/n\to\alpha_c$, \eqref{eq:conditional 2mm - scaling limit goal} gives
\[
    \log\frac{\calZ_0}{\E{\calZ_0}}
    \ \xrightarrow{\ d\ }\ W_c,
\]
Moreover,
\eqref{eq:pretime first moment} with $t := s_0 + \ell$ gives
\[
    \frac12p_c^\ell\,\E{\calZ_0}
    =\left(\frac{C_\kappa}{2}+o(1)\right)p_c^t.
\]
Since $(t,z)\mapsto\exp\left(-C_\kappa p_c^t e^{z}/2\right)$ is bounded and uniformly continuous on $[-T, T]\times \mathbb R$, it follows uniformly for $|t|\leq T$ that 
\[\E{\exp\left\{-\frac{1}{2}p_c^{\ell}\calZ_0\right\}}=\E{\exp\left\{-\frac{C_\kappa}{2} p_c^t e^{W_c}\right\}}+o(1).\]
Then \eqref{eq:extinction from pretime} gives, uniformly
for $|t| \le T$ and $\alpha_cn+t\in\mathbb N$,
\[
    \PP{\sol_{n,m,\kappa}=\emptyset}
    =\E{\exp\left\{-\frac{C_\kappa}{2}p_c^te^{W_c}\right\}}+o(1)
    =f_\kappa(t)+o(1).
\]
Finally, $p_c\in(0,1)$ and $e^{W_c}>0$ almost surely, so the dominated convergence theorem gives the claimed limits for $f_\kappa(t)$ as $t \to \pm \infty$.
\end{proof}

\section*{Acknowledgements}
WP supported in part by NSF grant DMS-2348743.

\bibliographystyle{plain}
\bibliography{references}

\appendix

\section{Planar discrete Gaussian}

\begin{proof}[Proof of \cref{prop:dg2}]
Write $S=(\ZZ\setminus\{0\})^2$ and
\[
    R=r^2-1,\qquad T=\hat r^2-1,\qquad s=r\hat r\theta.
\]
For $q(x,y)=Ax^2+Bxy+Cy^2$, the sum
$\sum_{(x,y)\in S}e^{-q(x,y)}$ is finite exactly when
$A,C>0$ and $B^2<4AC$. Also, no nonzero linear combination of $x^2,xy,y^2$ is
constant on $S$.

The results of \cite[Proposition~3.2, Theorems~3.3--3.4,
and Appendix~B.1]{wainwright-jordan} therefore reduce the claim
to the following check: for every nonzero quadratic $q(x,y)=Ax^2+Bxy+Cy^2$
bounded below on $S$,
\[
    Ar^2+Bs+C\hat r^2>\inf_{(x,y)\in S}q(x,y).
\]
Such a quadratic necessarily has $A,C\ge0$ and
$|B|\le2\sqrt{AC}$; otherwise it tends to $-\infty$ along
integer multiples of a suitable vector.
Taking $x,y\in\{-1,1\}$ gives
\[
    \inf_{(x,y)\in S}q(x,y)\le A+C-|B|.
\]
By the hypothesis, choose $\varepsilon\in(0,1)$ such that
$|s|-1\le(1-\varepsilon)\sqrt{RT}$. Then
\begin{align*}
    Ar^2+Bs+C\hat r^2-(A+C-|B|)
    &=AR+CT+Bs+|B|\\
    &\ge AR+CT-(1-\varepsilon)\sqrt{RT}\,|B|\\
    &\ge\varepsilon(AR+CT)>0,
\end{align*}
where we used $|B|\le2\sqrt{AC}$ and
$2\sqrt{RATC}\le AR+CT$. The cited result gives the claimed density and unique coefficients.
\end{proof}

\section{Auxiliary estimates for the second moment profile}
\label{app:profile-estimates}

\subsection{Numerical convention}\label{app:numer-conventions}

All decimal inequalities in this appendix are rigorous interval bounds.
Finite expressions involving
$\exp$, $\log$, $\sqrt{\cdot}$, $\operatorname{erf}$, and
$\operatorname{erfc}$ are evaluated using arbitrary-precision ball
arithmetic; here
\[
    p(x)=\operatorname{erf}(x/\sqrt2),
    \qquad
    \overline\Phi(x)=\frac12\operatorname{erfc}(x/\sqrt2),
\]
where (recall) $\overline\Phi(x) = \PP{W>x}$ for a standard Gaussian variable $W$.
The parameter $t = t(r):=1/(2\sigma(r)^2)$ is enclosed by bisection using the
strictly monotone moment equation
\begin{equation}\label{eq:interval-moment-equation}
    r^2=
    \frac{\sum_{j\ge1}j^2e^{-tj^2}}
         {\sum_{j\ge1}e^{-tj^2}}.
\end{equation}

Every positive series is truncated with a geometric or integral bound on
the omitted tail. Every displayed decimal is rounded outward. After the
analytic reductions we present in the following subsections, the remaining checks only concern explicit
functions of at most two variables on the stated domains. We state the
certified bounds and omit the
unilluminating subdivision arithmetic. 

We repeatedly use the elementary inequalities
\begin{equation}\label{eq:Gaussian-tail-two-sided}
    \frac{x}{1+x^2}\phi(x)
       \le\overline\Phi(x)
       \le\frac{\phi(x)}x,
    \qquad x>0.
\end{equation}
We also use that for any $u>0$,
\begin{equation}\label{eq:theta-sum-elementary}
    \sum_{m\in\ZZ}e^{-um^2}-1
       \le\frac{2e^{-u}}{1-e^{-3u}},
    \qquad
    \sum_{m\in\ZZ}e^{-um^2}
       \le1+\sqrt{\frac\pi u}.
\end{equation}
The first inequality follows by comparing consecutive terms in the sum, and the second by comparing
the sum with its Gaussian integral.

\subsection{Proof of \cref{lem:Gaussian-overlap-estimates}}
\label{app:Gaussian-small-window}

\begin{proof}
Let $(h_j)_{j\ge0}$ be the orthonormal Hermite polynomials in
$L^2(\gamma)$, and expand
\[
    \frac{\ind_{\{|x|\le K\}}}{p(K)}
       =1+\sum_{j\ge1}a_jh_j(x).
\]
Symmetry gives $a_j=0$ for odd $j$, and correlated standard
Gaussians satisfy
\[
    \E{h_j(W)h_\ell(W_\theta)}
       =\ind_{\{j=\ell\}}\theta^j.
\]
Consequently,
\[
    \frac{p^{(2)}(K,K,\theta)}{p(K)^2}
       =1+\sum_{j\ge1}a_{2j}^2\theta^{2j}.
\]
Using Parseval's identity and $h_2(x)=(x^2-1)/\sqrt2$, we obtain
\[
    \sum_{j\ge1}a_{2j}^2=\frac{1-p(K)}{p(K)},
    \qquad
    a_2^2=\frac{d(K)^2}{2}.
\]
The bounds \eqref{eq:Gaussian-Hermite-coarse} and \eqref{eq:Gaussian-Hermite-refined} now follow from $\log(1+x)\le x$. 

We next prove \eqref{eq:Gaussian-small-window}. We use the shorthands
\[
    P(\theta)=p^{(2)}(K,K,\theta),
    \qquad
    z=\frac{K^2\theta}{1-\theta^2}.
\]
The derivative identity \eqref{eq:energy_derivative} gives
\[
    P'(\theta)
      =\frac{2e^{-K^2/(1-\theta^2)}\sinh z}
             {\pi\sqrt{1-\theta^2}}.
\]
Define $R(\theta):=(1-\theta^2)P'(\theta)/\theta.$ Logarithmic differentiation gives
\begin{equation}\label{eq:log-diff-R}
    \theta(1-\theta^2)\frac{R'(\theta)}{R(\theta)}
      =z(1+\theta^2)\coth z-2z\theta-1.
\end{equation}
For $0<K\le2$, we claim
\begin{equation}\label{eq:coth-scalar}
    z(1+\theta^2)\coth z-2z\theta
       \le1+e^{-K^2}\theta^2.
\end{equation}
If $z\le1$, use
$z\coth z\le1+z^2/3$, and set
$u=K^2/(1-\theta^2)=z/\theta$. The claim reduces to
\[
    1-e^{-u(1-\theta^2)}-2u
      +\frac{1+\theta^2}{3}u^2\le0.
\]
For $u\le3$, we have
$1-e^{-u(1-\theta^2)}\le u(1-\theta^2)$, so the left-hand side is at most $u(1-\theta^2) - 2u + (1+\theta^2)u^2 / 3 = u(1+\theta^2)(u/3-1)\leq 0$.  If $u>3$, the conditions
$z\le1$ and $K\le2$ give $u\le2+\sqrt5$, and then since $\theta^2 u^2 = z^2 \leq 1$, the left-hand side above is
at most $(u^2-6u+4)/3<0$.

If $z\ge1$, then we use the inequality $z(1+\theta^2)\coth z - 2z\theta \leq z(1-\theta)^2 + (1+\theta^2) e^{-z}$. To prove the claim it is thus enough to show
\[
    z(1-\theta)^2+(1+\theta^2)e^{-z}\le1.
\]
The left-hand side is convex in $z$, and since $K \leq 2$, we have $1\le z\le\frac{4\theta}{1-\theta^2}.$
Evaluation at these two endpoints proves the inequality for
$\theta\ge\sqrt5-2$, which is precisely the range in which the interval
for $z$ is nonempty.  This proves \eqref{eq:coth-scalar}.

Since $p(K)\le\sqrt{2/\pi}\,K$, we have
$d(K)^2\ge e^{-K^2}$. From \eqref{eq:log-diff-R} and 
\eqref{eq:coth-scalar}, we obtain
\[
    R'(\theta)\le d(K)^2P'(\theta).
\]
Also note that $R(0)=d(K)^2P(0)$.  Integrating first this inequality and then
\[
    \frac{P'(\theta)}{P(\theta)}
       \le d(K)^2\frac{\theta}{1-\theta^2}
\]
gives
\[
    \log\frac{P(\theta)}{P(0)}
       \le-\frac{d(K)^2}{2}\log(1-\theta^2),
\]
as required.
\end{proof}

\subsection{Proof of \cref{lem:entropy-interpolation}}
\label{app:entropy-interpolation}

\begin{proof}
It suffices to prove the lemma for $0\leq \theta < 1$; the endpoint $\theta=1$ follows from continuity.
Apply the relative entropy variational formula with
\[
    f_\theta(x,y)
      :=\frac{t\,(2\theta xy-\theta^2(x^2+y^2))}{1-\theta^2},
\] where (recall) $t=1/(2\sigma(r)^2)$.
We immediately see that the expectation of $f_\theta$ under $\mu^{(2)}_{r,r,\theta}$ is zero. On the other hand, we compute
\[\E[\mu_r\otimes\mu_r]{e^{f_\theta}}=\frac1{Z_r^2}\sum_{x,y\ne0}
        \exp\left(-\frac{t\,(x^2+y^2-2\theta xy)}
                           {1-\theta^2}\right)=:M_r(\theta).
\]
The relative entropy variational formula then gives 
\begin{equation}\label{eq:entropy-M-bound}
    D_r(\theta)\ge-\log M_r(\theta).
\end{equation}

For $r \geq 1.2$, we have $t\le t(1.2)<.63$. Then,
\[
    Z_r = \sum_{m\in \ZZ \setminus\{0\}} e^{-t m^2} \ge2(e^{-.63}+e^{-.63 \cdot 4})>1.1,
\]
which proves that $0<c_r<1$.

For ease of notation, let $s:=\sqrt{1-\theta^2}$ and $c:=\pi^2/t$, and define the two series 
\[
    \mathcal S(u) :=\sum_{m\in\ZZ}e^{-um^2}, \qquad \mathcal A(\theta) :=\sum_{m,n\in\ZZ} e^{-c(m^2+2\theta mn+n^2)},
\]
noting that $Z_r = \mathcal S(t)-1$.
Poisson summation gives 
\begin{align}
    Z_r&=\sqrt{\frac{\pi}{t}}\,\mathcal S\left(\frac{\pi^2}{t}\right)-1,\label{eq:Z-Poisson}\\
    Z_r^2M_r(\theta)
      &=\frac{\pi}{t}s\mathcal A(\theta)-2\sqrt{\frac{\pi}{t}}\,s\, \mathcal S\left(\frac{\pi^2}{t}s^2\right)+1,
      \label{eq:M-Poisson}
\end{align}
We now claim that
\begin{equation}\label{eq:entropy-interpolation-theta-bound}
    Z_r^2M_r(\theta)
       \le sZ_r^2+1.1(1-s)Z_r.
\end{equation}
The proof only uses estimates of positive theta series. If $s\ge.98$, then note that combining \eqref{eq:Z-Poisson} and \eqref{eq:M-Poisson} gives
\begin{equation}\label{eq:Zr-bashing}Z_r^2 M_r(\theta) - sZ_r^2 = \frac{\pi}{t} s(\mathcal A(\theta) - \mathcal A(0)) - 2\sqrt{\frac{\pi}{t}}\,s\left(\mathcal S\left(\frac{\pi^2}{t}s^2\right)-\mathcal S\left(\frac{\pi^2}{t}\right)\right)+(1-s).\end{equation}
Pairing $(m,n)$ with $(m,-n)$ gives
\[
    \mathcal A(\theta)-\mathcal A(0)
       \le\frac{2\pi^4}{t^2}\theta^2
          \left(\sum_{m\in\ZZ}m^2e^{-\pi^2(1-\theta)m^2/t}\right)^2.
\]
The first term in the right-hand side of \eqref{eq:Zr-bashing} is thus at most
$2.4\cdot10^{-7}(1-s)$. The second term there (involving $\mathcal S$) is nonpositive since $s \leq 1$ and $\mathcal S(u)$ is decreasing in $u$. Since $Z_r \geq 1.1$, the entire right-hand side of \eqref{eq:Zr-bashing} is at most $(2.4\cdot 10^{-7} + 1)(1-s) \leq 1.1 (1-s)Z_r$. This proves \eqref{eq:entropy-interpolation-theta-bound} in this case.

Suppose next that $s<.98$, and for ease of notation, set $v=t/s^2$. Isolating the diagonal
$m+n=0$ in $\mathcal A(\theta)$, we change variables to $m+n,m-n$, using that $m^2 + 2\theta mn + n^2 = (1+\theta)(m+n)^2/2 + (1-\theta)(m-n)^2/2.$ Applying \eqref{eq:theta-sum-elementary}, separate the common leading contribution $1+\sqrt{\frac{\pi}{v}}\left(\sqrt{\frac{\pi}{t}}-2\right)$. After this contribution is removed from both sides of \eqref{eq:entropy-interpolation-theta-bound}, it suffices to show that the remaining diagonal, off-diagonal, and Poisson-correction errors have total size at most $(1-s)\left(\sqrt{\frac{\pi}{t}}-2\right)$, which is a positive quantity since $t < \pi/4$. Normalizing the three error terms by this quantity gives, respectively:
\begin{align*}
 E_d(v,s)
   &:=\frac{\sqrt{\pi/v}\,\eta(\pi^2/v)}{1-s},\\
 E_o(v,s)
   &:=\frac{\sqrt{\pi/v}\,\eta\left(\frac{\pi^2(1+\theta)}{2vs^2}\right)\left(1+\eta\left(\frac{\pi^2}{2v(1+\theta)}\right)\right)}
   {(1-s)(1-2s\sqrt{v/\pi})},\\
 E_c(v,s)
   &:=\frac{2s\,\eta(\pi^2/(vs^2))}
   {(1-s)(1-2s\sqrt{v/\pi})},
\end{align*}
where $\eta(u):=2e^{-u}/(1-e^{-3u})$.
The inequality $E_d+E_o+E_c<1$ is then a sufficient condition for
\eqref{eq:entropy-interpolation-theta-bound}.
Interval
evaluation gives
\[
\begin{array}{c|cccc}
v\text{-range}&E_d&E_o&E_c&\text{total}\\ \hline
(0,1]&<.0009&<.1756&<.00030&<.177\\
{[1,30]}&<.8684&<.0007&<.00003&<.870.
\end{array}
\]
The bounds are uniform for
$0<s\le\min\{.98,\sqrt{.63/v}\}$, including the limiting boundaries,
and prove \eqref{eq:entropy-interpolation-theta-bound} for $v\le30$.
These are direct interval evaluations of the explicit theta sums, under the
conventions described in \cref{app:numer-conventions}.

Finally, if $v\ge30$, the terms $x=y$ in the original sum defining
$M_r(\theta)$ contribute at most $Z_r$. Then, 
\eqref{eq:theta-sum-elementary} bounds all other terms by
\[
    \left(\mathcal S(v(1+\theta)/2)-1\right)\mathcal S(v(1-\theta)/2)
       <1.3\cdot10^{-6}\left(1+\sqrt{\frac{\pi}{t}}\right)<.1(1-s)Z_r.
\]
Since $s<\sqrt{.63/30}$, this again gives
\eqref{eq:entropy-interpolation-theta-bound}. Consequently,
\[
    M_r(\theta)
       \le s+\frac{1.1(1-s)}{Z_r}
       =c_r+(1-c_r)s.
\]
Together with \eqref{eq:entropy-M-bound}, this proves the first inequality
in \eqref{eq:entropy-interpolation}.  The second follows from
$-\log(1-x)\ge x$ and
$1-\sqrt{1-\theta^2}\ge\theta^2/2$. 
\end{proof}

\subsection{The scalar estimate for bulk regime}
\label{app:finite-profile-estimates}
We first record the monotonicities needed below.  Since $d(r^2)/dt=-\Var_{\mu_r}(X^2)$, we have
\[
    \frac{d}{d(r^2)}q(r) =2\left(t-\frac{r^2}{\Var_{\mu_r}(X^2)}\right)<0
\]
by \cref{prop:DG moments}.  Thus $q(r)$ is decreasing.  Moreover, $t=t(r)$ is decreasing, so $Z_r=2\sum_{j\ge1}e^{-tj^2}$ is increasing and $1/Z_r$ is decreasing.

We now verify the coefficient comparison \eqref{eq:bulk-scalar-estimate} for the range of parameters described in \cref{lem:bulk-profile}.
The functions $q(r)$ and $1/Z_r$ are decreasing in $r$.  Since
$K\ge\max\{2,10/r\}$,
interval evaluation at the endpoints of the following ranges gives
\[
\begin{array}{c|c|c}
r\text{-range}
  &(1-1.1/Z_r)/2&q(r)/K^2\\
\hline
{[1.2,1.5]}&>.055&<.041\\
{[1.5,2]}&>.237&<.060\\
{[2,3]}&>.338&<.117\\
{[3,4]}&>.407&<.187\\
{[4,5]}&>.435&<.280\\
{[5,\infty)}&>.449&<.273
\end{array}
\]
For the finite ranges, the entries are obtained from
\eqref{eq:interval-moment-equation}, using five terms of the defining
series and a geometric tail bound.  For $r\ge5$, we use the inequalities $q(r)\le q(5), Z_r\ge Z_5,$ and $K\ge2$.
This proves \eqref{eq:bulk-scalar-estimate}.

\subsection{The small-window profile}
\label{app:small-window-profile}
Throughout this subsection, we use the notation \[t = \frac{1}{2\sigma(r)^2}, \quad s = \sqrt{1-\theta^2}, \quad s_0 = p(K)^{1/d(K)^2}, \quad u_0 = \frac{\pi^2}{t}\left(1-\sqrt{1-s_0^2}\right),\quad \mathcal T_2(u):= \sum_{m\in\ZZ} m^2 e^{-um^2}.\]
We first prove a scalar estimate for $K\le2$.

\begin{lemma}[Small-window scalar estimate]
\label{lem:small-window-scalar}
Suppose
\begin{equation}\label{eq:small-window-domain}
    r\ge5,\qquad
    0<K\le2,\qquad
    rK\ge10,\qquad
    \frac{q(r)}{d(K)}
       \ge1+\frac1{r^2K^2}.
\end{equation}
Then
\begin{align}
    Z_r^2(1-q(r)d(K))
       \frac{s_0\log(1/s_0)}{1-s_0}
       &>1.01,\label{eq:small-window-main-scalar}\\
    \frac{4\pi^5}{t^3}\mathcal T_2(u_0)^2
       &<10^{-12}.
       \label{eq:small-window-remainder-bound}
\end{align}
\end{lemma}

\begin{proof}
Define $f(u)=u/(e^u-1).$
The left-hand side of
\eqref{eq:small-window-main-scalar} is
\begin{equation}\label{eq:small-window-E}
    \mathcal E(r,K)
       :=Z_r^2(1-q(r)d(K))f\left(\frac{-\log p(K)}{d(K)^2}\right).
\end{equation}
The remainder of the proof uses only positive one-dimensional series. We have
\[
    \frac1{d(K)}=\sum_{j\ge0}\frac{K^{2j}}{(2j+1)!!},
    \qquad
    p(K)=\sqrt{\frac2\pi}\,Ke^{-K^2/2}
          \sum_{j\ge0}\frac{K^{2j}}{(2j+1)!!},
\]
while Poisson summation gives
\[
    q(r)=1+\frac{1-2\pi^{5/2}/t^{3/2}\cdot \mathcal T_2(\pi^2/t)}{Z_r}<1+\frac1{Z_r},
    \qquad
    Z_r>\sqrt{2\pi}\,r-1.54
\]
for $r\ge5$. For $r\geq 80$ we may improve the second inequality above by replacing $1.54$ with $1.502$.

These identities also show that, for fixed $r$, both factors in
\eqref{eq:small-window-E} are increasing in $K$ on the feasible
subinterval $K\le1$.
On this subinterval, the map
\[
    K\longmapsto
    \frac{q(r)}{d(K)}-1-\frac1{r^2K^2}
\]
is strictly increasing, tends to $-\infty$ as $K\downarrow0$, and is
positive at $K=1$. Let $k_0(r)$ be its unique zero,
and define $K_0(r)=\max\{10/r,k_0(r)\}$. It is therefore enough to
evaluate $\mathcal E(r,K)$ at $K=K_0(r)$. The series above, with
geometric bounds on their positive tails, give the following outward-rounded
lower bounds:
\[
\begin{array}{c|c}
\text{range}&\mathcal E(r,K)\\ \hline
1\le K\le2&>16\\
K=10/r,\ 10\le r\le80&>1.93\\
K=10/r,\ 80\le r\le81.5&>1.02\\
K=k_0(r),\ r\ge81.4&>1.044.
\end{array}
\]
The last two ranges overlap because bisection in the two monotone defining
equations shows that the change of branch occurs in $81.4<r<81.5$.
For the unbounded last range, the substitution $x=1/r$, together with
the Poisson representations above, extends the defining expressions
continuously to $x=0$; interval evaluation on
$x\in [0,1/81.4]$ gives the  bound in the last entry of the table.
This proves \eqref{eq:small-window-main-scalar}, and in particular
$q(r)d(K)<1$.

For the second estimate, the same positive-series bounds give
\[
    s_0\ge .46K\quad(K\le1),\qquad
    s_0\ge .4\quad(1\le K\le2),
\]
and hence
\[
    u_0>1.5r^2\quad(K\ge1),\qquad
    u_0>1.91r^2K^2\ge191\quad(K\le1).
\]
Since $r\geq 5$ we have $q(r) > 1$, so $\pi^5/t^3<8\pi^5r^6$. Furthermore, by comparing consecutive terms, we obtain $\mathcal T_2(u)\le2e^{-u}/(1-4e^{-3u}),$
and checking the endpoints of the same four monotone ranges yields
\eqref{eq:small-window-remainder-bound}.
\end{proof}

\begin{proof}[Proof of \cref{lem:small-window-profile}]
First, we note that \cref{lem:small-window-scalar} immediately gives $0<q(r)d(K)<1$.  Now, using the Poisson formula
\eqref{eq:M-Poisson} only for $s\ge s_0$, we claim that
\begin{equation}\label{eq:small-window-M-chord}
    M_r(\theta)\le s+C_{r,K}(1-s),
\end{equation}
where $M_r(\theta)$ was defined in \cref{app:entropy-interpolation} and 
\[
    C_{r,K}
       :=\frac{1+4\pi^5/t^3 \cdot \mathcal T_2(u_0)^2}{Z_r^2}.
\]
Indeed, subtracting $s$ times the identity at $\theta=0$ in
\eqref{eq:M-Poisson} and pairing the summands indexed by
$(m,n)$ and $(m,-n)$ gives
\[
    \mathcal A(\theta)-\mathcal A(0)
       \le\frac{2\pi^4}{t^2}\theta^2
          \left(\sum_{m\in\ZZ}
            m^2e^{-\frac{\pi^2}{t}(1-\theta)m^2}\right)^2.
\]
Also, we have $\mathcal S(\pi^2 s^2 / t)\ge \mathcal S(\pi^2/t)$, so the contribution by terms involving $\mathcal S$ to the difference between
the proposed upper bound and $M_r(\theta)$ is nonnegative.
Finally,
\[
    \theta^2\le2(1-s),
    \qquad
    \frac{\pi^2}{t}(1-\theta)\ge
       \frac{\pi^2}{t}\left(1-\sqrt{1-s_0^2}\right)=u_0.
\]
These inequalities prove \eqref{eq:small-window-M-chord}.

Since $0<q(r)d(K)<1$, the function
\[
    g(s):=\frac{s^{q(r)d(K)}-s}{1-s}
\]
is increasing on $s\in(0,1)$, and
\[
    g(s_0)
      \ge(1-q(r)d(K))
        \frac{s_0\log(1/s_0)}{1-s_0}.
\]
Thus \cref{lem:small-window-scalar} implies that $C_{r,K}<g(s_0)\le g(s)$ for $s\ge s_0$.
From \eqref{eq:entropy-M-bound} and
\eqref{eq:small-window-M-chord}, we obtain
\[
    D_r(\theta)>-q(r)d(K)\log s.
\]
On the other hand, \eqref{eq:Gaussian-small-window} gives
\[
    \frac{q(r)}{d(K)}G_K(\theta)
       \le-q(r)d(K)\log s.
\]
Combining these yields \eqref{eq:small-window-profile} for $s\ge s_0$.

For $s\le s_0$, the entropy
$H(\mu^{(2)}_{r,r,\theta})$ is even and concave (\cref{prop:F2 entropy derivatives}), so $D_r(\theta)$
is increasing on $0\le\theta\le1$.  Also
\[
    G_K(\theta)\le G_K(1)=-\log p(K),
\]
and
\[
    \frac{q(r)}{d(K)}(-\log p(K))
       =-q(r)d(K)\log s_0.
\]
The strict inequality already proved at $s=s_0$ therefore covers the
remaining interval.
\end{proof}

\subsection{The near-binary profile}
\label{app:near-binary-profile}
As before, we use the notation $t = 1/(2\sigma(r)^2)$ throughout this subsection.
We begin with the following estimate.
\begin{lemma}[Near-binary entropy]\label{lem:near-binary-entropy}
If $1<r\le1.2$, then
\begin{equation}\label{eq:near-binary-entropy-bound}
    D_r(\theta)\ge\frac{\theta^2}{8},
    \qquad 0\le\theta\le1.
\end{equation}
\end{lemma}

\begin{proof}
For any $0\le u\le1/4$, the relative entropy variational formula gives
\[
    D_r(\theta)
      \ge u\theta-\log\E{e^{uXY/r^2}},
\]
where $X,Y$ are independent with law $\mu_r$.  By symmetry, the
exponential inside the expectation may be replaced by its hyperbolic
cosine.  Using the easy bounds
\[
    \cosh z\le1+\frac{z^2}{2}e^{|z|}
    \qquad\text{and}\qquad
    2|XY|\le X^2+Y^2,
\]
we obtain
\[
    \E{e^{uXY/r^2}}
      \le1+\frac{u^2}{2}
       \left(\frac{\E{X^2e^{uX^2/(2r^2)}}}{r^2}\right)^2.
\]
Recalling the notation $\mathcal T_2(v)=\sum_{j\in\ZZ}j^2e^{-vj^2},$
the ratio in parentheses is $\mathcal T_2(t-u/(2r^2))/\mathcal T_2(t).$
This ratio decreases with $t$ and increases with $u/(2r^2)$.  Since $t\ge t(1.2)>.627$ and $0\le u/(2r^2)\le\frac18$,
it suffices to evaluate
\[
    \frac{\mathcal T_2(.627-1/8)}{\mathcal T_2(.627)}<1.45.
\]
Five terms followed by an integral bound for the positive tail certify this inequality.  It follows that
\[
    \log\E{e^{uXY/r^2}}\le2u^2.
\]
Taking $u=\theta/4$ proves
\eqref{eq:near-binary-entropy-bound}.
\end{proof}

Define
\[
    I_{\rm bin}(\theta)
      :=\frac{1+\theta}{2}\log(1+\theta)
        +\frac{1-\theta}{2}\log(1-\theta),
    \qquad
    \alpha_{\rm bin}(K):=\frac{\log2}{-\log p(K)}.
\]
We use the binary profile inequality
\begin{equation}\label{eq:ALS-global}
    \alpha_{\rm bin}(K)G_K(\theta)
       \le I_{\rm bin}(\theta),
    \qquad K>4,\quad 0<\theta<1,
\end{equation}
proved in \cite[Lemma~4.6]{ALS22}.  The inequality is strict in the
displayed open interval.

\begin{proof}[Proof of \cref{lem:near-binary-profile}]
Let $\alpha=q(r)/d(K)$, and write
$F_2(\theta)=F_2(rK,\alpha,r,r,\theta)$. Since $rK\ge10$ and
$r\le1.2$, we have $K\ge25/3$.

Suppose first that $q(r)\le8$.  Using
\eqref{eq:Gaussian-Hermite-coarse}, $1-p(K) = 2\overline \Phi(K)$, and
$\overline\Phi(K)\le\phi(K)/K$,
\begin{equation}\label{eq:big-q-bound}
    \alpha G_K(\theta)
      \le\frac{q(r)}{K^2}\theta^2
      \le\frac8{(25/3)^2}\theta^2
      <.116\theta^2.
\end{equation}
This is smaller than \eqref{eq:near-binary-entropy-bound}.

Suppose now that $q(r)>8$.  Then
$t(r)=q(r)/(2r^2)>25/9.$
Factoring out the $j=1$ term in the series defining $Z_r$ and its
second moment gives
\[
    r^2-1
      \le\sum_{j\ge2}(j^2-1)e^{-(j^2-1)t(r)}
      <7.22\cdot10^{-4}.
\]
Here the $j=2$ term is retained and the remaining positive tail is
bounded geometrically.  In particular,
\begin{equation}\label{eq:near-binary-delta}
    \delta:=r^2-1<7.3\cdot10^{-4}.
\end{equation}
Writing $X=S|X|$, where the sign $S$ is uniform on $\{\pm1\}$ and independent of $|X|$, we have
$\E{|X|-1}\le\delta/3$.  The geometric distribution maximizes entropy
among nonnegative integer-valued random variables with a fixed mean, and
therefore
\[
    H(\mu_r)
       \le\log2+g(\delta/3)
       <1.0033\log2,
    \qquad
    g(x):=(1+x)\log(1+x)-x\log x.
\]
The assumption $H(\mu_r) + \frac{q(r)}{d(K)} \log p(K) \geq 0$ implies
\begin{equation}\label{eq:near-binary-alpha-upper}
    \alpha\le1.0033\,\alpha_{\rm bin}(K).
\end{equation}

For $(X,Y)\sim\mu^{(2)}_{r,r,\theta}$, let $S=\sign(X)$ and
$T=\sign(Y)$.  Applying data processing to
\eqref{eq:Dr-relative-entropy}, and using
$\E{|X||Y|-1}\le r^2-1$, gives
\begin{equation}\label{eq:sign-projection}
    D_r(\theta)
      \ge I_{\rm bin}\bigl([\theta-\delta(1-\theta)]_+\bigr).
\end{equation}
Indeed,
\[
    \E{ST}
      =\E{XY}-\E{ST(|X||Y|-1)}
      \ge\theta r^2-\delta
      =\theta-\delta(1-\theta).
\]
For $\theta\ge1/4$,
\begin{equation}\label{eq:binary-information-loss}
    I_{\rm bin}(\theta-\delta(1-\theta))
       \ge.9955\,I_{\rm bin}(\theta).
\end{equation}
To see this, note that
\[
    I_{\rm bin}(\theta)
     -I_{\rm bin}(\theta-\delta(1-\theta))
     \le\delta(1-\theta)\arctanh\theta,
\]
while $I_{\rm bin}(\theta)\ge\theta^2/2$; the resulting ratio is at
most $.0045$.

For $\theta\le1/4$,
\eqref{eq:Gaussian-Hermite-refined},
\eqref{eq:near-binary-alpha-upper}, and
\eqref{eq:Gaussian-tail-two-sided} give
\begin{equation}\label{eq:near-binary-center}
    \alpha G_K(\theta)<.045\theta^2
       <D_r(\theta).
\end{equation}
Here we used
\[
    \alpha_{\rm bin}(K)\frac{1-p(K)}{p(K)}
       \le\frac{\log2}{p(K)}<.7,
    \qquad
    \alpha_{\rm bin}(K)d(K)^2<10^{-12}.
\]

Now, if $\alpha\le.995\alpha_{\rm bin}(K),$
then \eqref{eq:ALS-global}--\eqref{eq:sign-projection} prove the main claim
for $\theta\ge1/4$, and \eqref{eq:near-binary-center} covers the
remaining interval.  

It thus remains to consider $\alpha>.995\alpha_{\rm bin}(K).$
Using
\[
    -\log p(K)\le\frac{1-p(K)}{p(K)} = \frac{2\overline{\Phi}(K)}{p(K)}
       \le\frac{2\phi(K)}{Kp(K)},
\]
we obtain
\[
    \frac{d(K)}{-\log p(K)}\ge K^2.
\]
Recall that stationarity gives $\alpha = q(r)/d(K)$, which yields 
\[
    q(r)>.995(\log2)K^2,
    \qquad
    t(r)>.239K^2.
\]
The defining series therefore yields
\begin{equation}\label{eq:binary-exponential-radius}
    r^2-1\le4e^{-3t(r)}
       \le4e^{-.717K^2}.
\end{equation}

We first treat
\[
    \frac14\le\theta\le
    \theta_U:=1-\frac1{3K^2}.
\]
Let $P_K(\theta)=p^{(2)}(K,K,\theta)$. Integrating
\eqref{eq:energy_derivative}, using the two-sided Gaussian tail bound
\eqref{eq:Gaussian-tail-two-sided}, and evaluating the resulting
one-variable expressions on the intervals
\[
    [1/4,1/2],\quad[1/2,.8],\quad[.8,.95],\quad[.95,.98]
\]
gives
\[
    I_{\rm bin}(\theta)
     -1.0033\,\alpha_{\rm bin}(K)G_K(\theta)>.0315,
    \qquad \frac14\le\theta\le.98.
\]
For $.98\le\theta\le\theta_U$, let
$J_\pm=\PP{W>K,\ \pm W_\theta>K}$. Conditioning on
$(W+W_\theta,W-W_\theta)$ gives
\[
    J_+<.693\overline\Phi(K),\qquad J_-<10^{-3}\overline\Phi(K).
\]
Since
\[
    P_K(\theta)=p(K)^2+2(J_++J_-)-4\overline\Phi(K)^2,
\]
it follows that
$\alpha_{\rm bin}(K)G_K(\theta)<.482$, and hence
\begin{equation}\label{eq:binary-middle-gap}
    I_{\rm bin}(\theta)
     -1.0033\,\alpha_{\rm bin}(K)G_K(\theta)>.0315,
    \qquad \frac14\le\theta\le\theta_U.
\end{equation}
All constants in the above computations are direct
outward-rounded evaluations, uniform for $K\ge25/3$.

The loss from replacing $\theta$ by
$\theta-\delta(1-\theta)$ in \eqref{eq:sign-projection} is at most
\[
    \delta(1-\theta)\arctanh\theta<.0045\log2 < .0032.
\]
Thus \eqref{eq:near-binary-alpha-upper} and
\eqref{eq:binary-middle-gap} prove the profile inequality on
$[1/4,\theta_U]$.

Finally, let $\theta=1-\varepsilon\in(\theta_U,1)$.  By \eqref{eq:binary-exponential-radius},
\[
    r^2-1<10^{-21},
    \qquad
    \E{(Y-X)^2}=2r^2\varepsilon.
\]
Since $H(X)\le H(\mu_r)$, and since an integer-valued random variable
of second moment $2r^2\epsilon$ satisfies
\[
    H(Y-X)\le H(|Y-X|)+\PP{Y\ne X}\log2
       \le g(2r^2\epsilon)+2r^2\epsilon\log2,
\]
we obtain $H(\mu^{(2)}_{r,r,\theta})-H(\mu_r)\le2.5\varepsilon\log(e/\varepsilon).$
Concavity of $H(\mu_{r, r, \theta}^{(2)})$ (\cref{prop:F2 entropy derivatives}) therefore yields
\begin{equation}\label{eq:near-binary-entropy-derivative}
    \partial_\theta H(\mu^{(2)}_{r,r,\theta})
       \ge-2.5\log\frac e\varepsilon.
\end{equation}
On the other hand, the derivative computation
\eqref{eq:energy_derivative} yields the following direct estimate.
Since $\alpha>.995\alpha_{\rm bin}(K)$,
$P_K(\theta)\le p(K)$, and
$-p(K)\log p(K)\le1-p(K)=2\overline\Phi(K)$,
\[
    \frac{\alpha}{P_K(\theta)}
       \ge\frac{.995\log2}{2\overline\Phi(K)}.
\]
The upper bound in \eqref{eq:Gaussian-tail-two-sided} therefore gives
\[
 \alpha\partial_\theta\log P_K(\theta)
 \ge
 \frac{.995\log2}{\sqrt{2\pi}}\,
 \frac{K}{\sqrt{\varepsilon(2-\varepsilon)}}\,
 \exp\left(-\frac{K^2\varepsilon}{2(2-\varepsilon)}\right)
 \left(1-\exp\left(-\frac{2K^2\theta}{1-\theta^2}\right)\right).
\]
Here $K^2\varepsilon\le1/3$, $K\ge25/3$, and $\theta \geq 0.98$, so $1-\exp(-2K^2\theta/(1-\theta^2)) \geq 1-e^{-100}$.  Hence
\begin{equation}\label{eq:near-binary-energy-derivative}
    \alpha\partial_\theta\log P_K(\theta)
       >.17K\varepsilon^{-1/2}
       >.16K\varepsilon^{-1/2}.
\end{equation}
At $\varepsilon=1/(3K^2)$ and $K=25/3$, the absolute value of the
right-hand side of \eqref{eq:near-binary-entropy-derivative} is less
than $15.85$, while the right-hand side of
\eqref{eq:near-binary-energy-derivative} is greater than $19.24$.
The comparison improves as $K$ grows or $\varepsilon$ decreases.
Thus $F_2'(\theta)>0$ on $(\theta_U,1)$.  Finally,
\[
    F_2(1)=H(\mu_r)+\alpha\log p(K)
       \le2(H(\mu_r)+\alpha\log p(K))=F_2(0),
\]
so the profile inequality $F_2(\theta) < F_2(0)$ follows on this last interval $\theta\in(\theta_U, 1)$.

Finally, we prove that $q(r)d(K) < 1$. If $q(r)\le8$, then by \eqref{eq:energy_derivative} and the computation directly following it, Taylor expansion at $\theta=0$ gives 
\[
    \alpha G_K(\theta)
       =\frac{q(r)d(K)}2\theta^2+O(\theta^4).
\]
Here, we have recalled $\alpha=q(r)/d(K)$, and then used symmetry of $G_K$ to eliminate the odd terms in the Taylor expansion. Hence, combining this expansion with \eqref{eq:big-q-bound} and taking $\theta \to 0$ gives $q(r)d(K)\le2\cdot0.116<1$. If $q(r)>8$, \eqref{eq:near-binary-center} similarly gives $q(r)d(K)\le2\cdot0.045<1$. This concludes the proof.

\end{proof}

\end{document}